\documentclass[12pt,intlimits,oneside,english]{smfbook}
\usepackage[francais,english]{babel}
\usepackage{smfthm}
\usepackage[applemac]{inputenc}
\usepackage{amsmath}
\usepackage{latexsym}
\usepackage{amssymb}
\usepackage{xcolor}
\usepackage[colorlinks=true]{hyperref}
\usepackage{mathrsfs}
\theoremstyle{plain}

\usepackage[colorlinks=true]{hyperref}
\numberwithin{equation}{section}
\numberwithin{figure}{section}
\def\expo_#1{{\rm e}^{#1}}
\def\rh{\rho}
\def\si{\sigma}
\def\f{\varphi}
\def\t{\theta}
\def\R{{\mathbb R}}
\def\C{{\mathbb C}}
\def\D{{\mathbb D}}
\def\T{{\mathbb T}}
\def\Z{{\mathbb Z}}
\def\S{{\mathbb S}}

\def\N{{\mathbb N}}
\def\s{\vskip 0.25cm\noindent}
\def\pa{\partial}
\def\build#1_#2^#3{\mathrel{\mathop{\kern 0pt#1}\limits_{#2}^{#3}}}
\def\td_#1,#2{\mathrel{\mathop{\build\longrightarrow_{#1\rightarrow #2}^{}}}}
\def\e{\varepsilon}
\newcommand{\ben}{\begin{equation}}
\newcommand{\een}{\end{equation}}

\newcommand{\beno}{\begin{eqnarray*}}
\newcommand{\eeno}{\end{eqnarray*}}
\makeindex

\newtheorem{theorem}{Theorem}
\newtheorem{corollary}{Corollary}
\newtheorem{proposition}{Proposition}
\newtheorem{lemma}{Lemma}
\newtheorem{remark}{Remark}
\newtheorem{definition}{Definition}
\date{August 27, 2015}
\begin{document}
\title{The cubic Szeg\H{o} equation and Hankel operators}
\author{Patrick G\'erard}
\address{Universit\'e Paris-Sud XI, Laboratoire de Math\'ematiques
d'Orsay, CNRS, UMR 8628, et Institut Universitaire de France} \email{{\tt Patrick.Gerard@math.u-psud.fr}}

\author[S. Grellier]{Sandrine Grellier}
\address{F\'ed\'eration Denis Poisson, MAPMO-UMR 6628,
D\'epartement de Math\'ematiques, Universit\'e d'Orleans, 45067
Orl\'eans Cedex 2, France} \email{{\tt
Sandrine.Grellier@univ-orleans.fr}}

\subjclass{35B15, 47B35, 37K15}

\begin{abstract}

This monograph  is an expanded version of the preprint  \cite{GG6}
It is devoted to the dynamics on Sobolev spaces of the cubic Szeg\H{o} equation on the circle $\S ^1$,
$$ i\pa _t u=\Pi (\vert u\vert ^2u)\ .$$
Here $\Pi $ denotes the orthogonal projector from $L^2(\S ^1)$ onto the subspace $L^2_+(\S ^1)$ of functions with nonnegative Fourier modes.
We  construct a nonlinear Fourier transformation on $H^{1/2}(\S ^1)\cap L^2_+(\S ^1)$ allowing to describe explicitly the solutions of this equation
with data in $H^{1/2}(\S ^1)\cap L^2_+(\S ^1)$. This explicit description implies  almost-periodicity of every solution in $H^{\frac 12}_+$. Furthermore, it allows to display the following turbulence phenomenon. For a dense $G_\delta $ subset of initial data in $C^\infty (\S ^1)\cap L^2_+(\S ^1)$, the solutions tend to infinity in $H^s$ for every $s>\frac 12$ with  super--polynomial growth on some sequence of times, while they go back to their  initial data on another sequence of times tending to infinity. This  transformation is defined by solving a general inverse spectral problem involving singular values  of a Hilbert--Schmidt Hankel operator and of its shifted Hankel operator.  
\end{abstract}

\begin{altabstract}
Cette monographie est une version étendue de la prépublication  \cite{GG6}.\\
Elle est consacrée à l'étude de la dynamique, dans les espaces de Sobolev, de l'équation de Szeg\H{o} cubique sur le cercle $\S^1$,
 $$ i\pa _t u=\Pi (\vert u\vert ^2u)\ ,$$
où $\Pi$ désigne le projecteur orthogonal de $L^2(\S ^1)$ sur le sous-espace  $L^2_+(\S ^1)$ des fonctions à modes de Fourier positifs ou nuls.
On construit une transformée de Fourier non linéaire sur  $H^{1/2}(\S ^1)\cap L^2_+(\S ^1)$ permettant de résoudre explicitement cette équation avec données initiales dans  $H^{1/2}(\S ^1)\cap L^2_+(\S ^1)$. Ces formules explicites entraînent la presque périodicité des solutions dans $H^{\frac 12}_+$. Par ailleurs, elles permettent de mettre en évidence le phénomène de turbulence suivant. Pour un $G_\delta$ dense de données initiales de $C^\infty (\S ^1)\cap L^2_+(\S ^1)$, les solutions tendent vers l'infini à vitesse sur--polynomiale en norme $H^s(\S ^1)$  pour tout $s>\frac 12$ sur une suite de temps, alors qu'elles retournent vers  leur donnée initiale sur une autre suite de temps tendant vers l'infini.
Cette transformation  est définie {\it via} la résolution d'un problème spectral inverse lié aux valeurs singulières d'un opérateur de Hankel Hilbert-Schmidt et de son opérateur décalé.

\end{altabstract}

\keywords{Cubic Szeg\H{o} equation, integrable system, NLS, Hankel operator, spectral analysis}
\thanks { The authors are grateful to Z.Hani, T. Kappeler, S. Kuksin, V. Peller, A. Pushnitski and N. Tzvetkov for valuable discussions and comments about earlier versions of  this work, to L. Baratchart for pointing them references \cite{AAK} and \cite{Bar}, and to A. Nazarov for references concerning Bateman--type formulae. The first author is supported by ANR ANAE 13-BS01-0010-03.}

\maketitle
\tableofcontents
\mainmatter
\chapter{Introduction} \label{intro}

The  large time behavior of solutions to Hamiltonian partial differential equations is an important problem in mathematical physics. 
In the case of finite dimensional Hamiltonian systems, many features of the large time behavior of trajectories are described using the topology of the phase space. For a given infinite dimensional system, several natural phase spaces, with different topologies, can be chosen, and the large time properties may strongly depend on the choice of such topologies. For instance, it is known that the cubic defocusing nonlinear Schr\"odinger equation 
$$i\pa _tu+\Delta u=\vert u\vert ^2u$$
posed on a Riemannian manifold $M$ of dimension $d=1,2,3$ with sufficiently uniform properties at infinity, defines a global flow on the Sobolev spaces  $H^s(M)$ for every $s\ge 1$. In this case, a typical large time behavior of interest is the boundedness of trajectories. On the energy space $H^1(M)$, the  conservation of energy trivially implies that all the trajectories are bounded. On the other hand, the existence of unbounded  trajectories in $H^s(M)$ for $s>1$ as well as bounds of the growth in time of the $H^s(M)$ norms is a long standing problem \cite{Bo1}. As a way to detect and to measure the transfer of energy to small scales, this problem is naturally connected to wave turbulence. In \cite{Bo2} and \cite{St}, it has been established that big $H^s$ norms grow at most polynomially
in time. Existence of unbounded trajectories only recently \cite{HPTV} received a positive answer in some very special cases, while several instability results were already obtained \cite{CKSTT}, \cite{GuK}, \cite{H}, \cite{HP}, \cite{GHP}. Natural model problems for studying these phenomena seem to be those for which the calculation of solutions is the most explicit, namely integrable systems. Continuing with the example of nonlinear Schr\"odinger equations, a typical example is the one dimensional cubic nonlinear Schr\"odinger defocusing equation (\cite{ZS}). However, in this case, the set of conservation laws is known to control the whole regularity of the solution, so that all the trajectories of $H^s(M)$ are bounded in $H^s(M)$ for every nonnegative integer $s$. 
In fact, when the equation is posed on the circle, the recent results of \cite{GrK} show that, for every such $s$, the trajectories in $H^s(M)$ are almost periodic in $H^s(M)$. 
\s
The goal of this monograph is to study an integrable infinite dimensional system, connected to a nonlinear wave equation, with a dramatically different large time behavior of its trajectories according to the regularity of the phase space. 
\s
Following \cite{MMT} and \cite{ZGPD}, it is natural to change the dispersion relation by considering the family of equations, $\alpha< 2$,
$$i\partial_t u-|D|^\alpha u=|u|^2u$$
posed on the circle $\S ^1$, where the operator $\vert D\vert^\alpha $ is defined by
$$\widehat {\vert D\vert^\alpha u} (k)=\vert k\vert^\alpha \hat u(k)$$
for every distribution $u\in \mathcal D'(\S ^1)$. 
Numerical simulations in \cite{MMT} and \cite{ZGPD} suggest weak turbulence. For $1<\alpha<2$, it is possible to recover at most polynomial growth of the Sobolev norms (\cite{T}).
In the case $\alpha=1$, the so-called half-wave equation
\begin{equation}\label{halfwave}
i\pa _tu-\vert D\vert u=\vert u\vert ^2u
\end{equation}
defines a global flow on $H^s(\S ^1)$\index{$H^s$|see{Sobolev spaces}} for every $s\ge \frac 12$ ( \cite{GG5} --- see also \cite{Po2}). In that case, the only available bound of the $H^s$-norms is
${\rm e}^{ct^2}$ (\cite{T}), due to the lack of dispersion. Therefore, the special case $\alpha=1$ seems to be a more favorable framework for displaying wave turbulence effects.\\
Notice that this equation can be reformulated as a system, using the Szeg\H{o} projector $\Pi $\index{Szeg\H{o} projector} \index{$\Pi$|see{Szeg\H{o} projector}} defined on $\mathcal D'(\S ^1)$ as
$$\widehat {\Pi u} (k)={\bf 1}_{k\ge 0} \hat u(k)\ .$$
Indeed, setting $u_+:=\Pi u\ ,u_-:=(I-\Pi )u\ ,$ equation (\ref{halfwave})  is equivalent to the system
$$
\left \{ \begin{array}{ll} i(\pa _t+\pa _x)u_+&=\Pi (\vert u\vert ^2u) \\  i(\pa _t-\pa _x)u_-&=(I-\Pi )(\vert u\vert ^2u) \end{array} \right .
$$
Furthermore, if the initial datum $u_0$ satisfies $u_0=\Pi u_0$ and belongs to $H^s(\S ^1)$, $s>1$, with a small norm $\e $, 
then the corresponding solution $u$ is approximated  by the solution $v$ of 
\begin{equation}\label{szegotrans}
 i(\pa _t+\pa _x)v=\Pi (\vert v\vert ^2v)\ ,
 \end{equation}
for $\vert t\vert \le \frac 1{\e ^2} \log \frac 1\e \ $ \cite{GG5} \cite{Po2}. In other words, the first nonlinear effects in the above system arise through a decoupling of the two equations. Notice that an elementary change of variable in equation (\ref{szegotrans}) reduces it to
\begin{equation}\label{szegoIntro}
i\pa _tu=\Pi (\vert u\vert ^2u)\ .
\end{equation} 
Equation (\ref{szegoIntro}) therefore appears as a model evolution for describing the dynamics of the half-wave equation (\ref{halfwave}). 
We introduced it in \cite{GG1} under the name {\sl cubic Szeg\H{o} equation}\index{cubic Szeg\H{o} equation}.  In this reference, global wellposedness of the initial value problem was established on 
$$H^s_+(\S ^1):=H^s(\S ^1)\cap L^2_+(\S ^1)\ ,\ s\ge \frac 12 \ ,$$
\index{Sobolev spaces}
where $L^2_+(\S ^1)$\index{$L^2_+(\S ^1)$} denotes the range of $\Pi $ on $L^2(\S ^1)$, namely the $L^2$ functions on the circle with only nonnegative Fourier modes.
Notice that $L^2_+(\S ^1)$ can be identified to the Hardy space of holomorphic functions on the unit disc $\D $ with $L^2$ traces on the boundary.
Similarly, the phase space $H^{\frac 12}_+(\S ^1)$ can be identified to the Dirichlet space of holomorphic functions on $\D $ with derivative in $L^2(\D )$ for the Lebesgue measure. In the sequel, we will freely adopt either the representation on $\S ^1$ or the one on $\D $. 
\s
The new feature   discovered in \cite{GG1} is that equation (\ref{szegoIntro}) enjoys a Lax pair structure\index{Lax pair}, in connection with Hankel operators \index{Hankel operator} --- see the end of the introduction for a definition of these operators. As a first consequence, we proved that the $H^s$ norms, for $s>1$, evolve at most with exponential growth. 

Using this Lax pair structure, we also proved that  (\ref{szegoIntro}) admits finite dimensional invariant symplectic manifolds of arbitrary dimension, made of rational functions of the variable $z$ in $\D $. Furthermore, the dynamics on these manifolds is integrable in the sense of Liouville. In \cite{GG2}, we introduced action angle variables on open dense subsets of these invariant manifolds and on a dense $G_\delta $ subset of $H^{1/2}_+(\S ^1)$. Finally, in \cite{GG3}, we established an explicit formula for all solutions of (\ref{szegoIntro}) which allowed us to prove quasiperiodicity \index{quasiperiodicity}of all trajectories made of rational functions. In particular, this implies that every such trajectory is bounded in all the $H^s$ spaces. However, the explicit formula in \cite{GG3} was not adapted to study almost periodicity of non rational solutions nor boundedness of non-rational $H^s$-trajectories for $s>\frac 12$.
\s
The main purpose of this monograph is to construct a global nonlinear Fourier transform\index{Non linear Fourier transform} on the space $H^{1/2}_+(\S ^1)$ allowing to describe more precisely the cubic Szeg\H{o} dynamics. As a consequence, we obtain the following results. First we introduce a convenient notation. We denote by $Z$ the nonlinear evolution group defined by (\ref{szegoIntro}) on $H^{1/2}_+(\S ^1)$. In other words, for every $u_0\in H^{1/2}_+(\S ^1)$, $t\mapsto Z(t)u_0$ \index{$Z(t)$}is the solution $u\in C(\R ,H^{1/2}_+)$ of equation (\ref{szegoIntro}) such that $u(0)=u_0$. We also set $C^\infty_+(\S ^1):=C^\infty (\S ^1)\cap L^2_+(\S^1).$\index{$C^\infty_+(\S ^1)$}
\begin{theorem}\label{dynamicIntro}
\text{ }
\begin{enumerate}
\item For every $u_0\in H^{1/2}_+(\S ^1)$, the mapping $$t\in \R \mapsto Z(t)u_0\in H^{1/2}_+(\S ^1)$$ is almost periodic\index{almost periodicity}.
\item There exists initial data $u_0\in C^\infty _+(\S ^1)$ and sequences $({\overline t}_n)$, $({\underline t}^n)$ tending to infinity such that $\forall s>\frac 12\ ,\forall M\in \Z_+ $, 
$$\frac{\Vert Z({\overline t}_n)u_0\Vert _{H^s}}{|{\overline t}_n|^M}\td_n,\infty \infty \ ,$$
and $Z({\underline t}^n)u_0\td_n,\infty u_0$ in $C^\infty _+$. Furthermore, the set of such initial data is a dense $G_\delta $ subset of $C^\infty_+(\S ^1)$. 
\end{enumerate}
\end{theorem}
\begin{remark}
\begin{itemize}
\item One could wonder about the influence of the Szeg\H{o} projector in the equation (\ref{szegoIntro}) and consider the equation without the projector $\Pi$
$$i\pa_tu = |u|^2u .$$ In that case, $u(t)={\rm e}^{-it|u_0|^2}u_0$ and 
$$\Vert u(t)\Vert_{H^s}\simeq C|t|^s.$$
Hence the action of the Szeg\H{o} projector both accelerates the energy  transfer to high frequencies, and facilitates the transition to low frequencies.
\item The above theorem is an expression of   some intermittency phenomenon. Notice that, on the real line, explicit examples with infinite limit at infinity are given in \cite{Po} with $\Vert Z(t)u_0\Vert _{H^s}\simeq \vert t\vert^{2s-1}$.
\item The exponential rate is expected to be optimal but the question remains open for the cubic Szeg\H{o} equation. However, for the perturbation
$$i\pa_t u=\Pi(|u|^2u)+\hat u(0),$$ 
explicit examples such that the Sobolev norms tends exponentially to infinity are given in \cite{X}.
\item Existence of unbounded trajectories for the half-wave equation (\ref{halfwave}) is an open problem.  However, unbounded trajectories have been recently exhibited in \cite{X2} for $$i\pa_tu+\frac{\pa^2}{\pa x^2}u-|D_y|u=|u|^2u, \; (x,y)\in \R^2$$
by adapting the method developed in \cite{HPTV}.
\end{itemize}
\end{remark}

We now give an overview of the nonlinear Fourier transform\index{Non linear Fourier transform}. First we introduce some additional notation. Given a positive integer $n$,
we set \index{$\Omega_n$}$$\Omega _n:=\{s_1>s_2>\dots>s_n>0\}\subset \R^n\ . $$ Given a nonnegative integer $d\ge 0$, we recall that a Blaschke product \index{Blaschke product}of degree $d$ is a rational function on $\C $ of the form
$$\Psi (z)=\expo_{-i\psi}\prod _{j=1}^d \frac{z-p_j}{1-\overline p_jz}\ ,\   \psi \in \T \ ,\ p_j\in \D\  .$$ 
Alternatively, $\Psi$ can be written as 
$$\Psi (z)=\expo_{-i\psi}\frac{P(z)}{z^d\overline P\left (\frac 1z\right )}\ ,$$
where $\psi \in \T $ is called the angle of $\Psi $ and $P$ is a monic polynomial of degree $d$ with all its roots in $\D $.  Such polynomials are called Schur polynomials. We denote by $\mathcal B_d$\index{$\mathcal B_d$|see{Blaschke product}} the set of Blaschke products of degree $d$. It is a classical result --- see {\it e.g.} \cite{He} or section \ref{Blaschke} --- that $\mathcal B_d$ is diffeomorphic to $\T \times \R ^{2d}$. 
\s
Given a $n$-tuple $(d_1,\dots ,d_n)$ of nonnegative integers, we set \index{$\mathcal S_{d_1,\dots ,d_n}$}
$$\mathcal S_{d_1,\dots ,d_n}:=\Omega _n\times \prod _{r=1}^n \mathcal B_{d_r}\ ,$$
endowed with the natural product topology.
Given a sequence $(d_r)_{r\ge 1}$ of nonnegative integers, we \index{$\mathcal S_{(d_r)}^{(2)}$} 
 denote by $ \mathcal S_{(d_r)}^{(2)}$ the set of pairs $((s_r)_{r\ge 1}, (\Psi _r)_{r\ge 1})\in \R ^\infty \times \prod _{r=1}^\infty \mathcal B_{d_r}$
such that
$$s_1>\dots >s_n>\dots >0\ ,\ \sum _{r=1}^\infty (d_r+1)s_r^2<\infty \ .$$
We also endow $ \mathcal S_{(d_r)}^{(2)}$ with the natural topology. \s
Finally, we denote by $\mathcal S_n$ \index{$\mathcal S_n$}the union of all $\mathcal S_{d_1,\dots ,d_n}$
over all the $n$-tuples $(d_1,\dots ,d_n)$, and by $\mathcal S_\infty ^{(2)}$ \index{$\mathcal S_\infty ^{(2)}$} the union of all $ \mathcal S_{(d_r)}^{(2)}$ over all the sequences $(d_r)_{r\ge 1}$. 
Given $ ({\bf s},{\bf \Psi})\in \mathcal S_n$ and $z\in \C $,  we define the matrix \index{$\mathscr C(z)$}$\mathscr C(z):=\mathscr C ({\bf s},{\bf \Psi})(z)$ as follows. If $n=2q$, the coefficients of $\mathscr C ({\bf s},{\bf \Psi})(z)$ are given by
\begin{equation}\label{matrix C}
c_{jk}(z):=\frac{s _{2j-1}-s _{2k}z\Psi _{2k}(z)\Psi _{2j-1}(z)}{s _{2j-1}^2-s _{2k}^2}\ ,\ j,k=1,\dots ,q\ .
\end{equation}
If $n=2q-1$, we use the same formula as above, with $s_{2q}=0$. 
\begin{theorem}\label{FT}\index{$u ({\bf s},{\bf \Psi})$}
For every $n\ge 1$, for every $ ({\bf s},{\bf \Psi})\in \mathcal S_n$, for every $z\in \overline {\D }$, the matrix $\mathscr C ({\bf s},{\bf \Psi})(z)$ is invertible. We set 
\ben \label{luminy}
u ({\bf s},{\bf \Psi})(z)=\langle {\mathscr C(z)}^{-1}(\Psi _{2j-1}(z))_{1\le j\le q}, {\bf 1} \rangle ,
\een
where 
$${\bf 1}:=\left (\begin{array}{l} 1 \\ .\\ .\\ .\\ 1\end{array}\right )\ ,\text{ and }\langle X,Y\rangle :=\sum _{k=1}^q X_kY_k .$$
 For every $ ({\bf s},{\bf \Psi})\in \mathcal S_\infty ^{(2)}$, the sequence $ (u_q)_{q\ge 1}$ with $$u_q:=u((s_1,\dots ,s_{2q}), (\Psi _1,\dots ,\Psi _{2q} )),$$ is strongly convergent in
 $H^{1/2}_+(\S ^1)$. We denote its limit by $u ({\bf s},{\bf \Psi})$. 
 
 The mapping 
 $$ ({\bf s},{\bf \Psi})\in \bigcup_{n=1}^\infty \mathcal S_n\cup \mathcal S^{(2)}_\infty \longmapsto u ({\bf s},{\bf \Psi})\in H^{1/2}_+\setminus \{ 0\} $$
 is bijective. Furthermore, its restriction to every $\mathcal S_{(d_1,\dots ,d_n)} $ and to $\mathcal S_{(d_r)}^{(2)}$ is a homeomorphism onto its range.

Finally, the solution at time $t$  of equation (\ref{szegoIntro}) with initial data $u_0=u ({\bf s},{\bf \Psi})$ is $u ({\bf s},{\bf \Psi}(t))$, where
$$\Psi _r(t)= \expo_{i(-1)^rs_r^2t}\Psi _r\ .$$
\end{theorem}
Using Theorem \ref{FT}, it is easy to prove the first assertion of Theorem \ref{dynamicIntro}. As for the second assertion of Theorem \ref{dynamicIntro}, it heavily relies on a new phenomenon, which is the loss of continuity of the map $ ({\bf s},{\bf \Psi})\mapsto u ({\bf s},{\bf \Psi})$ as several consecutive $s_r$'s are collapsing.
\s
Let us explain briefly how the nonlinear Fourier transform\index{Non linear Fourier transform} is related to spectral analysis. If $u\in H^{1/2}_+(\S ^1)$, recall that the Hankel operator\index{Hankel operator} of symbol $u$ is the operator $H_u:L^2_+(\S ^1)\rightarrow L^2_+(\S ^1)$ defined by 
$$H_u(h)=\Pi (u\overline h)\ .$$
It can be shown that  $H_u^2$ is a positive selfadjoint trace class operator. If $S$ is the shift operator defined by
$$Sh (z)=zh(z)\ ,$$
$H_u$ satisfies 
$$S^*H_u=H_uS=H_{S^*u}\ .$$
We denote by $K_u$ this new Hankel operator. Let us say that a positive real number $s$ is a singular value\index{singular value} associated to $u$ if $s^2$ is an eigenvalue of $H_u^2$ or $K_u^2$. The main point in Theorem \ref{FT} is that the list $s_1>\dots >s_r>\dots $ is the list of singular values associated to $u=u ({\bf s},{\bf \Psi})$, and that the corresponding list $\Psi_1,\dots ,\Psi _r,\dots $ describes the action of $H_u$ and of $K_u$ on the possibly multidimensional eigenspaces of $H_u^2$, $K_u^2$ respectively. This makes  more precise a theorem of Adamyan-Arov-Krein about the structure of Schmidt pairs \index{Schmidt pairs} of Hankel operators \cite{AAK}. We refer to section \ref{Blaschke} for more details.  As a consequence of Theorem \ref{FT}, we get inverse spectral theorems on Hankel operators, which generalize to singular values with arbitrary multiplicity the ones we had proved in \cite{GG2} and \cite{GG3} for simple singular values. Therefore the nonlinear Fourier transform can be seen as an inverse spectral transform. \s
We close this introduction by describing the organization of this monograph. In chapter 2 we recall the Lax pair structure discovered in \cite{GG1} and  its application to large time bounds for high Sobolev norms of solutions to the cubic Szeg\H{o} equation. Chapter 3 is devoted to the study of singular values of  Hankel operators $H_u$ and $K_u$. In particular we introduce the Blaschke products $\Psi $, which provide the key of the understanding of multiplicity phenomena. Chapter 4 contains the proof of the first part of Theorem \ref{FT}, establishing the one to one character and the continuity of the nonlinear Fourier transform. Chapter 5  completes the proof of Theorem \ref{FT} by describing the Szeg\H{o} dynamics through the nonlinear Fourier transform, and infers the almost periodicity \index{almost periodicity}property of solutions as claimed in the first part of Theorem \ref{dynamicIntro}. In Chapter 6, we establish the Baire genericity of unbounded trajectories in $H^
 s$ for every $s>\frac 12$ as stated in the second part of  Theorem \ref{dynamicIntro}. The main ingredient is an
 instability phenomenon induced by the collapse of consecutive singular values. Finally, Chapter 7 addresses various geometric aspects of our nonlinear Fourier transform in connection to the Szeg\H{o} hierarchy defined in \cite{GG1}. Specifically, we show how the nonlinear Fourier transform allows to define action angle variables on special invariant symplectic submanifolds of the phase space, and we discuss  the structure of the corresponding 
 Lagrangian tori obtained by freezing the action variables.\s
Let us mention that  part of these results was the content of the preliminary preprint \cite{GG6}, and that another part was announced in the proceedings paper \cite{GG7}.

\chapter{Hankel operators and the Lax pair structure}\label{chapter Lax pair}\index{Lax pair}

As we pointed out in the introduction, Hankel operators \index{Hankel operator}arise naturally to understand the cubic Szeg\H{o} equation.
In this chapter, we first give some basic definition and properties of Hankel operators and we recall the Lax pair structure of the cubic Szeg\H{o} equation.
Finally, we show how to use this structure to get large time bounds for high Sobolev norms of solutions to the cubic Szeg\H{o} equation.\index{Sobolev spaces}

 \section{Hankel operators $H_u$ and $K_u$}\index{Hankel operator}
 
Let $u\in H^{\frac 12}_+(\S^1)$. We denote by $H_u$\index{$H_u$|see {Hankel operator}} the $\C $--antilinear operator defined on $L^2_+(\S^1)$ as
$$H_u(h)=\Pi (u\overline h)\ ,\ h\in L^2_+(\S^1 )\ .$$
In terms of Fourier coefficients, this operator reads
$$\widehat{H_u(h)} (n)=\sum _{p=0}^\infty \hat u(n+p)\overline{\hat h(p)}\ .$$
In particular, its Hilbert--Schmidt norm 
$$\Vert H_u \Vert_{HS}=(\sum_{n,p\ge 0}\vert \hat u(n+p)\vert^2)^{1/2}= (\sum_{\ell\ge 0}(1+\ell)|\hat u(\ell)|^2)^{1/2}\simeq \Vert u\Vert _{H^{1/2}}$$ is finite for $u\in H^{\frac 12}_+(\S^1 )$ and 
\begin{equation}\label{borne}
\Vert H_u\Vert_{\mathcal L(L^2_+)}\le \Vert H_u\Vert_{HS}\ .
\end{equation}
We call $H_u$ the Hankel operator of symbol $u$. It is well known from  Kronecker's theorem, \cite{Kr}, \cite{N}, \cite{Pe2}, that $H_u$ is of finite rank if and only if $u$ is a rational function without poles in the closure of the unit disc. 

In fact, the definition of Hankel operators may be extended to a larger class of symbol.
By the  Nehari theorem (\cite{Ne}), $H_u$ is well defined and bounded on $L^2_+(\S^1)$ if and only if $u$ belongs to  $\Pi(L^\infty(\S^1))$ or equivalently to $BMO_+(\S^1)$.  Moreover, by the Hartman theorem (\cite{Ha}), it is a compact operator if and only if $u$ is the projection of a continuous function on the torus,  or equivalently if and only if it belongs to $VMO_+(\S^1)$\index{$VMO_+(\S^1)$} with equivalent norms. In the following, we will mainly consider the class of Hilbert-Schmidt Hankel operators. We will generalize part of our result to compact Hankel operators in section \ref{compact}.\\

Notice that this definition of Hankel operators is different from the standard ones used in references \cite{N}, \cite{Pe2}, where Hankel operators are rather defined as linear operators from $L^2_+$ into its orthogonal complement. The link between these two definitions can be easily established by means of the involution
$$f^\sharp ( {\rm e}^{ix})= {\rm e}^{-ix} \overline {f({\rm e}^{ix})}\ .$$
Notice also that $H_u$ satisfies the following self adjointness identity,
\begin{equation}\label{selfadjoint}
(H_u(h_1)\vert h_2)=(H_u(h_2)\vert h_1)\ ,\ h_1, h_2\in L^2_+(\S^1 )\ .
\end{equation}
In particular, $H_u$ is a $\R$-linear symmetric operator for the real inner product
$$\langle h_1,h_2\rangle:= {\rm Re} (h_1\vert h_2)\ .$$ 
A fundamental property of Hankel operators is their connection with the shift operator $S$, defined on $L^2_+(\S^1 )$ as
$$ Su ({\rm e}^{ix})= {\rm e}^{ix}u({\rm e}^{ix})\ .$$
This property reads
$$S^*H_u=H_uS=H_{S^*u}\ ,$$
where $S^*$ denotes the adjoint of $S$. We denote by $K_u$ \index{$K_u$|see {Shifted Hankel operator}} this operator, and call it the shifted Hankel operator\index{Shifted Hankel operator} of symbol $u$. Hence
\ben \label{Ku}
K_u:=S^*H_u=H_uS=H_{S^*u}\ .
\een
Notice that, for $u\in H^{1/2}_+(\S^1)$, $K_u$ is Hilbert--Schmidt and symmetric as well. As a consequence, operators $H_u^2$ and $K_u^2$ are $\C $--linear
trace class positive operators on $L^2_+(\S^1 )$. Recall that the singular values of $H_u$ and $K_u$ correspond to the square roots of the eigenvalues of the self-adjoint positive operators $H_u^2$ and $K_u^2$.  Moreover, operators $H_u^2$ and $K_u^2$ are related by the following important identity,
\begin{equation}\label{Ku2}
K_u^2=H_u^2-(\cdot \vert u)u\ .
\end{equation}
Let us consider the special case where $H_u$ is an operator of finite rank $N$. From identity (\ref{selfadjoint}), $H_u$ and $H_u^2$ have the same kernel, hence have the same rank. Then (\ref{Ku2}) implies that the rank of $K_u$ is $N-1$ or $N$. 
\begin{definition}\label{defV(d)}
If $d$ is a positive integer, we denote by $\mathcal V(d)$ \index{$\mathcal V(d)$}the set of symbols $u$ such that the sum of the rank of $H_u$ and of the rank of $K_u$ is $d$.
\end{definition}
The Kronecker theorem can be made more precise by the following statement, see the appendix of \cite{GG1}. The set $\mathcal V(d)$ is a complex K\"ahler $d$--dimensional submanifold of $L^2_+(\S ^1)$, which consists of functions of the form
$$u({\rm e}^{ix})=\frac{A({\rm e}^{ix})}{B({\rm e}^{ix})}\ ,$$
where $A, B$ are polynomials with no common factors, $B$ has no zero in the closed unit disc, $B(0)=1$, and
\begin{itemize}
\item If $d=2N$ is even, the degree of $A$ is at most $N-1$ and the degree of $B$ is exactly $N$.
\item If $d=2N-1$ is odd, the degree of $A$ is exactly $N-1$ and the degree of $B$ is at most $N-1$.
\end{itemize}

\section{The Lax pair structure}\label{Lax}\index{Lax pair}
In this section, we recall the Lax pairs associated to the cubic Szeg\H{o} equation, see \cite{GG1}, \cite{GG2}. 
First we introduce the notion of a Toeplitz operator. Given $b\in L^\infty (\S^1 )$, we define $T_b:L^2_+\to L^2_+$ as
\begin{equation}\label{Toeplitz}T_b(h)=\Pi (bh)\ ,\ h\in L^2_+\ .
\end{equation}
Notice that $T_b$ is bounded and $T_b^*=T_{\overline b}$. 
The starting point is the following lemma.
\begin{lemma}\label{abc}
Let $a, b, c\in H^s_+$, $s>\frac 12$. Then
$$H_{\Pi (a\overline b c)}=T_{a\overline b}H_c+H_aT_{b\overline c} -H_aH_bH_c\ .$$
\end{lemma}
\begin{proof}
Given $h\in L^2_+$, we have
\begin{eqnarray*}
H_{\Pi (a\overline b c)}(h)&=&\Pi (a\overline b c\overline h)=\Pi (a\overline b\Pi (c\overline h))+\Pi (a\overline b(I-\Pi )(c\overline h))\\
&=&T_{a\overline b}H_c(h)+H_a(g)\ ,\ g:=b\overline {(I-\Pi )(c\overline h)}\ .
\end{eqnarray*}
Since $g\in L^2_+$, 
$$g=\Pi (g)=\Pi (b\overline ch)-\Pi (b\overline {\Pi (c\overline h)})=T_{b\overline c}(h)-H_bH_c(h)\ .$$
This completes the proof.
\end{proof}
Using Lemma \ref{abc} with $a=b=c=u$, we get
\begin{equation}\label{Hu3}
H_{\Pi (\vert u\vert ^2u)}=T_{\vert u\vert ^2}H_u+H_uT_{\vert u\vert ^2}-H_u^3\ .
\end{equation}
\begin{theorem}\label{Lax pair}\index{Lax pair}
Let $u\in C^\infty (\R ,H^s_+), s>\frac 12, $ be a solution of (\ref{szegoIntro}). Then
\begin{eqnarray*}
\frac{dH_u}{dt}&=&[B_u,H_u]\ ,\ B_u:=\frac i2H_u^2-iT_{\vert u\vert ^2}\ ,\\
\frac{dK_u}{dt}&=&[C_u,K_u]\ ,\ C_u:=\frac i2K_u^2-iT_{\vert u\vert ^2}\ .
\end{eqnarray*}
\end{theorem}
\begin{proof} Using equation (\ref{szegoIntro}) and identity (\ref{Hu3}), 
$$\frac{dH_u}{dt}=H_{-i\Pi (\vert u\vert ^2u)}=-i H_{\Pi (\vert u\vert ^2u)}=-i(T_{\vert u\vert ^2}H_u+H_uT_{\vert u\vert ^2}-H_u^3)\ .$$
Using the antilinearity of $H_u$, this leads to the first identity. For the second one, we observe that
\begin{equation}\label{KPi}
K_{\Pi (\vert u\vert ^2u)}=H_{\Pi (\vert u\vert ^2u)}S=T_{\vert u\vert ^2}H_uS+H_uT_{\vert u\vert ^2}S-H_u^3S\ .
\end{equation}
Moreover, notice that
$$T_b(Sh)=ST_b(h)+(bSh\vert 1)\ .$$
In the case $b=\vert u\vert ^2$, this gives
$$T_{\vert u\vert ^2}Sh =ST_{\vert u\vert ^2}h+(\vert u\vert ^2Sh\vert 1)\ .$$
Moreover,
$$(\vert u\vert ^2Sh\vert 1)=(u\vert u\overline {Sh})= (u\vert K_u(h))\ .$$
Consequently,
$$H_uT_{\vert u\vert ^2}Sh=K_uT_{\vert u\vert ^2}h+(K_u(h)\vert u)u\ .$$
Coming back to (\ref{KPi}), we obtain
$$K_{\Pi (\vert u\vert ^2u)}=T_{\vert u\vert ^2}K_u+K_uT_{\vert u\vert ^2}-(H_u^2-(\cdot \vert u)u) K_u\ .$$
Using identity (\ref{Ku2}), this leads to
\begin{equation}\label{Ku3}
K_{\Pi (\vert u\vert ^2u)}=T_{\vert u\vert ^2}K_u+K_uT_{\vert u\vert ^2} - K_u^3\ .
\end{equation}
The second identity is therefore a consequence of antilinearity and of 
$$\frac{dK_u}{dt}=-iK_{\Pi (\vert u\vert ^2u)}\ .$$
\end{proof}
In the sequel, we denote by $\mathcal L(L^2_+)$ the Banach space  of bounded linear operators on $L^2_+$.
Observing that $B_u, C_u$ are linear and antiselfadjoint, we obtain, following a classical argument due to Lax \cite{L},
\begin{corollary}\label{UV}
Under the conditions of Theorem \ref{Lax pair}, define $U=U(t)$, $V=V(t)$ the solutions of the following linear ODEs
on $\mathcal L(L^2_+)$,\index{Lax pair}
$$\frac {dU}{dt}=B_uU\ ,\ \frac{dV}{dt}=C_uV\ ,\ U(0)=V(0)=I\ .$$
Then $U(t), V(t)$ are unitary operators and
$$H_{u(t)}=U(t)H_{u(0)} U(t)^*\ ,\ K_{u(t)}=V(t)K_{u(0)} V(t)^*\ .$$
\end{corollary}

In other words, the Hankel operators $H_u$ and $K_u$ remain unitarily equivalent to the Hankel operators associated to their initial data under the cubic Szeg\H{o} flow for initial datum in $H^s(\S^1)$, $s>1/2$.  
 In particular, the cubic Szeg\H{o} flow preserves the eigenvalues of $H_u^2$ and of $K_u^2$, hence the singular values of $H_u$ and $K_u$. By a standard continuity argument, this result extends to initial data in $H^{1/2}_+(\S^1)$. We recover that the Hilbert-Schmidt norm-- being the $\ell^2$-norm of the singular values of $H_u$-- hence the $H^{1/2}_+$ norm of the symbol is a conservation law. It is therefore natural to study the spectral properties of these Hankel operators\index{Hilbert-Schmidt}. 
 
As a special case, the ranks of $H_u$ and $K_u$ are conserved by the  cubic Szeg\H{o} flow, which therefore acts as a Hamiltonian flow on all the symplectic manifolds $\mathcal V(d)$.
 
\section{Application: an exponential bound for Sobolev norms}\index{Sobolev spaces}

In this short section, we show how the Lax pair structure, combined with harmonic analysis results on Hankel operators,
lead to an a priori bound on the possible long time growth of the Sobolev norms of the solution.
\s
From the work of Peller (\cite{Pe2}), a Hankel operator belongs to the Schatten class $S_p$, $1\le p<\infty$,  if and only if its symbol belongs to the Besov space $B^{1/p}_{p,p}(\S ^1)$, with
$${\rm Tr}(|H_u|^p)\simeq  \Vert u\Vert_{B^{1/p}_{p,p}}^p\ .$$
In particular, since, for every $s>1$,  $H^s_+(\S ^1)\subset B^1_{1,1}(\S ^1)\subset L^\infty (\S ^1)$, we obtain
\begin{equation}\label{Peller}
\Vert u\Vert_{L^\infty}\le A{\rm Tr}(|H_u|)\le B_s \Vert u\Vert_{H^s}, \; s>1.
\end{equation}
 This provides the following result, already observed in \cite{GG1}.
 \begin{theorem}\label{Sobbound}(\cite{GG1}).
 For every datum $u_0\in H^s_+$ for some $s>1$, the solution is bounded in $L^\infty $,  with
\begin{equation}\label{Linftybound}
\sup_{t\in \R}\Vert Z(t)u_0\Vert_{L^\infty}\le C_s \Vert u_0\Vert_{H^s}\ .
\end{equation}
Furthermore,\begin{equation}\label{Hsbound}
\Vert Z(t)u_0\Vert_{H^s}\le \Vert C_s'u_0\Vert_{H^s}{\rm e}^{C'_s\Vert u_0\Vert _{H^s}^2 \vert t\vert },\; s>1.
\end{equation} 
\end{theorem}
\begin{proof}
The first statement is an immediate consequence of (\ref{Peller}) and of the conservation law
$${\rm Tr}(\vert H_{Z(t)u_0}\vert )={\rm Tr}(\vert H_{u_0}\vert )\ .$$
From Duhamel formula, $u(t):=Z(t)u_0$ is given by 
$$u(t):=u_0 -i\int_0^t \Pi(|u|^2u)(t') dt'\ ,$$ so that 
\begin{eqnarray*}\Vert u(t)\Vert_{H^s}&\le& \Vert u_0\Vert_{H^s}+\int_0^t \Vert |u(t')|^2u(t')\Vert_{H^s}dt'\\
&\le&\Vert u_0\Vert_{H^s}+D_s\int_0^t \Vert u(t')\Vert_{L^\infty}^2\Vert u(t')\Vert_{H^s}dt\\
&\le&\Vert u_0\Vert_{H^s}+D_sC'_s\int_0^t \Vert u_0\Vert_{H^s}^2\Vert u(t')\Vert_{H^s}dt\\
\end{eqnarray*}
and the bound (\ref{Hsbound}) follows from the Gronwall lemma.
\end{proof}

\chapter{Spectral analysis}\label{chapter spectral analysis}\index{Hankel operator}

In this chapter, we establish several properties of singular values of Hankel operators $H_u$ and $K_u$ introduced in Chapter 1. In particular, we prove -- see Lemma \ref{EF} ---  that, for  a given singular value of either $H_u$ or $K_u$,  the difference of the multiplicities is $1$. In section \ref{Bateman} we review trace and determinant formulae relating the singular values to other norming constants. In section \ref{sectionAAK}, we revisit two important theorems by Adamyan--Arov--Krein on the structure of Schmidt pairs \index{Schmidt pairs}for Hankel operators, and on approximation of arbitrary symbols by rational symbols. Section \ref{Blaschke} is the most important one of this chapter. Choosing special Schmidt pairs, we construct Blaschke products \index{Blaschke product}which allow to describe the action of $H_u$ and $K_u$ on the eigenspaces of $H_u^2$ and $K_u^2$.

\section{Spectral decomposition of the operators $H_u$ and $K_u$}\label{SpectralAnalysis}

This section is devoted to a precise spectral analysis of operators $H_u^2$ and $K_u^2$ on the  closed range of $H_u$. This spectral analysis is closely related to the construction of our non linear Fourier transform\index{Non linear Fourier transform}.

For every $s\ge 0$ and $u\in VMO_+(\S^1 ):=VMO(\S^1 )\cap L^2_+(\S^1 )$\index{$VMO_+(\S^1)$}, we set\index{$E_u(s)$} \index{$F_u(s)$}
\ben \label{propres}
E_u(s):=\ker (H_u^2-s^2I)\ ,\ F_u(s):=\ker (K_u^2-s^2I)\ .
\een
Notice that $E_u(0)=\ker H_u\ ,\  F_u(0)=\ker K_u$. Moreover, from the compactness of $H_u$, if $s>0$, $E_u(s)$ and $F_u(s)$ are finite dimensional.
Using (\ref{Ku}) and (\ref{Ku2}), one can show the following result.

\begin{lemma}\label{EF}
Let $u\in VMO_+(\S^1 )\setminus \{ 0\} $ and $s>0$ such that $$E_u(s)+ F_u(s)\ne \{ 0\}\ .$$
Then one of the following properties holds. 
\begin{enumerate}
\item $\dim E_u(s)=\dim F_u(s)+1$,  $u \not \perp E_u(s)$, and $F_u(s)=E_u(s)\cap u^\perp $.
\item $\dim F_u(s)=\dim E_u(s)+1$,  $u \not \perp F_u(s)$, and $E_u(s)=F_u(s)\cap u^\perp $.
\end{enumerate}
\end{lemma}

\begin{proof}
Let $s>0$ be such that $E_u(s)+F_u(s)\neq \{0\}$. We first claim that either $u\perp E_u(s)$ or $u\perp F_u(s)$. Assume first $u\not\perp E_u(s)$,  then there exists $h\in E_u(s)$ such that $(h\vert u)\neq 0$. From equation (\ref{Ku2}), $$-(h\vert u)u=(K_u^2-s^2I)h\in (F_u(s))^\perp\ ,$$ hence $u\perp F_u(s)$. Similarly, if $u\not\perp F_u(s)$, then $u\perp E_u(s)$.\\
 Assume $u\perp F_u(s)$. Then, for any $h\in  F_u(s)$, as $K_u^2=H_u^2-(\cdot\vert u)u$, $H_{u}^2(h)=K_u^2(h)=s^2 h$, hence  $  F_u(s)\subset   E_u(s)$. We claim that this inclusion is strict. Indeed, suppose it is an equality. Then $H_{u}$ and $K_{u}$ are both 
automorphisms of the vector space $$N:= F_u(s)=   E_u(s)\ .$$ 
Consequently, since $K_u=S^*H_u$, $S^*(N)\subset N$. 
On the other hand, 
 since every $h\in N$ is orthogonal to $u$, we have
$$0=(H_{u}(h)\vert u)=(1\vert H_{u}^2 h)=s ^2(1\vert h)\ ,$$
hence $N\perp 1$.  Therefore, for every $h\in N$, for every integer $k$,
$(S^*)^k(h)\perp 1$. Since $S^k(1)=z^k$, we conclude that all the Fourier coefficients of $h$ are $0$, hence $N=\{ 0\}$,
a contradiction. Hence, the inclusion of $F_u(s)$ in $E_u(s)$ is strict and, necessarily $u\not\perp E_u(s)$ and $F_u(s)=E_u(s)\cap u^\perp$. One also has $\dim E_u(s)=\dim F_u(s)+1$.\\
One proves as well that if $u\perp E_u(s)$ then $u\not\perp F_u(s)$, $E_u(s)=F_u(s)\cap u^\perp$ and $\dim F_u(s)=\dim E_u(s)+1$.
This gives the result.
\end{proof}
We define\index{$\Sigma_H(u)$}\index{$\Sigma_K(u)$}
\begin{equation}\label{SigmaH}
\Sigma_H(u):=\left\{s\ge 0;\; u\not  \perp  E_u(s)\right\}, 
\end{equation}
\begin{equation}\label{SigmaK}
\Sigma_K(u):=\left\{s\ge 0;\; u\not \perp F_u(s)\right\}.
\end{equation}
 Remark first that $0\notin \Sigma_H(u)$, since $u=H_u(1)$ belongs to the range of $H_u$ hence, is orthogonal to its kernel. 
 From Lemma \ref{EF}, if $s\in \Sigma_H(u)$ then $\dim E_u(s)=\dim F_u(s)+1$ and if $s\in \Sigma_K(u)\setminus\{0\}$, $\dim F_u(s)=\dim E_u(s)+1$. Consequently,  $\Sigma_H(u)$ and $\Sigma_K(u)$ are disjoint. We will use the following terminology.\\
If $s\in \Sigma_H(u)$, we say that $s$ is  {\sl $H$-dominant}. \\
If $s\in \Sigma_K(u)\setminus\{0\}$, we say that $s$ is  {\sl $K$-dominant}.\\
Elements of the set $\Sigma _H(u)\cup (\Sigma _K(u)\setminus \{ 0\} )$ are called {\sl singular values associated to $u$}. 
If $s$ is such a singular value, the {\sl dominant multiplicity} of $s$ is defined to be the maximum of $\dim E_u(s)$ and of $\dim F_u(s)$.

\begin{lemma}\label{rigidity}
\begin{enumerate}
\item $\Sigma_H(u)$ and $\Sigma_K(u)$ have the same cardinality.
\item If $(\rho_j)$ denotes the decreasing sequence of elements of $\Sigma _H(u)$, and $(\sigma_k)$ 
the decreasing sequence of  elements of  $\Sigma_K(u)$, then
 $$\rho_1>\sigma _1>\rho_2>\dots $$
 \end{enumerate}
 \end{lemma}
 
\begin{proof}
Since $K_u^2=H_u^2-(\, .\, \vert u)u$, the cyclic spaces generated by $u$ under the action of $H_u^2$ and $K_u^2$, namely
$$\langle u\rangle _{H_{u}^2} :={\rm clos}\;{\rm span}{\{H_u^{2k},\;k\in \N\}}$$
$$\langle u\rangle _{K_{u}^2} :={\rm clos}\;{\rm span}{\{K_u^{2k},\;k\in \N\}}$$ are  equal.
By the spectral theory of $H_{u}^2$ and of $K_u^2$, we have the orthogonal decompositions,
$$L^2_+=\overline{\oplus_{s\ge 0}E_u(s)}=\overline{\oplus_{s\ge 0}F_u(s)}\ .$$
Writing $u$ according to these two orthogonal decompositions yields
$$u=\sum _{\rho \in \Sigma _H(u)}u_\rho =\sum _{\sigma  \in \Sigma _K(u)}u'_\sigma \ .$$
Consequently, the cyclic spaces are given by
$$\langle u\rangle _{H_{u}^2} =\overline {\oplus _{\rho \in \Sigma _H(u)}\C u_\rho }\ ,\ \langle u\rangle _{K_{u}^2} =\overline {\oplus _{\sigma \in \Sigma _K(u)}\C u'_\sigma }\ .$$
 This proves that $\Sigma _H(u)$ and $\Sigma _K(u)$ have the same --- possibly infinite --- number of elements. 
\s
Now, we show that the singular values are alternatively $H$-dominant and $K$-dominant. Recall  that the singular values of a bounded operator $T$ on a Hilbert space $\mathcal H$, are given by the following min-max formula. \index{min-max formula} For every $m\ge 1$, denote by $\mathcal F _m$ the set of linear subspaces of $\mathcal H$ of dimension at most $m$.  The $m$-th singular value of $T$ is given by
\begin{equation}\label{minmax}\lambda_m(T)=\min_{F\in \mathcal F_{m-1}}\max_{f\in F^\perp, \Vert f\Vert=1}\Vert T(f)\Vert.\end{equation} 
Using equation (\ref{Ku2}) and this formula, we get
$$\lambda_1(H_u^2)\ge \lambda_1(K_u^2)\ge \lambda_2(H_u^2)\ge \lambda_2(K_u^2)\ge \dots$$
Let $s_1^2=\lambda_1(H_u^2)$. We claim that $s_1$ is $H$-dominant. Indeed, denote by $m_1$ the dimension of $E_u(s_1)$. Hence, $$s_1^2=\lambda_1(H_u^2)=\dots=\lambda_{m_1}(H_u^2)>\lambda_{m_1+1}(H_u^2).$$  Since $$ \lambda_{m_1+1}(K_u^2)\le \lambda_{m_1+1}(H_u^2)<s_1^2,$$ the dimension of $F_u(s_1)$ is at most $m_1$. Hence by Lemma \ref{EF}, it is exactly $m_1-1$. Let $s_2^2=\lambda_{m_1}(K_u^2)$. Denote by $m_2$ the dimension of $F_u(s_2)$ so that $$s_2^2=\lambda_{m_1}(K_u^2)=\dots=\lambda_{m_1+m_2-1}(K_u^2)>\lambda_{m_1+m_2}(K_u^2)\ge \lambda_{m_1+m_2+1}(H_u^2).$$ As before, it implies that the dimension of $E_u(s_2)$ is at most $m_2$.  By Lemma \ref{EF}, the dimension of $E_u(s_2)$ is $m_2-1$ and $s_2$ is $K$-dominant. An easy induction argument allows to conclude.

\end{proof}

At this stage, we introduce special classes of functions  on the circle, connected to the properties of the associated singular values. Given a finite sequence $(d_1,\dots ,d_n)$ of nonnegative integers, we denote by $ \mathcal V_{(d_1,\dots ,d_n)}$  \index{$\mathcal V_{(d_1,\dots ,d_n)}$}the class of functions $u$ having $n$ associated singular values $s_1>\dots >s_n$, of dominant multiplicities $d_1+1,\dots , d_n+1$ respectively. In view of definition \ref{defV(d)}, we observe that
\begin{equation}\label{inclusionV}
 \mathcal V_{(d_1,\dots ,d_n)}\subset \mathcal V(d)\ ,\ d=n+2\sum_{r=1}^nd_r\ .
\end{equation}
Similarly, given a sequence $(d_r)_{r\ge 1}$ of nonnegative integers, 
$\mathcal V^{(2)}_{(d_r)_{r\ge 1}} $ the class of functions $u\in H^{\frac 12}(\S ^1)$ having a decreasing sequence of associated singular  values with dominant multiplicities $(d_r+1)_{ r\ge 1}$.

 \begin{lemma}\label{rigidity2}
Let  $\rho\in \Sigma_H(u)$ and $\sigma\in \Sigma_K(u)$; let $u_\rho$ and $u'_\sigma$ denote respectively the orthogonal projections of $u$ onto $E_u(\rho)$, and onto $F_u(\sigma)$,  then
\begin{enumerate}
 \item If $\rho\in \Sigma_H(u)$, \begin{equation}\label{urho}
u_\rho=\Vert u_\rho\Vert^2\sum_{\sigma\in \Sigma_K(u)}\frac{u'_\sigma}{\rho^2-\sigma^2}\ ,\end{equation}
 \item If $\sigma\in \Sigma_K(u)$,
\begin{equation}\label{usigma}u'_\sigma=\Vert u'_\sigma\Vert^2 \sum_{\rho\in \Sigma_H(u)}\frac{u_\rho}{\rho^2-\sigma^2}\ .\end{equation}
\item A nonnegative number $\sigma $ belongs to $\Sigma _K(u)$ if and only if it does not belong to $\Sigma _H(u)$ and
\begin{equation}\label{eqsigma}\sum_{\rho \in \Sigma _H(u)}\frac{\Vert u_\rho \Vert ^2}{\rho ^2-\sigma ^2}=1\ .\end{equation}
\end{enumerate}
\end{lemma}
\begin{proof}
Let us prove (\ref{urho}) and (\ref{usigma}). Observe that, by the Fredholm alternative, for $\sigma >0$, 
$H_{u}^2-\sigma ^2I$ is an automorphism of $ E_u(\sigma)^\perp $. Consequently, if moreover $\sigma \in \Sigma _K(u)$,
$u\in E_u(\sigma )^\perp $ and there exists $v\in E_u(\sigma )^\perp $ unique such that 
$$(H_u^2-\sigma ^2I)v=u\ .$$
We set $v:=(H_{u}^2-\sigma^2I)^{-1}(u)$. If $\sigma =0\in \sigma _K(u)$, of course $H_u^2$ is no more a Fredholm operator. However
let us prove that there still exists $w\in E_u(0)^\perp $ such that
$$H_u^2(w)=u\ .$$
Indeed, $E_u(0)=\ker H_u^2=\ker H_u$, $F_u(0)=\ker K_u^2=\ker K_u$, and since $K_u=S^*H_u$, therefore $E_u(0)\subset F_u(0)$. As $u=H_u(1)\in {\rm Ran}(H_u)\subset E_u(0)^\perp$, the hypothesis $u\not \perp F_u(0)$ implies that the inclusion $E_u(0)\subset F_u(0)$ is strict. This means that there exists $w\in F_u(0)\cap E_u(0)^\perp$, $w\neq 0$. It means $S^*H_u(w)=0$ hence $H_u w$ is a constant function which is not zero since $w\in E_u(0)^\perp$. Multiplying $w$ by a convenient complex number we may assume that $H_u(w)=1$, whence $H_u^2(w)=u$, and this characterizes $w\in E_u(0)^\perp$. Again we set $w:=(H_u^2)^{-1}(u)\ .$

For every $ \sigma \in \Sigma _K(u)$, the equation
$$K_{u}^2h=\sigma ^2h $$
is equivalent to
$$(H_{u}^2-\sigma^2I)h=(h\vert u)u\ ,$$
or
$h\in \C (H_{u}^2-\sigma^2I)^{-1}(u)\oplus  E_u(\sigma)\ ,$
with \begin{equation}\label{condsigma}  ((H_{u}^2-\sigma^2I)^{-1}(u)\vert u)=1\ . \end{equation}
We apply this property to $h=u'_\sigma$. Since $u'_\sigma \in E_u(\sigma )^\perp $, there exists $\lambda\in\C$ with $u'_\sigma=\lambda(H_{u}^2-\sigma^2I)^{-1}(u)$.
Applying equation (\ref{condsigma}), this  leads to
$$\frac { u'_{\sigma }}{\Vert u'_{\sigma }\Vert ^2}  =(H_{u}^2-\sigma^2I)^{-1}(u)\ .$$
In particular, if $\rho\in\Sigma_H(u),$ $\sigma\in \Sigma_K(u)$,
$$\left(\frac{u'_\sigma}{\Vert u'_\sigma\Vert^2}\Big \vert \frac{u_\rho}{\Vert u_\rho\Vert^2}\right)=\frac 1{\rho^2-\sigma^2}\ .$$
This leads to  equations (\ref{urho}) and (\ref{usigma}). Finally, equation (\ref{eqsigma}) is nothing but the expression of (\ref{condsigma})  in view of  equation (\ref{urho}).
\s
This completes the proof of Lemma \ref{rigidity2}.
\end{proof}
\section{Some Bateman-type formulae}\label{Bateman}

In this section, we establish some formulae linking the singular values of a Hankel operator and its shifted operator.\\
Let $u\in VMO_+(\S^1)$. For $\rho\in \Sigma_H(u)$, we denote by $\sigma(\rho)$ the biggest element of $\Sigma_K(u)$ which is smaller than $\rho$. When  ${\displaystyle \prod _{\rho\in \Sigma_H(u)}\frac{\si (\rho)^2}{\rh ^2}=0}$, we define $\rho(\sigma)$ to be the biggest element of $\Sigma_H(u)$ which is smaller than $\sigma$. This is always well defined for infinite sequences tending to zero. In the case of finite sequence, the hypothesis implies that $0\in\Sigma_K(u)$ and lemma \ref{EF} and lemma \ref{rigidity} allows to define $\rho(\sigma)$ in that case for any $\sigma\in\Sigma_k(u)\setminus\{0\}$.\\
In this section, we state and prove  several formulae connecting these sequences. These formulae are based on the special case of a general formula for the resolvent of a finite rank perturbation of an operator, which seems to be due to Bateman \cite{Ba} in the framework of Fredholm integral equations. Further references can be found in Chap. II, sect. 4.6 of  \cite{KK}, section 106 of \cite{AG} and \cite{Na}, from which we borrowed this information.
\begin{proposition}\label{formulae}
The following functions coincide respectively for $x$ outside the set $\{\frac 1{\rho^2}\}_{\rho\in \Sigma_H(u)}$ and outside the set $\{\frac 1{\sigma^2}\}_{\sigma\in \Sigma_K(u)\setminus\{0\}}$.
\begin{equation}\label{J(x)Appendix}
\prod_{\rho\in\Sigma_H(u)}\frac{1-x\sigma(\rho)^2}{1-x\rho^2}=1+x\sum_{\rho\in \Sigma_H(u)} \frac{\Vert u_\rho\Vert^2}{1-x\rho^2}
\end{equation}
\begin{equation}\label{1/J(x)Appendix}
\prod_{\rho\in\Sigma_H(u)}\frac{1-x\rho^2}{1-x\sigma(\rho)^2}=1-x\left (\sum_{\sigma\in \Sigma_K(u)}\frac{\Vert u'_\sigma\Vert ^2}{1-x\sigma^2}\right )\ .
\end{equation}

Furthermore, 
\begin{equation}\label{sommenu}
1-\sum _{\rho\in\Sigma_H(u)}\frac{\Vert u_\rho\Vert^2}{\rho^2}=\prod _{\rho\in \Sigma_H(u)}\frac{\si(\rho) ^2}{\rh ^2}\ ,
\end{equation}
and, if ${\displaystyle \prod _{\rho\in \Sigma_H(u)}\frac{\si(\rho) ^2}{\rh ^2}=0}$, 
\begin{equation}\label{sommenurho}
\sum _{\rho\in\Sigma_H(u)}\frac{\Vert u_\rho\Vert^2}{\rho^4}=\frac 1{\rh _{\rm max}^2}\prod _{\sigma\in \Sigma_K(u)\setminus\{0\}}\frac{\si ^2}{\rh (\sigma)^2}\ .
\end{equation}
Here $\rh _{\rm max}$ denotes the biggest element of $\Sigma_H(u)$.
For $\rho\in \Sigma_H(u)$, the $\Vert u_\rho\Vert^2$'s are given by 
\begin{equation}\label{tau}
\Vert u_\rho\Vert^2= \left(\rho^2-\sigma(\rho)^2\right)\prod_{\rho'\ne \rho}\frac{\rho^2-\sigma(\rho')^2}{\rho^2-{\rho'}^2}\ .\end{equation}
and the $\Vert u'_{\sigma(\rho)}\Vert^2$'s are given by 
\begin{equation}\label{kappa}
\Vert u'_{\sigma(\rho)}\Vert^2=(\rho^2-\sigma(\rho)^2)\prod_{\rho'\ne \rho}\frac{\sigma(\rho)^2-{\rho'}^2}{\sigma(\rho)^2-\sigma(\rho')^2}.
\end{equation}
The kernel of $K_u$ is non trivial if and only if ${\displaystyle \prod _{\rho\in \Sigma_H(u)}\frac{\si(\rho) ^2}{\rh ^2}=0}$, and in that case
$$\Vert u'_0\Vert^2=\rh _{\rm max}^2\prod _{\sigma\in \Sigma_K(u)\setminus\{0\}}\frac{\rh (\sigma)^2}{\si ^2}$$
\end{proposition}
\begin{proof}
For $x\notin \{\frac 1{\rho^2}\}_{\rho\in\Sigma_H(u)}$, we set $$J(x):=((I-xH_u^2)^{-1}(1)\vert 1).$$
We claim that
\begin{equation}\label{J}
J(x)=\prod_{\rho\in\Sigma_H(u)} \frac{1-x\sigma(\rho)^2}{1-x\rho^2}.
\end{equation}
Indeed, let us first assume that $H_u^2$ and $K_u^2$ are  of trace class and compute the trace of $(I-xH_u^2)^{-1}-(I-xK_u^2)^{-1}$. We  write
$$[(I-xH_u^2)^{-1}-(I-xK_u^2)^{-1}](f)=\frac x{J(x)}(f\vert (I-xH_u^2)^{-1}u)\cdot (I-xH_u^2)^{-1} u.$$
Consequently, taking the trace, we get
$${\rm Tr}[(I-xH_u^2)^{-1}-(I-xK_u^2)^{-1}]=\frac x{J(x)}\Vert (I-xH_u^2)^{-1}u\Vert^2.$$
Since, on the one  hand
$$\Vert (I-xH_u^2)^{-1} u\Vert^2= ((I-xH_u^2)^{-1}H_u^2(1)\vert 1)=J'(x)$$
and on the other hand
\begin{eqnarray*}
{\rm Tr}[(I-xH_u^2)^{-1}-(I-xK_u^2)^{-1}]&=&x{\rm Tr}[H_u^2(I-xH_u^2)^{-1}-K_u^2(I-xK_u^2)^{-1}]\\
&=&x \sum _{\rho\in \Sigma_H(u)}\left ( \frac{\rho^2}{1-\rho^2x}-\frac{\sigma(\rho)^2}{1-\sigma(\rho)^2x}\right )\ ,\end{eqnarray*}
where  we used  lemmas \ref{rigidity} and \ref{rigidity2}.
 On the other hand,
 \begin{equation}\label{traceformula}
  \sum _{\rho\in\Sigma_H(u)}\left ( \frac{\rho^2}{1-\rho^2x}-\frac{\sigma(\rho)^2}{1-\sigma(\rho)^2x}\right )=\frac{J'(x)}{J(x)}\ ,\ x\notin \left \{ \frac 1{\rho^2}, \frac 1{\sigma(\rho)^2} \right \}_{\rho\in\Sigma_H(u)}\ .
 \end{equation}

 This gives equality (\ref{J}) for $H_u^2$ and $K_u^2$ of trace class. To extend this formula to compact operators, we remark that $\sum _{\rho\in\Sigma_H(u)}(\rho^2-\sigma(\rho)^2)$ is summable since, in case of infinite sequences, we have $0\le \rho^2-\sigma(\rho)^2\le \rho^2-\rho(\sigma(\rho))^2$ so that it leads to a telescopic serie  which converges since $(\rho^2)$ tends to zero from the compactness of $H_u^2$. Hence the infinite product in formula (\ref{J}) and the above computation makes sense for compact operators.

On the other hand, for $x\notin \{\frac 1{\rho^2}\}$, 
\begin{eqnarray*}
J(x)&=&((I-xH_u^2)^{-1}(1)\vert 1)=1+x((I-xH_u^2)^{-1}(u)\vert u)\\
&=&1+x(\sum_\rho( I-xH_u^2)^{-1}(u_\rho)\vert u))=1+x\sum_{\rho\in\Sigma_H(u)} \frac{\Vert u_\rho\Vert^2}{1-x\rho^2}
\end{eqnarray*}
hence
\begin{equation}
\prod_{\rho\in\Sigma_H(u)} \frac{1-x\sigma(\rho)^2}{1-x\rho^2}=1+x\sum_{\rho\in\Sigma_H(u)}\frac{\Vert u_\rho\Vert^2}{1-x\rho^2}\ .
\end{equation}
Passing to the limit as $x$ goes to $-\infty $ in (\ref{J(x)Appendix}), we obtain (\ref{sommenu}). If we assume that
the left hand side of (\ref{sommenu}) cancels, then (\ref{J(x)Appendix}) can be rewritten as
$$\prod_{\rho\in\Sigma_H(u)}\frac{1-x\sigma(\rho)^2}{1-x\rho^2}=\sum_{\rho\in\Sigma_H(u)}\frac{\Vert u_\rho\Vert^2}{\rh ^2(1-x\rho^2)}.$$
Multiplying by $x$ and passing to the limit as $x$ goes to $-\infty $ in this new identity, we obtain (\ref{sommenurho}).
Multiplying  both terms of (\ref{J(x)Appendix}) by $(1-x\rho^2)$ and letting $x$ go to $1/\rho^2$. We get
$$\Vert u_\rho\Vert^2= \left(\rho^2-\sigma(\rho)^2\right)\prod_{\rho'\ne \rho}\frac{\rho^2-\sigma(\rho')^2}{\rho^2-{\rho'}^2} .$$
For Equality (\ref{1/J(x)}), we do almost the same analysis. First, we establish as above that
$$\frac{1}{J(x)}=1-x((I-xK_u^2)^{-1}(u)\vert u)=1-x\sum_{\sigma\in\Sigma_K(u)} \frac{\Vert u'_\sigma\Vert^2}{1-x\sigma^2}.$$

Identifying this expression with $$\frac 1{J(x)}=\prod_\rho\frac{1-x\rho^2}{1-x\sigma(\rho)^2}$$ we get, for $\rho\in\Sigma_H(u)$
$$\Vert u'_{\sigma(\rho)}(\rho)\Vert^2=(\rho^2-\sigma(\rho)^2)\prod_{\rho'\ne \rho}\frac{\sigma(\rho)^2-{\rh'}^2}{\sigma(\rho)^2-\sigma(\rho')^2}.$$
If $\sigma=0\in\Sigma_K(u)$, identifying the term in $x$ in both terms, one obtains $\Vert u'_0\Vert\ne 0$ if and only if ${\displaystyle \prod _{\rho\in \Sigma_H(u)}\frac{\si(\rho) ^2}{\rh ^2}=0}$, and in that case, identifying the corresponding factor of $x$ gives
$$\Vert u'_0\Vert^2=\rh _{\rm max}^2\prod _{\sigma\in \Sigma_K(u)\setminus\{0\}}\frac{\rh (\sigma)^2}{\si ^2}.$$

\end{proof}

As a consequence of the previous lemma, we get the following couple of corollaries.
\begin{corollary}For any $\rho,\rho'\in\Sigma_H(u)$, $\sigma,\sigma'\in \Sigma_K(u)\setminus\{0\}$, we have
 \begin{equation}\label{sommesimpletau}\sum_{\rho\in\Sigma_H(u)}\frac{\Vert u_{\rho}\Vert^2}{{\rho}^2-\sigma^2}=1\end{equation}
  \begin{equation}\label{sommesimplekappa}\sum_{\sigma \in\sigma_K(u)} \frac{\Vert u'_\sigma\Vert ^2}{\rho^2-\sigma^2}=1\end{equation}
 \begin{equation}\label{sommedoubletau}\sum_{\rho\in\Sigma_H(u)} \frac{\Vert u_{\rho}\Vert^2}{(\rho^2-\sigma^2)(\rho^2-{\sigma'}^2)}=\frac 1{\Vert u'_\sigma\Vert^2}\delta_{\sigma \sigma'}\end{equation}
  \begin{equation}\label{sommedoublekappa}\sum_ {\sigma \in\Sigma_K(u)}\frac{\Vert u'_\sigma\Vert^2}{(\sigma^2-\rho^2)(\sigma^2-{\rho'}^2)}=\frac {1}{\Vert u_\rho\Vert^2}\delta_{\rho\rho'}\end{equation}
\end{corollary}

\begin{proof}
The  first two equalities (\ref{sommesimpletau}) and (\ref{sommesimplekappa}) are obtained by making $x=\frac 1{\sigma^2}$ and $x=\frac 1{\rho^2}$ respectively  in formula (\ref{J(x)Appendix}) and formula (\ref{1/J(x)Appendix}).
For equality (\ref{sommedoubletau}) in the case $\sigma=\sigma'$, we first make the change of variable $y=1/x$ in formula (\ref{J(x)Appendix}) then differentiate both sides with respect to $y$ and make $y=\sigma^2$. 
Equality (\ref{sommedoublekappa}) in the case $\rho=\rho'$ follows by differentiating equation (\ref{1/J(x)Appendix}) and making $x=\frac 1{\rho^2}$. 
Both equalities in the other cases follow directly respectively from equality (\ref{sommesimpletau}) and equality (\ref{sommesimplekappa}).
\end{proof}

\begin{corollary}\label{kernel}
The kernel of $H_u$ is $\{ 0\} $ if and only if
$$\prod _{\rho\in \Sigma_H(u)}\frac{\sigma(\rho)^2}{\rh ^2}=0\ ,\ \prod _{\sigma\in\Sigma_K(u)}\frac{\sigma^2}{\rh(\sigma)^2}=\infty \ .$$
\end{corollary}
\begin{proof}
By the first part of theorem 4 in \cite{GG3} --- which is independent of multiplicity assumptions--- ,  the kernel of $H_u$ is $\{ 0\} $ if and only if $1\in  \overline{\rm Ran}(H_u)\setminus {{\rm Ran}}(H_u)$, where ${\rm Ran}(H_u)$
 denotes the range of $H_u$. On the other hand,
 $$u=\sum _{\rho\in \Sigma_H(u)} u_\rho=\sum _{\rho\in \Sigma_H(u)}\frac{H_u(H_u(u_\rho))}{\rh ^2}\ ,$$
 hence the orthogonal projection of $1$ onto $\overline{{\rm Ran}}(H_u)$ is
 $$\sum _{\rho\in \Sigma_H(u)}\frac{H_u(u_\rho)}{\rh ^2}\ .$$
 Consequently, $1 \in \overline{{\rm Ran}}(H_u)$ if and only if
 $$1=\sum _{\rho\in \Sigma_H(u)} \left \Vert \frac{H_u(u_\rho)}{\rh ^2}\right \Vert ^2=\sum _{\rho\in \Sigma_H(u)}\frac{\Vert u_\rho\Vert ^2}{\rh ^2}\ .$$
 Moreover, if this is the case,
 $$1=\sum _{\rho\in \Sigma_H(u)} \frac{H_u(u_\rho)}{\rh ^2}$$
 and $1\in {{\rm Ran}}(H_u)$ if and only if the series $$\sum _{\rho\in \Sigma_H(u)} \frac{u_\rho}{\rh^2}$$ converges, which is equivalent to
 $$\sum _{\rho\in \Sigma_H(u)} \frac{\Vert u_\rho\Vert^2}{\rh ^4}<\infty \ .$$
 Hence $1\in \overline{{\rm Ran}}(H_u) \setminus{{\rm Ran}}(H_u)$ if and only if
 $$\sum _{\rho\in \Sigma_H(u)}\frac{\Vert u_\rho\Vert^2}{\rh ^2}=1\ ,\ \sum _{\rho\in \Sigma_H(u)} \frac{\Vert u_\rho\Vert^2}{\rh ^4}=\infty \ ,$$
 which is the claim, in view of identities (\ref{sommenu}) and (\ref{sommenurho}).
\end{proof}

\section{Finite Blaschke products}\index{Blaschke product}

In this section, we recall the definition and the structure of finite Blaschke products which  appear in the spectral analysis of Hankel operators. A finite Blaschke product \index{Blaschke product}is an inner function of the form
$$\Psi (z)=\expo_{-i\psi}\prod _{j=1}^k \chi _{p_j}(z)\ ,\  \psi \in \T \ ,\ p_j\in \D\ ,\ \chi _p(z):=\frac{z-p}{1-\overline pz}\ ,\ p\in \D \ .$$
The integer $k$ is called the degree of $\Psi $. We denote by $\mathcal B_k$ the set of Blaschke products of degree $k$.\\

Alternatively, $\Psi\in\mathcal B_k$ can be written as 
$$\Psi (z)=\expo_{-i\psi}\frac{P(z)}{z^k\overline P\left (\frac 1z\right )}\ ,$$
where $\psi \in \T $ is called the angle of $\Psi $ and $P$ is a monic polynomial of degree $k$ with all its roots in $\D $.  Such polynomials are called Schur polynomials. We denote by $\mathcal O_d$ the open subset of $\C ^d$ made of $(a_1,\dots ,a_d)$ such that $P(z)=z^d+a_1z^{d-1}+\dots +a_d$ is a Schur polynomial. The following result implies that $\mathcal B_k$ is diffeomorphic to $\T \times \R ^{2k}$. It is connected to the Schur--Cohn criterion \cite{Sc}, \cite{Co}, and is classical  in control theory, see 
{\it e.g.} \cite{He} and references therein. For the sake of completeness, we give a self contained proof.

\begin{proposition}\label{Od}
For every $d\ge 1$ and $(a_1,\dots ,a_d)\in \C ^d$, the following two assertions are equivalent.
\begin{enumerate}
\item $(a_1,\dots ,a_d)\in \mathcal O_d$\ .
\item $\vert a_d\vert <1$ and $$\left (\frac{a_k-a_d\overline a_{d-k}}{1-\vert a_d\vert ^2}\right )_{1\le k\le d-1}\in \mathcal O_{d-1}\ .$$
\end{enumerate}
In particular, for every $d\ge 0$, $\mathcal O_d$ is diffeomorphic to $\R ^{2d}$.
\end{proposition}
\begin{proof}
Consider the rational functions
$$\chi (z)=\frac{z^d+a_1z^{d-1}+\dots +a_d}{1+\overline a_1z+\dots +\overline a_dz^d}\ ,$$
and
$$\tilde \chi (z)=\frac{\chi (z)-\chi (0)}{1-\overline {\chi (0)}\chi (z)}=z\frac{z^{d-1}+b_1z^{d-2}+\dots +b_{d-1}}{1+\overline b_1z+\dots +\overline b_{d-1}z^{d-1}}\ ,\ b_k:=\frac{a_k-a_d\overline a_{d-k}}{1-\vert a_d\vert ^2}\ .$$
If (1) holds true, then $\chi \in \mathcal B_d$, which implies 
\begin{equation}\label{chi}
\forall z\in \D, \vert \chi (z)\vert <1 \ ,\ \vert \chi (\expo_{ix})\vert =1\ .\ 
\end{equation}
In particular, $\chi (0)=a_d\in \D $, and therefore the numerator and the denominator of $\tilde \chi $ have no common root. Moreover,
\begin{equation}\label{chitilde}
\forall z\in \D, \vert \tilde \chi (z)\vert <1 \ ,\ \vert \tilde \chi (\expo_{ix})\vert =1\ .
\end{equation}
This implies $\tilde \chi \in \mathcal B_d$, hence (2). Conversely, if (2) holds, then $\tilde \chi $ satisfies (\ref{chitilde}) and has degree $d$, hence 
$$\chi (z)=\frac{\tilde \chi (z)+a_d}{1+\overline a_d\tilde \chi(z)}$$
satisfies (\ref{chi}) and has degree $d$, whence (1). 
\s
The second statement follows from an easy induction argument on $d$, since $\mathcal O_1=\D $ is diffeomorphic to $\R ^2$.
\end{proof}

\section{Two results by Adamyan--Arov--Krein}\label{sectionAAK}\index{Hankel operator}

We recall the proof of two important results  by Adamyan--Arov--Krein\index{AAK theorem}. The proof is translated from \cite{AAK} into our representation of Hankel operators, and is given for the convenience of the reader.

\begin{theo}[Adamyan, Arov, Krein \cite{AAK}] \label{AAK}
Let $u\in VMO_+(\S^1)\setminus \{ 0\} $. Denote by $(\lambda_k(u))_{k\ge 0}$ the sequence of singular values of $H_u$, namely the eigenvalues of $\vert H_u\vert :=\sqrt{H_u^2}$, in decreasing order, and repeated according to their multiplicity. Let $k\ge 0, m\ge 1,$ such that
$$\lambda_{k-1}(u)>\lambda_k(u)=\dots =\lambda_{k+m-1}(u)=s>\lambda_{k+m}(u)\ ,$$
with the convention $\lambda_{-1}(u):=+\infty $. 
\begin{enumerate}
\item For every $h\in E_u(s)\setminus \{ 0\} $, there exists a polynomial $P\in \C _{m-1}[z]$ such that
$$\forall z\in \D\ ,\ \frac{sh(z)}{H_u(h)(z)}=\frac{P(z)}{z^{m-1}\overline P\left (\frac 1z\right )}\ .$$
\item There exists a rational function $r$ with no pole on $\overline \D$ such that ${\rm rk} (H_r)=k$ and
$$\Vert H_u-H_r\Vert =s\ .$$
\end{enumerate}
\end{theo}
\begin{proof}
We start with the case $k=0$. In this case the statement (2) is trivial, so we just have to prove (1). This is a consequence of the following lemma.
\begin{lemma}\label{tops}
Assume $s=\Vert H_u\Vert $. For every $h\in E_u(s)\setminus \{ 0\} $, consider the following inner outer decompositions,
$$h=ah_0\ ,\ s^{-1}H_u(h)=bf_0\ .$$
If $c$ is an arbitrary inner divisor of $ab$,  $ab=cc'$, then $ch_0\in E_u(s)$, with
\begin{equation} \label{ch0}
H_u(ch_0)=sc'f_0\ ,\ H_u(c'f_0)=sch_0\ .
\end{equation}
 In particular, $a, b$ are finite Blaschke products and 
\begin{equation}\label{abBla}
\deg (a)+\deg (b) +1\le \dim E_u(s)\ .
\end{equation}
Furthermore, there exists an outer function $h_0$ such that, if $m:=\dim E_u(s)$,
\begin{equation}\label{Eunorm}
E_u(s)=\C _{m-1}[z] h_0\ ,
\end{equation}
and there exists $\varphi \in \T $ such that, for every $P\in \C _{m-1}[z]$, 
\begin{equation}\label{Hunorm}
H_u(Ph_0)(z)=s\expo _{i\varphi}z^{m-1}\overline P\left (\frac 1z\right )h_0(z)\ .
\end{equation}
\end{lemma}
Let us prove this lemma. We need a number of elementary properties of Toeplitz operators $T_b$ defined by equation (\ref{Toeplitz}), 
where $b$ is a function in $L^\infty _+:=L^2_+\cap L^\infty $, which we recall below. In what follows, $u\in BMO_+$.
\begin{enumerate}
\item 
$$H_uT_b=T_{\overline b}H_u=H_{T_{\overline b}u}\ .$$
\item If $\vert b\vert \le 1$ on $\S ^1$, 
$$H_u^2\ge T_{\overline b}H_u^2T_b\ .$$
\item If $\vert b\vert =1$ on $\S ^1$, namely $b$ is an inner function,  
$$\forall f\in L^2_+\ ,\ f=T_bT_{\overline b}f  \Longleftrightarrow \Vert f\Vert =\Vert T_{\overline b}f\Vert \ .$$
\end{enumerate}
Indeed, (1) is just equivalent to the elementary identities
$$\Pi (u\overline b\overline h)=\Pi (\overline b\Pi (u\overline h))=\Pi (\Pi (\overline bu)\overline h)\ .$$
As for (2),
we observe that $T_b^*=T_{\overline b}$ and 
$$\Vert T_{\overline b}h\Vert \le \Vert \overline bh\Vert \le \Vert h\Vert \ .$$
Hence, using (1),
$$(H_u^2h\vert h)-(T_{\overline b}H_u^2T_bh\vert h)=\Vert H_u(h)\Vert ^2-\Vert T_{\overline b}H_u(h)\Vert ^2\ge 0\ .$$
Finally, for (3) we remark that, if $b$ is inner,  $T_{\overline b}T_b=I$ and $T_bT_{\overline b}$ is the orthogonal projector onto the range of $T_b$, namely $bL^2_+$.  Since $\Vert T_{\overline b}f\Vert =\Vert T_bT_{\overline b}f\Vert $, (3) follows.
\s
Let us come back to the proof of Lemma \ref{tops}. Starting from
$$H_u(h)=sf\ ,\ H_u(f)=sh\ ,\ h=ah_0\ ,\ f=bf_0\ ,\ ab =cc'\ ,$$
we obtain, using property (1),
$$T_{\overline c'}H_u(ch_0)= H_u(cc'h_0)=T_{\overline b}H_u(h)=sf_0\ .$$
In particular, 
$$\Vert H_u(ch_0)\Vert \ge \Vert T_{\overline c'}H_u(ch_0)\Vert =s\Vert f_0\Vert =s\Vert f\Vert =s\Vert h\Vert =s\Vert ch_0\Vert \ .$$
Since $s=\Vert H_u\Vert$, all the above inequalities are equalities, hence $ch_0\in E_u(s)$, and, using (3), 
$$H_u(ch_0)=T_{c'}T_{\overline c'}H_u(ch_0)=sc'f_0\ .$$
The second identity in (\ref{ch0}) immediately follows.  Remark that, if $\Psi $ is an inner function of degree at least $d$, there exist $d+1$ linearly independent inner divisors of $\Psi $ in $L^\infty _+$. Then inequality (\ref{abBla}) follows. Let us come to the last part. Since $\dim E_u(s)=m$, there exists $h\in E_u(s)\setminus \{ 0\}$ such that the first $m-1$ Fourier coefficients of $h$ cancel, namely 
$$h=z^{m-1}\tilde h\ .$$
Considering the inner outer decompositions
$$\tilde h=ah_0\ ,\ H_u(h)=s bf_0\ ,$$
and applying the first part of the lemma, we conclude that $\deg (a)+\deg (b)=0$, hence, up to a slight change of notation, $a=b=1$,
and, for $\ell =0,1,\dots ,m-1$,
$$H_u(z^\ell h_0)=sz^{m-1-\ell}f_0\ ,\ H_u(z^{m-\ell-1} f_0)=sz^{\ell}h_0\ .$$
This implies
$$E_u(s)=\C _{m-1}[z] h_0=\C _{m-1}[z]f_0\ .$$
Since $\Vert h_0\Vert =\Vert h\Vert =\Vert f\Vert =\Vert f_0\Vert $, it follows that $f_0=\expo_{i\varphi} h_0$, and (\ref{Hunorm}) follows from the antilinearity of $H_u$. The proof of Lemma \ref{tops} is complete.
\s
Let us complete the proof of the theorem by proving the case $s<\Vert H_u\Vert $. The crucial  observation is the following.
\begin{lemma}\label{unimod}
There exists a function $\phi \in L^\infty $ such that $\vert \phi \vert =1$ on $\S ^1$, and such that the operators $H_u$ and $H_{s\Pi (\phi )}$ coincide on $E_u(s)$.
\end{lemma}
Let us prove this lemma. For every pair $(h,f)$ of elements of $E_u(s)$ such that  $H_u(h)=sf, H_u(f)=sh$, we claim that the function
$$\phi :=\frac{f}{\overline h}\ ,$$
 does not depend on the choice of the pair $(h,f)$. Indeed, it is enough to check that, if $(h',f')$ is another such pair,
 \begin{equation}\label{modulePaireSchmidt}
 f\overline h'=f'\overline h\ .
 \end{equation}
 In fact, for every $n\ge 0$,
 \beno s(f\overline h'\vert z^n)&=&(H_u(h)\vert S^nh')=((S^*)^nH_u(h)\vert h')\\
 &=&(H_u(S^nh)\vert h')=(H_u(h')\vert S^nh)=s(f'\overline h\vert z^n)\ .
 \eeno
 Changing the role of $(h,h')$ and $(f,f')$, we get the claim. Finally the fact $\vert \phi \vert =1$ comes from applying the above identity to the pairs $(h,f)$ and $(f,h)$. Then we just have to check that, for every such pair,
 $$H_{s\Pi (\phi )}(h)=s\Pi (\Pi (\phi )\overline h)=s\Pi (\phi \overline h)=sf \ .$$
 This completes the proof of Lemma \ref{unimod}. 
 \s
 Let us first consider part (2) of the Theorem. Introduce 
 $$v:=s\Pi (\phi )\ .$$
 We are going to show that $r:=u-v$ is a rational function with no pole on $\overline{\D}$, ${\rm rk}(H_r)=k$ and 
 $$\Vert H_u-H_r\Vert =s\ .$$
 Since, for every $h\in L^2_+$,
 $$H_v(h)=s\Pi (\phi \overline h)\ ,$$
 we infer $\Vert H_v\Vert \le s$, and  from $E_u(s)\subset E_v(s)$, we conclude 
 $$\Vert H_v\Vert =s\ .$$
 Because of (\ref{Ku}), $H_u$ and $H_v$ coincide on the smallest shift invariant closed subspace of $L^2_+$ containing $E_u(s)$.
 By Beurling's theorem \cite{B}, this subspace is $aL^2_+$ for some inner function $a$. 
 Then $H_r =0$ on $aL^2_+$, hence the rank of $H_r$ is at most the dimension of $(aL^2_+)^\perp $. 
 Since 
$$\Vert H_u-H_r\Vert =\Vert H_v\Vert =s<\lambda_{k-1}(u)\ ,$$
the rank of $H_r$ cannot be smaller than $k$, and the result will follow by proving that the dimension of $(aL^2_+)^\perp $ is  $k$.

We can summarize the above construction as
 $$H_{T_{\overline a}u}=H_uT_a=H_vT_a=H_{T_{\overline a}v}\ .$$
 The above Hankel operator is compact and its norm is at most $s$. In fact, if $H_u(h)=sf\ ,\ H_u(f)=sh$, with $f=a \tilde f$, it is clear from property (1) above that
 $$H_{T_{\overline a}u}(h)=s\tilde f\ ,\ H_{T_{\overline a}u}(\tilde f)=sh\ .$$
 In particular, 
$$\Vert  H_{T_{\overline a}u}\Vert =s\ .$$
Applying property (\ref{ch0}) from Lemma \ref{tops}, we conclude that there exists an outer function $h_0$ such that $$ch_0\in E_{T_{\overline a}u}(s)$$  for every inner divisor $c$ of $a$. Moreover, $a$ is a Blaschke product of finite degree $d$\index{Blaschke product}. Since $h_0$ is outer, it does not vanish at any point of $\D $, therefore, it is easy to find $d$ inner divisors $c_1,\dots ,c_d$ of $a$ such that $c_1h_0,\dots ,c_dh_0$ are linearly independent and generate a vector subspace $\tilde E$ satisfying
$$\tilde E\cap aL^2_+=\{ 0\} \ .$$
Consequently, we obtain
$$\tilde E\oplus E_u(s)\subset E_{T_{\overline a}u}(s),$$
whence
\begin{equation}\label{borneinfdim}
d':=\dim E_{T_{\overline a}u}(s)\ge d+m\ .
\end{equation}
On the other hand, by property (2) above, we have
$$H_{T_{\overline a}u}^2=T_{\overline a}H_u^2T_a \le H_u^2\ ,$$
therefore, from the min-max formula\index{min-max formula},
$$\forall n, \lambda_n(T_{\overline a}u)\le \lambda_n(u)\ .$$
In particular, from the definition of $d'$,
$$s=\lambda_{d'-1}(T_{\overline a}u)\le \lambda_{d'-1}(u)\ ,$$
which imposes, in view of the assumption,
$d'-1\le k+m-1$, in particular, in view of (\ref{borneinfdim}),
$$d\le k\ .$$
Finally, notice that, since $a$ has degree $d$, the dimension of  $(aL^2_+)^\perp $ is $d$.
Hence, by the min-max formula again, 
$$\lambda_d(u)\le \sup _{h\in aL^2_+\setminus \{ 0\} }\frac{\Vert H_u(h)\Vert }{\Vert h\Vert}\le s<\lambda_{k-1}(u)\ .$$
This imposes $d\ge k$, and finally
$$d=k\ ,\ d'=k+m\ ,$$
and part (2) of the theorem is proved. 

In order to prove part (1), we apply 
properties (\ref{Eunorm}) and (\ref{Hunorm}) of Lemma \ref{tops}. We describe elements of $E_{T_{\overline a}u}(s)$ as 
$$h(z)=Q(z)h_0(z)\ ,$$
where $h_0$ is outer and $Q\in \C _{k+m-1}[z]$. Moreover, if $h\in E_u(s)$, then $h=a\tilde h$, $H_u(h)=sa\tilde f$, where $\tilde h, \tilde f\in E_{T_{\overline a}u}(s)$.
This reads 
$$Q(z)=a(z)\tilde Q(z)\ ,$$
If we set
$$a(z)=\frac{z^k\overline D\left (\frac 1z\right )}{D(z)}\ ,$$
where $D\in \C _k[z]$ and $D(0)=1$, and $D$ has no zeroes in $\D $, this implies
$$Q(z)=z^k\overline D\left (\frac 1z\right )P(z)\ ,\ P\in \C _{m-1}[z]\ .$$
Moreover, 
$$H_{T_{\overline a}u}(h)(z)=s\expo_{i\varphi }z^{m+k-1}\overline Q\left (\frac 1z\right )h_0(z)=s\expo_{i\varphi } D(z)z^{m-1}\overline P\left ( \frac 1z\right )h_0(z)\ ,$$
and
$$H_u(h)(z)=a(z)H_{T_{\overline a}u}(h)(z)=s\expo_{i\varphi }  z^k\overline D\left (\frac 1z\right )z^{m-1}\overline P\left ( \frac 1z\right )h_0(z)\ .$$
Changing $P$ into $P\expo _{-i\varphi /2}$, this proves part (1) of the theorem.
\end{proof}

\section{Multiplicity and Blaschke products} \label{Blaschke}

In this section, we prove a refinement of part $(1)$ of Theorem \ref{AAK} which allows to describe precisely the structure of the eigenspaces of a Hankel operator and its shifted. For a Blaschke product  $\Psi\in \mathcal B_k$ given by  $$\Psi(z)=\expo_{-i\psi}\frac{P(z)}{z^k\overline P\left (\frac 1z\right )},$$ we shall denote by
$$D(z)=z^k\overline P\left (\frac 1z\right )$$
the normalized denominator of $\Psi $.\\
Assume $u\in VMO_+(\S^1 )$ and $s\in \Sigma_H(u)$. Then $H_u$ acts on the finite dimensional vector space $E_u(s)$.
It turns out that this action can be completely described by an inner function. A similar fact holds for the action of $K_u$ onto $F_u(s)$ when $s\in \Sigma_K(u)$, $s\neq 0$.

\begin{proposition}\label{action}
Let $s>0$ and $u\in VMO_+(\S^1 )$.
\begin{enumerate}
\item Assume $s\in \Sigma_H(u)$ and $m:=\dim E_u(s)=\dim F_u(s)+1$. Denote by $u_s$ the orthogonal projection of $u$ onto $E_u(s)$. There exists an inner function $\Psi _s\in \mathcal B_{m-1}$ such that
\begin{equation}\label{sus}
su_s=\Psi _sH_u(u_s)\ ,
\end{equation}
and if $D$ denotes the normalized denominator of $\Psi _s$, 
\begin{eqnarray} \label{descEF1}
E_u(s)&=&\left \{ \frac fDH_u(u_s)\ ,\ f\in \C _{m-1}[z]\right \} \ ,\\
F_u(s)&=&\left\{ \frac gDH_u(u_s)\ ,\ g\in \C _{m-2}[z]\right \}, 
\end{eqnarray}
and, for $a=0,\dots,m-1\ ,\ b=0,\dots ,m-2$, 
\begin{eqnarray} \label{actionHu}
H_u\left (\frac {z^a}DH_u(u_s)\right )&=&s\expo_{-i\psi _s}\frac {z^{m-a-1}}DH_u(u_s)\ , \\
K_u\left (\frac {z^b}DH_u(u_s)\right )&=&s\expo_{-i\psi _s}\frac {z^{m-b-2}}DH_u(u_s)\ ,
\end{eqnarray}
where $\psi _s$ denotes the angle of $\Psi _s$.
\item Assume $s\in \Sigma_K(u)$ and $\ell :=\dim F_u(s)=\dim E_u(s)+1$. Denote by $u_s'$ the orthogonal projection of $u$ onto $F_u(s)$. There exists an inner function $\Psi _s\in \mathcal B_{\ell -1}$ such that
\begin{equation}\label{Kus}
K_u(u_s')=s\Psi _su_s'\ ,
\end{equation}
and if $D$ denotes the normalized denominator of $\Psi _s$, 
\begin{eqnarray} \label{descFE2}
F_u(s)&=&\left \{ \frac fDu_s'\ ,\ f\in \C _{\ell -1}[z]\right \} \ ,\\
E_u(s)&=&\left \{ \frac {zg}Du_s'\ ,\ g\in \C _{\ell -2}[z]\right \}, 
\end{eqnarray}
and, for $a=0,\dots,\ell -1\ ,\ b=0,\dots ,\ell-2$, 
\begin{eqnarray} \label{actionKu}
K_u\left (\frac {z^a}Du_s'\right )&=&s\expo_{-i\psi _s}\frac {z^{\ell -a-1}}Du_s'\ , \\
H_u\left (\frac {z^{b+1}}Du_s'\right )&=&s\expo_{-i\psi _s}\frac {z^{\ell-b-1}}Du_s'\ ,
\end{eqnarray}
where $\psi _s$ denotes the angle of $\Psi _s$.
\end{enumerate}
\end{proposition}

\subsection{Case of $\rho \in \Sigma_H(u)$} 
Let $u\in VMO_+(\S^1)$.
Assume  that $\rho\in \Sigma_H(u)$  and $m:=\dim E_u(\rho)$. Recall  that $$E_u(\rho)=\ker (H_u^2-\rho^2I),\; F_u(\rho)=\ker(K_u^2-\rho^2I)$$and by Lemma \ref{EF}, $$F_u(\rho)=E_u(\rho)\cap u^\perp \ .$$ 
\subsubsection{Definition of $\Psi _\rho$.}
By definition $(\rho u_\rho,H_u(u_\rho)$ and $(H_u(u_\rho),\rho u_\rho)$ are Schmidt-pair for $H_u$. Hence, applying  equation (\ref{modulePaireSchmidt}) to $(u_\rho,H_u(u_\rho))$ and $(H_u(u_\rho),u_\rho)$, we obtain that, at every point of $\S^1$, 
$$\vert H_u(u_\rho)\vert ^2=\rho^2\vert u_\rho\vert ^2\ .$$

We  thus define
$$\Psi_\rho := \frac{\rho u_\rho}{H_u(u_\rho)}\ .$$
\subsubsection{The function $\Psi_\rho$ is an inner function.}
We know that $\Psi _\rho$ is of modulus $1$ at every point of $\S^1 $. Let  us show that $\Psi_\rho $ is in fact an inner function. By part 
$(1)$ of the Adamyan--Arov--Krein theorem, we already know that $\Psi _\rho$ is a rational function with no poles on the unit circle.
Therefore, it is enough to prove that $\Psi_\rho $ has no pole in the open unit disc. 
Assume that $q\in \D$ is a zero of $H_u(u_\rho)$, 
and let us show that $q$ is a zero of $u_\rho$ with at least the same multiplicity. 

Assume $H_u(u_\rho)(q)=0$ for some $q$, $|q|<1$. We prove that $u_\rho(q)=0$ (in fact $h(q)=0$ for any $h$ in $E_u(\rho)$).
Let us consider the continuous linear form 
$$\begin{array}{ccc} E_u(\rho)&\to &\C\\
h&\mapsto &h(q)\end{array}.$$ By the Riesz representation Theorem, there exists $f\in E_u(\rho)$ such that $h(q)=(h\vert f)$. Our aim is to prove that, if $H_u(u_\rho)(q)=0$ then $f=0$.
The assumption $H_u(u_\rho)(q)=0$ reads
$$0= (H_u(u_\rho)\vert f)= (H_u(f)\vert u_\rho)=(H_u(f)\vert u)=(1\vert H_u^2(f))=\rho^2(1\vert f).$$ Here we used the fact that, as $H_u(f)\in E_u(\rho)$ for any $f\in E_u(\rho)$, $$(H_u(f)\vert u_\rho)=(H_u(f)\vert u).$$ It implies that $H_u(f)\in u^\perp\cap E_u(\rho)=F_u(\rho)$, that $S^*f=\frac 1{\rho^2}S^*H_u^2f=\frac 1{\rho^2}K_uH_u(f)$ belongs to $F_u(\rho)$ and also $f=SS^*f$. Hence, 
$$\Vert f\Vert^2=(f\vert f)=f(q)=qS^*f(q)=q(S^*f\vert f)\le |q|\Vert S^*f\Vert \cdot \Vert f\Vert$$ 
 which is impossible except if $f=0$.
In particular $u_\rho(q)=0$. 
\s
A similar argument allows to show that if $q$ is a zero of order $r$ of $H_u(u_\rho)$, that is, if, for every $a\leq r-1$,
$$(H_u(u_\rho))^{(a)} (q)=0\ $$ then
the same holds for $u_\rho$ that is
 $u_\rho^{(a)}(q)=0$ for every $a\le r-1$.
\\

As a consequence of  the Adamyan--Arov--Krein theorem \ref{AAK}, we get that $\Psi_\rho $ is a Blaschke product of degree $m-1$, $m=\dim E_u(\rho)$.
\subsubsection{Description of $E_u(s)$, $F_u(s)$ and of the action of $H_u$ and of $K_u$ on them}

We start with proving the following lemma.
\begin{lemma}\label{crucialHuGeneral}
Let $f\in \mathbb H^\infty (\D )$ such that $\Pi (\Psi _\rho\overline f)=\Psi  _\rho\overline f\ .$
Then
$$H_u(fH_u(u _\rho))=\rho\Psi  _\rho\overline fH_u(u _\rho)\ .$$
\end{lemma}
The proof of the lemma is straightforward,
\begin{eqnarray*}
H_u(fH_u(u _\rho))&=&\Pi (u\overline f \overline {H_u(u _\rho)})=\Pi (\overline fH_u^2(u _\rho))=\rho^2\Pi (\overline fu _\rho)=\rho\Pi (\overline f \Psi _\rho H_u(u _\rho))\\
&=&\rho\Psi _\rho \overline fH_u(u _\rho)\ .
\end{eqnarray*}
Applying this lemma, we get, for any $0\le a\le m-1$,
 $$H_u\left(\frac{z^a }{D}H_u(u _\rho)\right)=\rho{\rm e}^{-i\psi}\frac{z^{m-1-a}}{D}H_u(u _\rho)\ .$$
Consequently, 
$\displaystyle E_u(\rho)=\frac{\C_{m-1}[z]}DH_u(u _\rho)$ and that the action of $H_u$ on $E_u(\rho)$ is as expected in Equation (\ref{actionHu}). It remains to prove that $$F_u(\rho)=\frac{\C_{m-2}[z]}DH_u(u _\rho) $$ and that the action of $K_u$ is described as in (\ref{actionHu}).
We have, for $0\le b\le m-2$,
\begin{eqnarray*}
K_u\left( \frac{z^b}D H_u(u _\rho)\right)&=&H_uS\left( \frac{z^b}D H_u(u _\rho)\right)= H_u\left( \frac{z^{b+1}}D H_u(u _\rho)\right)\\
&=&\rho{\rm e}^{-i\psi}\frac{z^{m-2-b}}{D} H_u(u _\rho)
\end{eqnarray*}
In particular, it proves that $\frac{\C_{m-2}[z]}DH_u(u _\rho)\subset F_u(\rho)$. As the dimension of $F_u(\rho)$ is $m-1$ by assumption, we get the equality.

\subsection{Case of $\sigma\in \Sigma_K(u)$}
The second part of the proposition, concerning the case of $\sigma\in \Sigma_K(u)$, can be proved similarly. The first step is to prove that $$\Psi_\sigma:=\frac{K_u(u'_\sigma)}{\sigma u'_\sigma}$$ is an inner function. The same argument as the one used above for the $K_u$ Schmidt pairs \index{Schmidt pairs} $(K_u(u'_\sigma),\sigma u'_\sigma)$ and $(\sigma u'_\sigma,K_u(u'_\sigma))$  gives that $\Psi_\sigma$ has modulus one. To prove that it is an inner function, we argue as before. Namely, using again part $(1)$ of the Adamyan-Arov--Krein theorem \ref{AAK}, for $S^*u$ in place of $u$, we prove that if $u'_\sigma$ vanishes at some $q\in \D$, $K_u(u'_\sigma)$ also vanishes at $q$ at the same order. 

Assume $u'_\sigma(q)=0$. As before, there exists $f\in F_u(\sigma)$ satisfying $h(q)=(h\vert f)$ for any $h\in F_u(\sigma)$. The assumption on $u'_\sigma$ reads $0=(u'_\sigma\vert f)=(u\vert f)$ hence $f$ belongs to $E_u(\sigma)=u^\perp\cap F_u(\sigma)$. The same holds for $H_u(f)$ and, in particular,  $0=(H_u(f)\vert u)=(1\vert H_u^2(f))=\sigma^2(1\vert f)$. Therefore, $f=SS^*f$ and $S^*f=K_uH_u(f)\in F_u(\sigma)$. One concludes as before since
$$\Vert f\Vert^2=f(q)=qS^*f(q)=q(S^*f\vert f)$$ which implies $f=0$.

One proves as well that if $q$ is a zero of order $r$ of $(u'_\sigma)$,  $K_u(u'_\sigma)$ vanishes at $q$ with the same order. Eventually, by theorem \ref{AAK}, $\Psi_\sigma\in\mathcal B_{\ell-1}.$ 
It remains to describe $E_u(\sigma)$ and $F_u(\sigma)$ and the action of $H_u$ and $K_u$ on them.
We start with a lemma analogous to Lemma \ref{crucialHuGeneral}.
\begin{lemma}\label{crucialKuGeneral}
Let $f\in \mathbb H^\infty (\D )$ such that $\Pi (\Psi_\sigma \overline f)=\Psi_\sigma \overline f$.Then
$$K_u(fu'_\sigma)=\sigma\Psi_\sigma \overline fu'_\sigma\ .$$
\end{lemma}
The proof of the lemma is similar to the one of Lemma \ref{crucialHuGeneral}.
Write
$$\Psi _\sigma(z)={\rm e}^{-i\psi} \frac{z^{\ell-1}\overline D\left(\frac 1z\right)}{D(z)},$$ where $D$ is some normalized polynomial of degree $k$.
From the above lemma, for $0\le a\le \ell-1$,
\begin{eqnarray*}
K_u\left(\frac{z^a}{D}u'_\sigma\right)=\sigma{\rm e}^{-i\psi}\frac{z^{\ell-1-a}}Du'_\sigma\ .
\end{eqnarray*}
%
%
%

In order to complete the proof, we just need to describe $E_u(\sigma)$ as the subspace of $F_u(\sigma)$ made with functions which vanish at $z=0$, or equivalently are orthogonal to $1$. We already know that vectors of $E_u(\sigma)$ are orthogonal to $u$, and that $H_u$ is a bijection from $E_u(\sigma)$ onto $E_u(\sigma)$. We infer that vectors of $E_u(\sigma)$ are orthogonal to $1$. A dimension argument allows to conclude.

\chapter{The inverse spectral theorem}\label{chapter inverse spectral}

We  now come to the analysis of the non-linear Fourier transform\index{Non linear Fourier transform} introduced in Theorem \ref{FT}.  
Given $u\in H^{1/2}_+(\S^1)\setminus \{ 0\} $, one can define, according to Lemma \ref{rigidity}, a finite or infinite sequence $s=(s_1>s_2>\dots )$ such that
\begin{enumerate}
\item The $s_{2j-1}$'s   are the $H$-dominant singular values associated to $u$\index{singular value}.
\item The $s_{2k}$'s are the $K$-dominant singular values associated to $u$.
\end{enumerate}
For every $r\ge 1$, associate to each $s_r$ a Blaschke product $\Psi _r$ by means of Proposition \ref{action}. This defines a mapping\index{mapping $\Phi$}
$$\Phi : H^{1/2} _+(\S^1 )\setminus \{ 0\} \longrightarrow \mathcal S_\infty^{(2)}\cup\bigcup_{n=1}^\infty\mathcal S_n\ .$$
We are going to prove that $\Phi$ is bijective and coincides with the inverse mapping of the non linear Fourier transform introduced in Theorem \ref{FT}.
In other words, with the notation of Theorem \ref{FT}, $\Phi(u)= ({\bf s},{\bf \Psi})\in \mathcal S_n$ if and only if
\ben
u=u ({\bf s},{\bf \Psi})
\een
Furthermore, if $(s_r,\Psi _r)_{r\ge 1}\in \mathcal S_\infty^{(2)} $, then
its preimage $u$ by $\Phi $ is given by
\begin{equation}
u=\lim _{q\rightarrow \infty} u_q
\end{equation}
where the limit takes place in $H^{\frac 12}_+(\S ^1)$ and $$u_q:=u((s_1,\dots, s_{2q}),(\Psi_1,\dots,\Psi_{2q})).$$ 
\s
Let us briefly describe the structure of this chapter. Firstly, we concentrate on the case of finite rank operators. This part itself includes two important steps. The first step is to prove injectivity of $\Phi $ by producing an explicit formula for $u$, which coincides with the one for $u({\bf s},{\bf \Psi})$ in Theorem \ref{FT} of the Introduction. The second step establishes surjectivity by combining algebraic calculations and compactness arguments. This step includes in particular the construction of Hankel operators with arbitrary multiplicities thanks to a process based on collapsing three consecutive singular values. Section \ref{sectionHS} extends these results to Hilbert--Schmidt operators, and section \ref{compact} to compact Hankel operators.

\section{The inverse spectral theorem in the finite rank case} \label{finiterank}

In this section, we consider the case of finite rank Hankel operators. 
Consider a symbol $u$ with $H_u$ of finite rank. Then the sets $\Sigma _H(u)$ and $\Sigma _K(u)$ are finite. We set
$$q:=\vert \Sigma _H(u)\vert =\vert \Sigma _K(u)\vert \ .$$
If 
\begin{eqnarray*}
\Sigma _H(u):=\{ \rho _j, j=1,\dots ,q\},\  \rho _1>\dots >\rho _q>0\ ,\\
\ \Sigma _K(u):=\{ \sigma _j, j=1,\dots ,q\},\  \sigma _1>\dots >\sigma  _q\ge 0\ ,
\end{eqnarray*}
we know from Lemma (\ref{rigidity}) that
\begin{equation}\label{Sigmaordered}
\rho _1>\sigma _1>\rho _2>\sigma _2>\dots >\rho _q>\sigma _q\ge 0\ .
\end{equation}
We set $n:=2q$ if $\sigma _q>0$ and $n:=2q-1$ if $\sigma _q=0$. For $2j\le n$, we set
$$s_{2j-1}:=\rho _j\ ,\ s_{2j}:=\sigma _j\ ,$$
so that the positive elements in the list (\ref{Sigmaordered}) read
\begin{equation}\label{ineqs} s_1>s_2>\dots >s_n>0\ .\end{equation}
Using  Proposition \ref{action}, we define $n$ finite Blaschke products $\Psi _1,\dots ,\Psi _n$ by
$$\rho _ju_j=\Psi _{2j-1}H_u(u_j)\ ,\ K_u(u'_j)=\sigma _j\Psi _{2j}u'_j\ ,\ 2j\le n\ ,$$
where $u_j$ denotes the orthogonal projection of $u$ onto $E_u(\rho _j)$, and $u'_j$ denotes the orthogonal projection of $u$ onto $F_u(\sigma _j)$.\\
Given a finite sequence $(d_1,\dots ,d_n)$ of nonnegative integers, recall that  $\mathcal V_{(d_1,\dots ,d_n)}$  is the set of symbols $u$ such that\index{singular value}
\begin{enumerate}
\item The singular values $s_{2j-1}$, $j\ge 1$, of $H_u$ in $\Sigma_H(u)$, ordered decreasingly, have respective multiplicities 
$$d_1+1,d_3+1,\dots \ .$$
\item The singular values $s_{2j}$, $j\ge 1$, of $K_u$ in $\Sigma_K(u)$, ordered decreasingly, have respective multiplicities 
$$d_2+1,d_4+1,\dots \ .$$
\end{enumerate}
 Our goal is to prove the following statement.
\begin{theorem}\label{mainfiniterank} The mapping \index{$\Phi_{d_1,\dots ,d_n}$}
$$\begin{array}{rcl}\Phi_{d_1,\dots ,d_n}:\mathcal V_{d_1,\dots , d_n}&\longrightarrow & \mathcal S_{d_1,\dots ,d_n}\\
u&\longmapsto& \left((s_r)_{1\le r\le n},(\Psi_r)_{1\le r\le n}\right)\end{array}$$ is a homeomorphism and  $\Phi_{d_1,\dots ,d_n}(u)= ({\bf s},{\bf \Psi})$ is equivalent to  
$$u=u ({\bf s},{\bf \Psi})$$ with the formula given by \ref{luminy}.\\
Moreover, the mapping $\Phi _{d_1,\dots ,d_n}^{-1}: \mathcal S_{d_1,\dots ,d_n} \rightarrow \mathcal V(d)$, where $d$ is given by (\ref{inclusionV}),  is a smooth immersion. 
As a consequence, the set 
$\mathcal V_{(d_1,\dots ,d_n)}$
is a submanifold of $\mathcal V(d)$ of dimension $$\dim \mathcal V_{(d_1,\dots ,d_n)}=2n+2\sum _{r=1}^nd_r\ .$$ 
\end{theorem}
\begin{proof}
The proof of Theorem \ref{mainfiniterank} involves several steps. Firstly,  we prove the continuity of $\Phi _{d_1,\dots ,d_n}$, and more precisely that every point of $ \mathcal V_{(d_1,\dots ,d_n)}$ has an open neihgborhood in
$\mathcal V(d)$, where $d$ is given by (\ref{inclusionV}), on which $\Phi _{d_1,\dots ,d_n}$ can be extended as a smooth mapping $\tilde \Phi _n$ into some finite dimensional manifold including $\mathcal S_{d_1,\dots ,d_n}$. Then we prove that $\Phi _{d_1,\dots ,d_n}$ is a homeomorphism in the case $n$ even, along the following lines :
\begin{itemize}
\item $\Phi _{d_1,\dots ,d_n}$  is injective, with the formula $$\Phi_{d_1,\dots ,d_n}(u)= ({\bf s},{\bf \Psi})\Rightarrow u=u({\bf s},{\bf \Psi}) .$$
\item $\Phi _{d_1,\dots ,d_n}$ is an open mapping.
\item $\Phi _{d_1,\dots ,d_n}$ is a proper mapping.
\item $\mathcal V_{(d_1,\dots ,d_n)}$ is not empty. 
\end{itemize}
Since the target space $\mathcal S_{(d_1,\dots,d_n)}$ is connected, these four items trivially lead to 
the homeomorphism result. The case $n$ odd is derived from a simple limiting argument.
Then the right inverse formula proved in the first item implies that, for every $ ({\bf s},{\bf \Psi})\in \mathcal S_{d_1,\dots ,d_n}$, $u({\bf s}, {\bf \Psi})\in  \mathcal V_{(d_1,\dots ,d_n)}$ and 
$$\Phi _{d_1,\dots ,d_n}(u({\bf s},{\bf \Psi }))=({\bf s},{\bf \Psi}).$$
The use of the smooth mapping $\tilde \Phi _n$ combined to the smoothness of the mapping  $({\bf s}, {\bf \Psi})\mapsto u({\bf s}, {\bf \Psi})$ will then complete the proof.
 \s

\subsection{Continuity of $\Phi _{d_1,\dots ,d_n}$ and local smooth extension}\label{continuity}

Fix $u_0\in \mathcal V_{(d_1,\dots ,d_n)}$. 
Let $\rho\in \Sigma_H(u_0)$. The orthogonal projector $P_\rho$ on the eigenspace $E_{u_0}(\rho )$ is given by 
$$P_\rho=\int_{\mathscr C_\rho}(zI-H_{u_0}^2)^{-1} \frac{dz}{2i\pi}$$ where $\mathscr C_\rho$ is a circle, centered at $\rho^2$ whose radius is small enough so that the closed disc $\overline D_\rho $ delimited by  $\mathscr C_\rho$   is at positive distance to the rest of the spectrum  of $H_{u_0}^2$. For $u$ in a neighborhood $V_0$  of $u_0$ in $H^{1/2}_+$, $C_\rho $ does not meet the spectrum of $H_u^2$, and one may consider
$$P_\rho^{(u)}:=\int_{\mathscr C_\rho}(zI-H_{u}^2)^{-1} \frac{dz}{2i\pi}$$ which is  a finite rank orthogonal projector smoothly dependent on  $u$. Hence, $P_\rho^{(u)}(u) $ is well defined and smooth. Since this vector is not zero for $u=u_0$, it is still not zero for every $u$ in $V_0$. 
\s
We can do the same construction with any  $\sigma\in \Sigma_K(u_0)\setminus \{ 0\} $. We have therefore constructed $n$ smooth functions $u\in V_0\mapsto P_r^{(u)}\ ,\ r=1,\dots ,n,$ valued in the finite orthogonal projectors, and satisfying 
$$P_r^{(u)}(u)\ne 0\ ,\ r=1,\dots ,n\ .$$
Moreover, by continuity,
$${\rm rk} P_r^{(u)}={\rm rk}P_r^{(u_0)}:=d_r+1\ .$$
In the special case $u\in  \mathcal V_{(d_1,\dots ,d_n)}$, $P_r^{(u)}$ is precisely the orthogonal projector onto $E_u(s_r)+F_u(s_r)$. 
We then define a  map $\tilde \Phi _n$ on $ V_0\cap \mathcal V(d)$ by setting
$$\tilde \Phi _n(u)= ((s_r(u))_{1\le r\le n} ; (\Psi _r(u))_{1\le r\le n})\ ,$$
with
\begin{eqnarray*}
s_{2j-1}(u):=\frac{\Vert H_u(P^{(u)}_{2j-1}(u))\Vert }{\Vert P^{(u)}_{2j-1}(u)\Vert }\  &,& \   s_{2k}(u):=\frac{\Vert K_u(P^{(u)}_{2k}(u))\Vert }{\Vert P^{(u)}_{2k}(u)\Vert }\ ,\\
\Psi _{2j-1}(u):=\frac{s_{2j-1}(u)P^{(u)}_{2j-1}(u)}{H_u(P^{(u)}_{2j-1}(u))}\ &,&\ \Psi _{2k}(u)=\frac{K_u(P^{(u))}_{2k}(u)}{s_{2k}(u)P^{(u)}_{2k}(u)}\ .
\end{eqnarray*}
The mapping $\tilde \Phi _n$ is smooth from $V_0\cap \mathcal V (d)$ into $\Omega _n\times \mathcal R_d^n $ , where $\mathcal R_d$ denotes the manifold of rational functions with  numerators and denominators of degree at most $\left [\frac{d+1}2\right ]$. Moreover, from Proposition \ref{action}, the restriction of $\tilde \Phi _n$ to $V_0\cap \mathcal V_{(d_1,\dots, d_n)} $ coincides with $\Phi _{d_1,\dots ,d_n}$. This proves in particular that $\Phi _{d_1,\dots ,d_n}$ is continuous. For future reference, let us state more precisely what we have proved.
\begin{lemma}\label{Phismooth}
For every $u_0\in \mathcal V_{(d_1,\dots ,d_n)}$, there exists a neighborhood $V$ of $u_0$ in $\mathcal V(d)$, $d=n+2\sum _{r=1}^n d_r\ $, and a smooth mapping $\tilde \Phi _n$ from this neighborhood into some manifold, such that the restriction of $\tilde \Phi _n$  to $V\cap  \mathcal V_{(d_1,\dots ,d_n)}$ coincides with $\Phi _{d_1,\dots ,d_n}$.
\end{lemma}
\subsection{The explicit formula, case $n$ even.}\label{section explicit}
Assume that $n=2q$ is an even integer.\\
The fact that the mapping $\Phi _{d_1,\dots ,d_n}$ is one-to-one follows from an explicit formula giving $u$ in terms of $\Phi _{d_1,\dots ,d_n}(u)$, which we  establish
in this subsection.

We use the expected description of elements of $\Phi^{-1}(\mathcal S_n)$ suggested by the action of $H_u, K_u$ onto the orthogonal projections $u_j,u'_k$  of $u$ onto the 
corresponding eigenspaces of $H_u^2, K_u^2$ respectively, namely
\begin{equation}\label{crucial}
\rho _ju_j=\Psi_{2j-1}H_u(u_j)\ ,\ K_u(u'_k)=\sigma _k\Psi_{2k} u'_k\ ,\ j,k=1,\dots, q\ ,
\end{equation}
where the $\Psi_r$'s  are  Blaschke products.
\s
Using the identity $I=SS^*+(\, .\, \vert 1)\ ,$ we obtain
$$H_u(u_j)=SK_u(u_j)+\Vert u_j\Vert ^2\ .$$
We then use  the formula (\ref{urho}), in this setting
\begin{equation}\label{uj}
u_j=\tau _j^2\sum _{k=1}^q \frac 1{\rho _j^2-\sigma _k^2}u'_k\ ,\ \tau _j^2=\Vert u_j\Vert ^2\ ,
\end{equation}
to get
\begin{eqnarray*}
\Psi _{2j-1}(z)H_u(u_j)(z)&=&\tau_j ^2\Psi _{2j-1}(z)\left ( z\sum _{k=1}^q   \frac{K_u(u'_k)(z)}{\rho _j^2-\sigma _k^2}  +1 \right )\\
&=&\tau_j ^2\left ( \sum _{k=1}^q   \frac{\sigma _kz\Psi _{2j-1}(z)\Psi _{2k}(z)u'_k(z)}{\rho _j^2-\sigma _k^2}  +\Psi _{2j-1}(z) \right )\ ,
\end{eqnarray*}
after using the second identity of (\ref{crucial}). On the other hand, using again (\ref{uj}),
$$\rho _ju_j(z)=\tau_j ^2\sum _{j=1}^k \frac{\rho _j\, u_k'(z)}{\rho _j^2-\sigma _k^2}\ .$$
Applying the first identity in (\ref{crucial}), we conclude
$$\sum _{k=1}^q \frac{\rho _j\ -z\sigma _kz\Psi _{2j-1}(z)\Psi _{2k}(z)}{\rho _j^2-\sigma _k^2}u'_k(z)=\Psi _{2j-1}(z)\ ,$$
which is precisely
\begin{equation}\label{CUprime}
\mathscr C(z)\mathcal U'(z) =(\Psi _{2j-1}(z))_{1\le j\le q}\ ,
\end{equation}
where  
\begin{equation}\label{c}
c_{k\ell }(z):=\frac{\rho _k -\sigma _\ell z\Psi _{2k-1}(z)\Psi _{2\ell }(z)}{\rho _k ^2- \sigma _{\ell }^2}\ ,
\end{equation}
and $\mathcal U'(z):=(u'_k(z))_{1\le k\le q}$. Since 
$$u=\sum _{k=1}^q u'_k\ ,$$
we conclude, for every $z$ such that $\mathscr C(z)$ is invertible, 
$$u(z)=\langle \mathscr C(z)^{-1}(\Psi _{2j-1}(z))_{1\le j\le q},{\bf 1}\rangle \ ,$$
as claimed by formula (\ref{luminy}).

Similarly, writing $H_u(u'_k)=SK_u(u'_k)+\Vert u'_k\Vert ^2$, and using (\ref{usigma}) in this setting which in this setting reads
\begin{equation}\label{uprimek}
u_k'=\kappa _k^2\sum _{j=1}^q \frac 1{\rho _j^2-\sigma _k^2}u_j\ ,\ \kappa _k^2=\Vert u'_k\Vert ^2\ ,
\end{equation}
we obtain that 
$$u_j=\Psi _{2j-1}h_j\ ,\ {\rm or}\ h_j:=\frac{1}{\rho _j}H_u(u_j)\ ,$$
where
$$\mathcal H(z):=\left(h_j(z)\right)_{1\le j\le q}\ ,$$
is the solution of
\begin{equation}\label{CH}
^t\mathscr C(z)\mathcal H(z)={\bf 1}\ .
\end{equation}
Let us come to  the invertibility of matrices $\mathscr C(z)$. Since $\mathscr C(0)$ is invertible, $\mathscr C(z)$ is invertible for $z$ in a small disc centered at $0$. This is enough for proving the injectivity of $\Phi _{d_1,\dots ,d_n}$.  However, let us prove that this property holds for every $z$ in the closed unit disc. Recall that
\begin{equation}\label{PsiPD}
\Psi _r(z)={\rm e}^{-i\psi _r}\frac{P_r(z)}{D_r(z)}\ ,\ D_r(z):=z^{d_r}\overline P_r\left (\frac 1z\right )\ ,
\end{equation}
where $P_r$ is a monic polynomial of degree $d_r$. Introduce the matrix $\mathscr C^\# (z)=(c^\# _{k\ell }(z))_{1\le k,\ell \le q}$ 
as
\begin{equation}\label{cdiese}
c^\# _{k\ell }(z)=\frac{\rho _k D_{2\ell }(z)D_{2k -1}(z)-\sigma _\ell z{\rm e}^{-i(\psi _{2\ell }+\psi _{2k -1})} P_{2\ell }(z)P_{2k -1}(z)}{\rho _k ^2-\sigma _\ell ^2}\ .
\end{equation}
Introducing  the notation 
$${\rm diag}(\lambda_j)_{1\le j\le q}:=\left(\begin{array}{cccc}
\lambda_1&0&\dots& 0\\
0&\ddots&\ddots&\vdots\\
\vdots&\ddots&\ddots&0\\
0&\dots& 0&\lambda_q\end{array}\right)$$
we have 
$$\mathscr C(z)={\rm diag}\left (\frac{1}{D_{2k-1}(z)}\right )\mathscr C^\# (z){\rm diag}\left (\frac{1}{D_{2\ell}(z)}\right )\ ,$$
so that $u(z)$ can be written equivalently as
$$u(z)=\langle \mathscr C^\# (z)^{-1}({\rm e}^{-i\psi _{2j-1}}P_{2j-1}(z))_{1\le j\le q}, (D_{2k}(z))_{1\le k\le q}\rangle\ ,$$
for every $z$  such that $\det \mathscr C^\# (z)\ne 0$. Using Cramer's formulae for the inverse of a matrix, it is easy to see that the degree of the denominator of the above rational function is at most $N=N:=q+\sum _{r=1}^n d_r\ ,$ and that the degree of the numerator is at most $N-1$. If $\det \mathscr C^\# (z)= 0$ for some $z$ in the closed unit disc, then, in order to preserve the analyticity of $u$,  $u$ should be written as a quotient of polynomials  of degrees smaller than $N-1$ and $N$ respectively, so that the rank of $H_u$ would be smaller than it is. Consequently, 
$\det \mathscr C^\# (z)\ne 0$ for every $z$ in the close unit disc, so that formula (\ref{luminy}) holds for every such $z$.

\subsection{Surjectivity in the case $n$ even.} 

Our purpose is now to prove that the mapping $\Phi _{d_1,\dots ,d_n}$ is onto. Since we got a candidate from the formula giving $u$
in the latter section, it may seem natural to try to check that this formula indeed provides an element $u$ of $\mathcal V_{(d_1,\dots ,d_n)}$
with the required $\Phi _{d_1,\dots ,d_n}(u)$. However, in view of the complexity of  formula (\ref{luminy}) , it seems difficult to infer from them the spectral properties of $H_u$ and $K_u$. We shall therefore use an indirect method,  by  proving that the mapping $\Phi _{d_1,\dots ,d_n}$ on $\mathcal V_{(d_1,\dots ,d_n)}$ is open, closed, and that the source space $\mathcal V_{(d_1,\dots ,d_n)}$ is not empty. Since the target space $\mathcal S_{(d_1,\dots,d_n)}$ is clearly connected, this will imply the surjectivity. A first step in proving that $\Phi _{d_1,\dots ,d_n}$ is an open mapping, consists in the construction of  antilinears operators $H$ and $K$ satisfying the
required  spectral properties, and which will be finally identified as $H_u$ and $K_u$. 
\subsubsection{Construction of the operators $H$ and $K$}
\noindent Let $$\mathcal P=((s_r)_{1\le r\le n}, (\Psi_r)_{1\le r\le n})$$
be an arbitrary element of  $ \mathcal S_{(d_1,\dots,d_n)}$ for some non negative integers $ d_r\ .$ We look for $u\in \mathcal V_{(d_1,\dots, d_n)}$,
$\Phi _{d_1,\dots ,d_n}(u)=\mathcal P$. We set $\rho _j:=s_{2j-1}\ ,\ \sigma _k:=s_{2k}\ ,\ 1\le j,k\le q\ .$
\s
Firstly, we define matrices $\mathscr C(z)$ and $\mathscr C^\# (z)$ using formulae (\ref{c}), (\ref{PsiPD}), (\ref{cdiese}). We assume moreover  the following open properties,
\begin{equation}\label{hypo1}
Q(z):=\det \mathscr C^\#(z)\ne 0\ ,\ z\in \overline \D\ ,\ {\rm deg}(Q)=N:=q+\sum _{r=1}^n d_r\ .
\end{equation}
We then define $\mathcal H(z)=(h_j(z))_{1\le j\le q}$ and $\mathcal U'(z)=(u'_k(z))_{1\le k\le q}$ to be the solutions of (\ref{CH}) and (\ref{CUprime}) respectively. We then define $\tau _j, \kappa _k >0$ by 
\begin{equation}\label{J(x)}
\prod_{j=1}^q \frac{1-x\sigma_j^2}{1-x\rho_j^2}=1+x\sum_{j=1}^q \frac{\tau_j^2}{1-x\rho_j^2}
\end{equation}
\begin{equation}\label{1/J(x)}
\prod_{j=1}^q \frac{1-x\rho_j^2}{1-x\sigma_j^2}=1-x\left (\sum_{k=1}^q \frac{\kappa_k^2}{1-x\sigma_k^2}\right )
\end{equation}
As a consequence, we recall from section \ref{Bateman} that the matrix
$$\mathcal A:=\left (\frac{1}{\rho _j^2-\sigma _k^2}\right )_{1\le j,k\le q}$$
is invertible, with 
$$\mathcal A^{-1}={\rm diag}(\kappa _k^2)^t\mathcal A {\rm diag} (\tau _j^2)\ .$$
Similarly, the matrix 
$$\mathcal B:=\mathcal A {\rm diag} (\kappa _k^2)$$
satisfies
$$\mathcal B^{-1}=^t\mathcal A {\rm diag} (\tau _j^2)\ .$$
From these identities and
$$\mathscr C(z)={\rm diag}(\rho _j)\mathcal A-z{\rm diag}(\Psi _{2j-1}(z))\mathcal A{\rm diag}(\sigma _k\Psi _{2k}(z)),$$ we easily prove the following lemma.
\begin{lemma}\label{BCPsi}
For every $z\in \overline \D$,
$$\mathscr C(z)\ ^t\mathcal B\ {\rm diag}(\Psi _{2\ell -1}(z))_{1\le \ell \le q}={\rm diag}(\Psi _{2j-1}(z))_{1\le j\le q}\ \mathcal B\ ^t\mathscr C(z)\ .$$
\end{lemma}
In view of Lemma \ref{BCPsi} and of the definition of $\mathcal H(z)$, we infer
$$\mathscr C(z)^t\mathcal B {\rm diag} (\Psi _{2j-1}(z))\mathcal H(z)= {\rm diag}(\Psi _{2j-1}(z))_{1\le j\le q}\ \mathcal B ({\bf 1}) ={\rm diag}(\Psi _{2j-1}(z))_{1\le j\le q}$$
where we have used $\mathcal B({\bf 1}) ={\bf 1}$, which is a consequence of identity (\ref{1/J(x)}). We conclude, in view of the definition of $\mathcal U'(z)$, 
\begin{equation}\label{U'PsiH}
\mathcal U'(z)=^t\mathcal B {\rm diag} (\Psi _{2j-1}(z))\mathcal H(z),
\end{equation}
or
\begin{equation}\label{u'Psih} 
u'_k=\kappa _k^2\sum _{j=1}^q \frac{\Psi _{2j-1}h_j}{\rho _j^2-\sigma _k^2}\ ,\ \Psi _{2j-1}h_j=\tau _j^2\sum _{k=1}^q \frac{u'_k}{\rho _j^2-\sigma _k^2}\ .\end{equation}
We are going to define  antilinear operators $H$ eand $K$ on the space $$W=\frac{\C_{N-1}[z]}{Q(z)}.$$  Cramers' formulae show that
 \begin{equation}\label{hDR}
 h_j(z)=D_{2j-1}(z)R_{2j-1}(z)\ ,
 \end{equation}
where the numerator of $R_{2j-1}$ is a polynomial of degree at most $N-1-d_{2j-1}$, the denominator being $Q$. Similarly,
$$u'_k(z)=D_{2k}(z)R_{2k}(z)\ ,$$
where the numerator of $R_{2k}$ is a polynomial of degree at most $N-1-d_{2k}$, the denominator being $Q$.
\s
We can therefore define the following elements of $W$,
\begin{eqnarray*}
e_{2j-1, a}(z)&:=&\frac{z^a}{D_{2j-1}(z)}h_j(z),\  0\le a\le d_{2j-1}\ ,\\ e_{2k,b}(z)&:=&\frac{z^b}{D_{2k}(z)} u_{k}'(z)\ ,\ 1\le b\le d_{2k}\ ,
\end{eqnarray*}
for $1\le j,k\le q$. 
 Let 
 \begin{eqnarray*}
 \mathcal E&:=&\left ((e_{2j-1,a})_{0\le a\le d_{2j-1}}\ ,\ (e_{2k,b})_{1\le b\le d_{2k}}\right )_{1\le j,k\le q} \ ,\\
\mathcal E'&:=&\left ((e_{2j-1,a})_{1\le a\le d_{2j-1}}\ ,\ (e_{2k,b})_{0\le b\le d_{2k}}\right )_{1\le j,k\le q} \ .
\end{eqnarray*}
We need a second open assumption.
\begin{equation}\label{hypo2}
\mathcal E\text{ and } \mathcal E'\text{ are bases of }W.
\end{equation}

We then define  antilinear operators $H$ and $K$ on $W$  by
\begin{eqnarray*}
H(e_{2j-1,a})&:=&\rho _j{\rm e}^{-i\psi _{2j-1}}e_{2j-1,d_{2j-1}-a}\ ,\ 0\le a\le d_{2j-1}\ ,\\
H(e_{2k,b})&:=&\sigma _k{\rm e}^{-i\psi _{2k}}e_{2k,d_{2k}+1- b},\ 1\le b\le d_{2k}\ ,\\
K(e_{2j-1,a})&:=& \rho _j {\rm e}^{-i\psi _{2j-1}}e_{2j-1,d_{2j-1}-a-1}    \ 0\le a\le d_{2j-1}-1\ ,\\
K(e_{2k,b})&:=&\sigma _k{\rm e}^{-i\psi _{2k}}e_{2k,d_{2k}- b},\ 0\le b\le d_{2k}\  .
\end{eqnarray*}
for $1\le j,k\le q$. 
\s
Notice that $S^*$ acts on $W$. Indeed, if $$Q(z):= Q(0)(1-c_1z-c_2z^2-\dots-c_N z^N),$$ we have
$$S^*\left (\frac {1}{Q}\right )=\sum_{j=1}^N c_j  \frac{z^{j-1}}{ Q(z)}\ .$$
As a first step, we prove that $S^*H=K$. Indeed, we just have to check this identity on each vector of the basis $\mathcal E$. In view of the above definition of $H$ and $K$, the identity is trivial on the vectors
$$e_{2j-1,a}\ ,\ 0\le a\le d_{2j-1}-1\ ,\ e_{2k,b}\ ,\ 1\le b\le d_{2k}\ .$$
It remains to prove it for $e_{2j-1,d_{2j-1}}$, or equivalently for $\Psi _{2j-1}h_j$. We have
$$S^*H(\Psi _{2j-1}h_j)=\rho _jS^*h_j\ .$$
This quantity can be calculated by applying $S^*$ to the equation satisfied by $\mathcal H(z)$, namely
$$^t\mathcal A {\rm diag}(\rho _j) \mathcal H(z)-z{\rm diag}(\sigma _k\Psi _{2k}(z))^t \mathcal A{\rm diag}{\Psi _{2j-1}(z)}\mathcal H(z)={\bf 1}\ .$$
We obtain
\beno
^t\mathcal A {\rm diag}(\rho _j) S^*\mathcal H(z)&=&{\rm diag}(\sigma _k\Psi _{2k}(z))^t \mathcal A{\rm diag}{\Psi _{2j-1}(z)}\mathcal H(z)\\
&=&{\rm diag}\left (\frac{\sigma _k\Psi _{2k}(z)}{\kappa _k^2}\right )\mathcal U'(z)\ ,
\eeno
where we have used (\ref{U'PsiH}).  We obtain, by using the formula for $\mathcal A^{-1}$, 
\begin{equation}\label{S*hj}
S^*H(\Psi _{2j-1}h_j)(z)=\rho _jS^*h_j(z)=\tau _j^2\sum _{k=1}^q \frac{\sigma _k\Psi _{2k}u'_k}{\rho _j^2-\sigma _k^2}\ .
\end{equation}
On the other hand, using the second part of (\ref{u'Psih}), 
$$K(\Psi _{2j-1}h_j)=\tau _j^2\sum _{k=1}^q \frac{Ku'_k}{\rho _j^2-\sigma _k^2}=\tau _j^2\sum _{k=1}^q \frac{\sigma _k\Psi _{2k}u'_k}{\rho _j^2-\sigma _k^2}\ .$$
Therefore we have proved $S^*H(\Psi _{2j-1}h_j)=K(\Psi _{2j-1}h_j)$, and finally $S^*H=K$.
\s
As a next step, we show the following identity
\begin{equation}\label{S*H}
\forall h\in W\ ,\ S^*HS^*(h)=H(h)-(1\vert h)u\ ,\ u:=\sum _{k=1}^q u'_k\ .
\end{equation}
It is enough to check this identity on any vector of the basis $\mathcal E$. In view of the definition of $H$, it is trivially satisfied on the vectors
$$e_{2j-1,a}\ ,\ 1\le a\le d_{2j-1}\ ,\ e_{2k,b}\ ,\ 1\le b\le d_{2k}\ .$$
Therefore we just have to check it on $e_{2j-1,0}$, or equivalently on $h_j$. On the one hand, we have, from the identity (\ref{S*hj}),
$$S^*h_j(z)=\frac{\tau _j^2}{\rho _j}\sum _{k=1}^q \frac{\sigma _k\Psi _{2k}u'_k}{\rho _j^2-\sigma _k^2}\ .$$
Applying $S^*H=K$, we get
\beno
S^*HS^*h_j(z)&=&\frac{\tau _j^2}{\rho _j}\sum _{k=1}^q \frac{\sigma _k^2u'_k}{\rho _j^2-\sigma _k^2}\\
&=&\rho _j \Psi _{2j-1}(z)h_j(z)-\frac{\tau _j^2}{\rho _j}u(z)\ ,
\eeno
where we have used (\ref{u'Psih} ) again. We conclude by observing that $\rho _j\Psi _{2j-1}h_j=H(h_j)$, and that
$$(1\vert h_j)=\overline {h_j(0)} =\frac{\tau _j^2}{\rho _j}\ ,$$
in view of the equation on $\mathcal H(z)$ for $z=0$,
$$\mathcal A {\rm diag}(\rho _j)\mathcal H(0)={\bf 1}\ ,$$
and of the expression of $\mathcal A^{-1}={\rm diag}(\tau _j^2) \mathcal B$.
\s
Finally, we prove that an operator satisfying equality (\ref{S*H}) is actually a Hankel operator. 

\begin{lemma}\label{coroS*HS*}
Let $N$ be a positive integer. Let $$Q(z):= 1-c_1z-c_2z^2-\dots-c_N z^N$$ be a complex valued polynomial with no roots in the closed unit disc. 
Set $$W:=\frac{\C_{N-1}[z]}{Q(z)}\subset L^2_+(\S^1)\ .$$
Let $H$ be an antilinear operator on $W$ satisfying
$$S^*HS^*=H-(1\vert \cdot)u$$ on $W$, for some $u\in W$.
Then $H$ coincides with the Hankel operator of symbol $u$ on $W$.
\end{lemma}
\begin{proof}
Consider the operator $\tilde H:= H-H_u$, then $S^*\tilde HS^*=\tilde H$ on $W$ and hence, it suffices to show that, if $H$ is an antilinear operator on $W$ such that $S^*HS^*=H$, then $H=0$.\\

The family $(e_j)_{1\le j\le N}$ where $$e_0(z)=\frac 1{Q(z)}\ , \ e_j(z)=S^j e_0(z), \  j=1,\dots,N-1$$ is a basis of $W$. 
Using that $$S^*HS^*=H$$
we get  on the one hand that $He_k=(S^*)^k He_0$. On the other hand, since
$$S^*e_0=S^*\left (\frac 1{Q}\right )=\sum_{j=1}^N c_j  e_{j-1} \ ,$$
this implies  $$He_0=S^*HS^*e_0=\sum_{j=1}^N\overline{c_j} (S^*)^j He_0\ ,$$ hence $\overline Q (S^*)H(e_0)=0$ .
Observe that, by the spectral mapping theorem,  the spectrum of $\overline{Q}(S^*)$ is contained into
$\overline Q(\overline \D)$,  hence $\overline{Q}(S^*)$ is one-to-one. We conclude that $H(e_0)=0$, and finally that $H= 0$.

\end{proof}

Applying Lemma  \ref{coroS*HS*} to our vector space $W$,  we conclude  that $H=H_u$, and consequently $K=S^*H_u=K_u$. In view of the definition of 
$H$ and $K$, we conclude that  $\Phi _{d_1,\dots ,d_n}(u)=\mathcal P$. 
\subsubsection{The mapping $\Phi _{d_1,\dots ,d_n}$ is open from $\mathcal V_{(d_1,\dots ,d_n)}$ to $\mathcal S_{(d_1,\dots,d_n)}$.  }
Notice that we have not yet completed the proof of Theorem \ref{mainfiniterank} since the previous calculations were made under the assumptions
(\ref{hypo1}) and (\ref{hypo2}). In other words, we proved that an element $\mathcal P$ of the target space satisfying  (\ref{hypo1}) and (\ref{hypo2})
is in the range of $\Phi _{d_1,\dots ,d_n}$. On the other hand, in section \ref{section explicit}, we proved that these properties are satisfied by the elements of the range of $\Phi _{d_1,\dots ,d_n}$. Since these hypotheses are clearly open in the target space, we infer that the range of $\Phi _{d_1,\dots ,d_n}$ is open.

\subsubsection{The mapping $\Phi _{d_1,\dots ,d_n}$ is closed.}\label{closed}
Let $(u^\e )$ be a sequence of $ \mathcal V_{(d_1,\dots ,d_n)}$ such that
$\Phi _{d_1,\dots ,d_n}(u^\e ):=\mathcal P^\e $ converges to some $\mathcal P$ in $\mathcal S_{(d_1,\dots,d_n)}$ as $\e $ goes to $0$.
In other words, 
$$\mathcal P^\e =\left((s_r^\e )_{1\le r\le 2q}, (\Psi_r^\e)_{1\le r\le  2q}\right)
 \longrightarrow \mathcal P=\left ((s_r)_{1\le r\le 2q}, (\Psi_r)_{1\le r\le 2q} \right) $$
in $\mathcal S_{(d_1,\dots,d_n)}$ as $\e\to 0$. We have to find $u$ such that $\Phi(u)=\mathcal P$. Since 
$$\Vert u^\e \Vert _{H^{1/2}} \simeq {\rm Tr}(H_{u^\e }^2)=\sum_{r=1}^{2q} d_r (s_r^\e)^2+\sum _{j=1}^q (s _{2j-1}^\e )^2$$
is bounded, we may assume, up to extracting a subsequence,  that $u^\e $ is weakly convergent to some $u$ in $H^{1/2}_+$, and strongly convergent in $L^2_+$ by the Rellich theorem. Moreover,
the rank of $H_u$ is at most $N=q+\sum_{r=1}^{2q}d_r$.

Denote by $u^\e _j$ the orthogonal projection of $u^\e $ onto $\ker (H_u^2-(s _{2j-1}^\e)^2 I)$, $j=1,\dots ,q$, and by $(u^\e _k)' $  the orthogonal projection of $u^\e $ onto $\ker (K_u^2-(s _{2k}^\e)^2 I)$, $k=1,\dots , q$. Since all these functions are bounded in $L^2_+$, we may assume that, for the weak convergence in $L^2_+$,
$$ u^\e _j\rightharpoonup v_j\ ,\ (u^\e _k)' \rightharpoonup v'_k\ .$$
Taking advantage of the strong convergence of $u^\e $ in $L^2_+$, we can pass to the limit in
$$(u^\e \vert u^\e _j)=\Vert u^\e _j\Vert^2=:(\tau ^\e _j)^2 \ ,\ (u^\e \vert (u^\e _k)')=\Vert (u^\e _k)'\Vert^2=:(\kappa _k^\e)^2\ ,$$
and obtain, thanks to the explicit expressions (\ref{tau}), (\ref{kappa}) of $\tau_j^2, \kappa _k^2$ in terms of the $s_r$, 
$$(u\vert v_j)=\tau _j^2>0\ ,\ (u\vert v_k')=\kappa _k^2>0\ ,$$
in particular $v_j\ne 0, v_k'\ne 0$ for every $j,k$.

On the other hand, passing to the limit in
\begin{eqnarray*}
s ^\e _{2j-1}u^\e _{j}&=&\Psi_{2j}^\e\  H_{u^\e }u^\e _j\ ,\ H_{u^\e }^2(u^\e _j)=(s ^\e _{2j-1})^2u^\e _j\ ,\\
K_{u^\e }(u^\e _k)'&=&s ^\e _{2k}\Psi_{2k}^\e\ (u^\e _k)'\ ,\ K_{u^\e }^2(u^\e _k)'=(s ^\e _{2k})^2(u^\e _k)'\ ,\\
u^\e &=&\sum _{j=1}^q u^\e _j =\sum _{k=1}^q (u^\e_k)' \ ,
\end{eqnarray*}
we obtain
\begin{eqnarray*}
s  _{2j-1}v_j&=&\Psi_{2j}\  H_{u }v _j\ ,\ H_{u }^2(v _j)=s _{2j-1}^2v _j\ ,\\
K_{u }v_k'&=&s_{2k}\Psi_{2k} \ v _k' \ ,\ K_u^2(v'_k)=s_{2k}^2v'_k\ ,\\
u &=&\sum _{j=1}^q v _j =\sum _{k=1}^q v_k' \ .
\end{eqnarray*}
This implies that $u\in \mathcal V_{(d_1,\dots ,d_n)}$, $v_j=u_j, v'_k=u'_k$, and $\Phi(u)=\mathcal P$. The proof of Theorem \ref{mainfiniterank} is thus 
complete in the case $n=2q$, under the assumption that $ \mathcal V_{(d_1,\dots ,d_n)}$ is non empty. 

\subsection{$ \mathcal V_{(d_1,\dots ,d_n)}$ is non empty, $n$ even }

Let $n$ be a positive even integer. The aim of this section is to prove that $\mathcal V _{(d_1,\dots,d_n)}$ is not empty for any multi-index $(d_1,\dots,d_n)$ of non 
negative integers.

The preceding section implies that, as soon as $\mathcal V _{(d_1,\dots,d_n)}$ is non empty, it is homeomorphic to $\mathcal S_{(d_1,\dots,d_n)}$, via the explicit formula (\ref{luminy}). We argue by induction on the integer  $d_1+\dots
+d_n$. In  the generic case consisting of simple eigenvalues (see \cite{GG1}), we proved that for any positive integer $q$, $\mathcal V _{(0,\dots,0)}$ (which was denoted by $\mathcal V_{\rm gen}(2q)$ in \cite{GG1}) is non empty. As a consequence, to any given sequence $((s_r),(\Psi_r))\in \Omega_{2q}\times \T^{2q}$ corresponds a unique $u\in \mathcal V_{(0,\dots,0)}$, the $s_{2j-1}^2$ being the simple eigenvalues of $H_u^2$ and the $s_{2j}^2$ the simple eigenvalues of $K_u^2$. This gives the theorem in the case $(d_1,\dots, d_n)=(0,\dots,0)$ for every $n$, which is one of the main theorems of \cite{GG2}.
Let us turn to the induction argument, which  is clearly a consequence of the following lemma.

\begin{lemma}\label{non empty}
Let $n=2q$, $(d_1,\dots,d_n)$ and $1\le r\le n-1$.
Assume $$\mathcal V_{(d_1,\dots,d_r,0,0, d_{r+1},\dots, d_n)}\text{ is non empty, }$$ then $$\mathcal V _{(d_1,\dots,d_{r-1},d_r+1,d_{r+1},\dots,d_n)}\text{ is non empty.}$$
\end{lemma}

\begin{proof} The main idea is to collapse three consecutive singular values\index{singular value}. More precisely, we will construct elements of $\mathcal V _{(d_1,\dots,d_{r-1},d_r+1,d_{r+1},\dots,d_n)}$ as limits of sequences in $\mathcal V_{(d_1,\dots,d_r,0,0, d_{r+1},\dots, d_n)}$ such that $s_r, s_{r+1}, s_{r+2}$ converge to the same value. The difficulty is to make the sequence converge strongly, and we will see that this requires some non resonance assumption on the 
corresponding angles.
\s
We  consider the case $r=1$.  The proof in the other cases  follows the same lines. Write $m_j:=d_{2j-1}+1$ and $\ell_k:=d_{2k}+1$. From the assumption, $$\mathcal V:=\mathcal V_{(d_1,0,0, d_{2},\dots, d_n)}\text{ is non empty, }$$ hence $\Phi$ establishes a diffeomorphism from $\mathcal V$ into $$\mathcal S_{(d_1,0,0, d_2,\dots,d_n)}\ .$$ 
Therefore, given $\rh >\si _2>\rh _3>\si _3>\dots >\rh _{q+1}>\si _{q+1}>0$, and $\Psi _1,\theta _1,\f _2, \Psi _4, \dots ,\Psi _{n+2}$, for every $\e >0$ small enough,  we define
$$u^\e :=u((s_1^\e ,\dots ,s_{2(q+1)}^\e ),(\Psi _1 ,\dots ,\Psi _{2(q+1)} ))$$
with
\beno
s_1^\e &:=&\rh +\e \ ,\ s_2^\e :=\rh \ ,\ s_3^\e :=\rh -\e\ ,\ s_{2k}^\e :=\sigma _k\ ,\ 2\le k\le q+1\ ,\\ 
s_{2j-1}^\e &:=&\rh _j\ ,\ 3\le j\le q+1\ , \Psi _2 :={\rm e}^{-i\t _1}\ ,\ \Psi _3:={\rm e}^{-i\f _2}\ .
 \eeno
By making $\e $ go to $0$, we are going to construct $u$ in $\mathcal V _{(d_1+1,d_2,\dots,d_n)}$, such that  $s_1(u)=\rh$ is of multiplicity $m_1+1=d_1+2$, $s_{2j-1}(u)=\rho_{j+1}$, $j=2,\dots ,q,$ is of multiplicity $m_{j}$ and $s_{2k}(u)=\sigma_{k+1}$ for $k=1,\dots ,q$, is of multiplicity $\ell_{k}$.
\s
First of all, observe that $u^\e $ is bounded in $H^{1/2}_+$, since its norm is equivalent to $${\rm Tr}(H_{u^\e }^2)=(d_1+1)(\rh +\e )^2+(\rh -\e )^2+d_2\si _2^2+(d_3+1)\rh _3^2+\dots $$ Hence, by the Rellich theorem, up to extracting a subsequence, $u^\e $ strongly converges in $L^2_+$ to some $u\in H^{1/2}_+$. Similarly, the orthogonal projections $u^\e _j$ and $(u^\e_k)'$ are bounded in $L^2_+$, hence are weakly convergent to $v_j$, $v'_k$. Arguing as in the previous subsection, we have
\beno
(u\vert v_1)&=&\lim _{\e \rightarrow 0}\Vert u_1^\e \Vert ^2=\lim _{\e \rightarrow 0}\frac{(\rh +\e )^2-\rh ^2}{(\rh +\e )^2-(\rh -    \e )^2}
\frac{\prod _{k\ge 2}((\rh +\e )^2-\si _k^2)}{\prod _{k\ge 3}((\rh+\e )^2-\rh _k^2)}\\
&=&\frac 12\frac{\prod _{k\ge 2}(\rh ^2-\si _k^2)}{\prod _{k\ge 3}(\rh^2-\rh _k^2)},\\
(u\vert v_2)&=&\lim _{\e \rightarrow 0}\Vert u_2^\e \Vert ^2=\lim _{\e \rightarrow 0}\frac{(\rh -    \e )^2-\rh ^2}{(\rh -   \e )^2-(\rh + \e )^2}
\frac{\prod _{k\ge 2}((\rh -    \e )^2-\si _k^2)}{\prod _{k\ge 3}((\rh-    \e )^2-\rh _k^2)}\\
&=&\frac 12\frac{\prod _{k\ge 2}(\rh ^2-\si _k^2)}{\prod _{k\ge 3}(\rh^2-\rh _k^2)}, 
\eeno
\beno
(u\vert v_j)&=&\lim _{\e \rightarrow 0}\Vert u_j^\e \Vert ^2\ , j\ge 3,\\
&=&\lim _{\e \rightarrow 0}\frac{\rh _j^2-\rh ^2}{(\rh_j^2-(\rh -    \e )^2)(\rh _j^2-(\rh +\e )^2)}
\frac{\prod _{k\ge 2}(\rh _j^2-\si _k^2)}{\prod _{k\ge 3, k\ne j}(\rh _j^2-\rh _k^2)}\\
&=&\frac 1{\rh _j^2-\rh ^2}\frac{\prod _{k\ge 2}(\rh _j^2-\si _k^2)}{\prod _{k\ge 3, k\ne j}(\rh _j^2-\rh _k^2)}
\eeno
and 
\beno
(u\vert v'_1)&=&\lim _{\e \rightarrow 0}\Vert (u_1^\e)' \Vert ^2\\&=&\lim _{\e \rightarrow 0}(\rh ^2-(\rh +\e )^2)(\rh^2-(\rh -    \e )^2)
\frac{\prod _{k\ge 3}(\rh ^2-\rh _k^2)}{\prod _{k\ge 2}(\rh^2-\si _k^2)}=0,\\
(u\vert v'_k)&=&\lim _{\e \rightarrow 0}\Vert (u_k^\e)' \Vert ^2\ , \ k\ge 2\ \\
&=&\lim _{\e \rightarrow 0}\frac{(\si_k ^2-(\rh +\e )^2)(\si _k^2-(\rh -    \e )^2)}{\si _k^2-\rh ^2}
\frac{\prod _{j\ge 3}(\si_k ^2-\rh _j^2)}{\prod _{j\ge 2, j\ne k}(\si _k^2-\si _j^2)}\\
&=&\frac 1{\si_k^2-\rh ^2}\frac{\prod _{j\ge 3}(\si_k ^2-\rh _j^2)}{\prod _{j\ge 2, j\ne k}(\si _k^2-\si _j^2)}
\eeno
In view of these identities, we infer that $v_j, j\ge 1$ and $v'_k,k\ge 2$ are not $0$, while $v'_1=0$. Passing to the limit into the identities
\begin{eqnarray*}
s ^\e _{2j-1}u^\e _{j}&=&\Psi_{2j}\  H_{u^\e }u^\e _j\ ,\ H_{u^\e }^2(u^\e _j)=(s ^\e _{2j-1})^2u^\e _j\ ,\\
K_{u^\e }(u^\e _k)'&=&s ^\e _{2k}\Psi_{2k}\ (u^\e _k)'\ ,\ K_{u^\e }^2(u^\e _k)'=(s ^\e _{2k})^2(u^\e _k)'\ ,\\
u^\e &=&\sum _{j=1}^q u^\e _j =\sum _{k=1}^q (u^\e_k)' \ ,
\end{eqnarray*}
we obtain
\begin{eqnarray*}
s  _{2j-1}v_j&=&\Psi_{2j}\  H_{u }v _j\ ,\ H_{u }^2(v _j)=s _{2j-1}^2v _j\ ,\\
K_{u }v_k'&=&s_{2k}\Psi_{2k} \ v _k' \ ,\ K_u^2(v'_k)=s_{2k}^2v'_k\ ,\\
u &=&\sum _{j=1}^{q+1} v _j =\sum _{k=2}^{q+1} v_k' \ ,
\end{eqnarray*}
hence 
$$\dim E_u(\rh _j)\ge m_j=d_{2j-1}+1\ ,\ j\ge 3, \ \dim F_u(\si _k)\ge \ell _k=d_{2k}+1\ ,\ k\ge 2\ .$$
On the other hand, we know that
\beno {\rm rk} (H_{u^\e}) +{\rm rk}(K_{u^\e})&=&(2d_1+1)+1+1+(2d_2+1)+\dots +(2d_{2q}+1)\\
&=&[2(d_1+1)+1]+\sum _{r=2}^{2q}(2d_{r}+1),
\eeno
and, consequently, ${\rm rk} (H_{u}) +{\rm rk}(K_{u})$ is not bigger than the right hand side. In order to conclude that $u\in \mathcal V_{(d_1+1,d_2,\dots ,d_n)}$, it therefore remains to prove that 
$$\dim E_u(\rh )\ge m_1+1=d_1+2\ .$$
We use the explicit formulae obtained in section \ref{section explicit}, which read
$$u^\e _j(z)= \Psi _{2j-1}(z)\frac{\det \mathscr C_j^\e (z)}{\det \mathscr C^\e(z)}\ ,$$ where 
 $\mathscr C^\e (z)$ denotes the matrix  $(c_{k\ell }^\e(z))_{1\le k,\ell \le q+1}$, with 
 $$c_{jk}^\e (z)=\frac{\rho _j-\sigma _kz\Psi _{2j-1}(z)\Psi _{2k}(z)}{\rho _j^2-\sigma _k^2}\ ,\ \rho _1=\rho +\e ,\sigma _1=\rho ,\rho _2=\rho -    \e \ ,$$ 
 and $\mathscr C^\e _j(z)$ denotes the matrix deduced from $\mathscr C^\e (z)$ by replacing the line $j$ by the line $(1,\dots ,1)$.
Notice that elements $c _{11}^\e (z)$ and $c_{21}^\e (z)$ in formulae (\ref{c}) are of order $\e ^{-1}$, hence we compute 
\beno
\lim _{\e \rightarrow 0}2\e \det \mathscr C^\e (z)
&=& \left|\begin{array}{lll}1-z\expo_{-i\t _1}\Psi_1& \displaystyle {\frac{\rh -z\sigma _2\Psi_1 \Psi_{4}}{\rh ^2-\si _2 ^2}}& \dots \\
\displaystyle {-(1-z\expo _{-i(\t _1+\f _2)}) }&   \displaystyle {\frac{\rh -z\sigma _2\expo _{-i\f _2} \Psi_{4}}{\rh ^2-\si _2 ^2}}& \dots \\
0 & \dots &\dots  \end{array}\right|\ .
\eeno 
We compute this determinant by using the following formula
$$\left \vert  \begin{array}{llll} a_{11} & a_{12} &\dots & a_{1N}\\
a_{21} &a_{22}& \dots &a_{2N}\\
0 & a_{32} & \dots &a_{3N}\\       
0 & \dots & \dots &\dots\\ 
\vdots &\vdots &\vdots &\vdots \\
0 & \dots & \dots &\dots 
\end{array}      \right \vert =\left \vert  \begin{array}{lll} b_2  &\dots & b_{N}\\
 a_{32} & \dots &a_{3N}\\       
  \dots & \dots &\dots\\ 
\vdots &\vdots &\vdots \\
 \dots & \dots &\dots 
\end{array}      \right \vert \ ,\ b_k:=a_{11}a_{2k}-a_{21}a_{1k}\ ,\ k\ge 2\ .$$
Observe that
\begin{eqnarray*}
b_k&=&(1-z\expo_{-i\t _1}\Psi_1) {\frac{\rh -z\sigma _k\expo_{-i\f_2} \Psi_{2k}}{\rh ^2-\si _k ^2}}+ (1-z\expo _{-i(\t _1+\f _2)} ){\frac{\rh -z\sigma _k\Psi_1 \Psi_{2k}}{\rh ^2-\si _k ^2}}\\
&=&2(1-q(z)z)\left ({\frac{\rh -z\sigma _k \tilde \Psi_1  \Psi_{2k}}{\rh ^2-\si _k ^2}}\right)
\end{eqnarray*}
where 
$$q(z)=\expo_{-i\t_1}\frac{   \Psi_1(z)+\expo_{-i\f_2}}{ 2}$$ and 
$$\tilde \Psi_1(z)=\frac{z\expo_{-i(\pi+\t_1+\f_2)}\Psi _1(z)+\frac {\Psi _1(z)+\expo_{-i\f _2}}2}{1-q(z)z}.$$
We obtain 
$$ \lim _{\e \rightarrow 0}2\e \det \mathscr C^\e (z)=2(1-q(z)z)\det((\tilde c_{jk})_{2\le j, k\le q+1})$$
where, for $k\ge 2$, $j \ge 3$,
\beno
\tilde c_{2k}&=& \frac{\rh -z\si _k \tilde \Psi_1  (z)\Psi_{2k}(z)}{ (\rho ^2-\si _k^2)}\\
\tilde c_{jk}&=&c_{jk}= \frac{\rh _j -z\sigma _k \Psi_{2j -1}(z)\Psi_{2k}(z)}{\rh _j ^2-\si _k ^2}\ .
 \eeno
 
We know that $\Psi_1$ is a Blaschke product of degree $m_1-1$. Let us verify that it is possible to choose $\varphi_2$ so that $ \tilde \Psi_1  $ is a Blaschke product of degree $m_1$.  We first claim that it is possible to choose $\varphi_2$ so that $1-q(z)z\neq 0$ for $|z|\le 1$. Assume $1-q(z)z=0$.
Then
\begin{equation}\label{chi1}
1=\frac 12\left (\expo_{-i\t_1}\Psi_1(z)z+{\rm e}^{-i(\f_2+\t_1)}z\right )\ .
\end{equation}
First notice that this clearly imposes $\vert z\vert =1$. Furthermore, this implies  equality in the Minkowski inequality, therefore there exists $\lambda>0$ so that $\Psi_1(z)=\lambda{\rm e}^{-i\f_2}$ and, eventually, that $\Psi_1(z)={\rm e}^{-i\f_2}$ since $|\Psi_1(z)|=1$. Inserting this in equation (\ref{chi1}) gives $z={\rm e}^{i(\f_2+\t_1)}$ so that $\Psi_1({\rm e}^{i(\f_2+\t_1)})={\rm e}^{-i\f_2}$. If this equality holds true for any choice of $\f_2$, by analytic continuation inside the unit disc, we would have
 $$\Psi_1(z)=\frac{{\rm e}^{i\t_1}}z$$ which is not possible since $\Psi _1$ is  a holomorphic function in the unit disc. Hence, one can choose $\varphi_2$ in order to have $1-q(z)z\neq 0$ for any $|z|\le 1$. It implies that $ \tilde \Psi_1  $ is a holomorphic rational function in the unit disc. Moreover, if $\vert z\vert =1$, 
 $$\tilde \Psi _1(z)=\expo_{-i(\pi+\t_1+\f_2)}\Psi _1(z)\frac{z-\overline {q(z)}}{1-zq(z)}\ ,$$
 hence $\vert \tilde \Psi _1(z)\vert =1$. We conclude that $\tilde \Psi _1$  is a Blaschke product. Finally, its  degree is $\deg(\Psi _1)+1=m_1$.
\s
 Next, we perform the same calculation with the numerator of $\det \mathscr C_j^\e(z)$ for $j=1,2$.
 We compute
 \beno &&\lim _{\e \rightarrow 0} 2\e\det \mathscr C_1^\e(z)=\\
&=& \left|\begin{array}{ccc}0& 1 & \dots 1 \\
\displaystyle {-(1-z\expo _{-i(\t _1+\f _2)} )}&  \displaystyle {\frac{\rh -z\sigma _2\expo _{-i\f _2} \Psi_{4}}{\rh ^2-\si _2 ^2}}& \dots \\
0 & \dots & \dots \end{array}\right|\\
&=& (1-z\expo _{-i(\t _1+\f _2)} )   \det\left(\begin{array}{ccc} 1&\dots &1 \\
c_{j2} &\dots & c_{j\, q+1}\end{array}\right)_{j\ge 3}\eeno
and 
 \beno &&\lim _{\e \rightarrow 0}2\e \det \mathscr C_2^\e(z)=\\
&=& \left\vert \begin{array}{ccc} \displaystyle {1-z\expo _{-i\t _1} \Psi_1}& \displaystyle {\frac{\rh -z\sigma _2\Psi_1\Psi_{4}}{\rh ^2-\si _2 ^2}} & \dots \\
0 &  1& \dots 1 \\
0& \displaystyle {\frac{\rh _3-z\sigma _2\Psi_5\Psi_{4}}{\rh _3^2-\si _2 ^2}} & \dots  \end{array}\right \vert \\
&=& (1-z\expo _{-i\t _1} \Psi_1)\det\left(\begin{array}{ccc} 1&\dots &1 \\
c_{j2} &\dots & c_{j\, q+1}\end{array}\right)_{j\ge 3}\eeno

Hence we have, for the weak convergence in $L^2_+$, 
 \beno v_1(z)&:=&\lim _{\e \rightarrow 0}u_1^\e(z)=\Psi _1(z)\frac{(1-z\expo _{-i(\t _1+\f _2)} )}{2(1-q(z)z)}\cdot\frac{ \det\left(\begin{array}{ccc} 1&\dots &1 \\
c_{j2} &\dots & c_{j\, q+1}\end{array}\right)_{j\ge 3}}{
\det \left(( \tilde c_{jk})_{2\le j,k\le q+1}\right )}\\
 v_2(z) &:=&\lim _{\e \rightarrow 0}u_2^\e(z)=   {\rm e}^{-i\varphi_2}\frac{{(1- z\expo_{-i\t _1}\Psi_1(z))}}{2(1-q(z)z)}\cdot\frac{ \det\left(\begin{array}{ccc} 1&\dots &1 \\
c_{j2} &\dots & c_{j\, q+1}\end{array}\right)_{j\ge 3}}{\det \left( ( \tilde c_{jk})_{2\le j,k \le q+1}\right )}\ .
 \eeno
Furthermore,  if $D_1$ denotes the normalized denominator of $\Psi_1$, we have
 \begin{eqnarray*}
  H_{u^\e}^2\left(\frac {z^a}{D_1(z)}\frac{u_1^\e}{\Psi_1}\right)&=&(\rh+\e)^2\frac {z^a}{D_1(z)}\frac{u_1^\e}{\Psi_1}, \ 0\le a \le m_1 -1,\\ 
  H_{u^\e}^2(u_2^\e)&=&(\rh-   \e)^2u_2^\e,\\
\end{eqnarray*}
  Passing to the limit in these identities as $\e $ goes to $0$,  we get
  \begin{eqnarray*}
  H_{u}^2\left(\frac {z^a}{D_1(z)}\frac{v_1}{\Psi_1}\right)&=&\rh^2\frac {z^a}{D_1(z)}\frac{v_1}{\Psi_1}, \  0\le a\le m_1-1, \\ 
  H_{u}^2(v_2)&=&\rh^2v_2\ .
\end{eqnarray*}

  It remains to prove that the dimension of the vector space generated by $$v_2\ ,\ \frac {z^a}{D_1(z)}\frac{v_1}{\Psi_1}, \ 0\le a\le m_1-1,  $$  is $m_1+1$. From the expressions of $v_1$ and $v_2$, it is equivalent to prove that  the dimension of the vector space spanned by the functions $$\ (1-{\rm e}^{-i\t_1}z\Psi_1(z))\ ,\ \frac{z^a}{D_1(z)}(1-{\rm e}^{-i(\f_2+\t_1)}z), \ 0\le a\le m_1-1,$$
  is $m_1+1$. We claim that our choice of $\f_2$ implies that this family is free. Indeed, assume that for some  coefficients $\lambda_a$, $0\le a \le m_1-1$, we have
  $$\sum_{a=0}^{m_1-1}\lambda_a\frac{z^a}{D_1(z)}=\frac{1-{\rm e}^{-i\t_1}z\Psi_1(z)}{1-{\rm e}^{-i(\f_2+\t_1)}z}$$
  then, as the left hand side is a holomorphic function in $\overline \D$, it would imply $\Psi_1({\rm e}^{i(\f_2+\t_1)})={\rm e}^{-i\f_2}$ but $\f_2$ has been chosen so that this does not hold. \\
This completes the proof.
  \end{proof}

\subsection{The case $n$ odd.}

 The proof of the fact that  $\Phi _{d_1,\dots ,d_n}$ is one-to-one is the same as in the case $n$ even. One has to prove that $\Phi _{d_1,\dots ,d_n}$ is onto. We shall proceed by approximation from the case $n$ even. We define $q=\frac {n+1}2$.

\noindent Let 
$$\mathcal P=((\rho _1,\sigma _1, \dots ,\rho _q), (\Psi_r)_{1\le r\le n})$$
be an arbitrary element of $\mathcal S_{(d_1,\dots,d_n)}$. We look for $u\in \mathcal V _{(d_1,\dots,d_n)}$ such that
$\Phi _{d_1,\dots ,d_n}(u)=\mathcal P$. Consider, for every $\e $ such that $0<\e < \rho _q$,
$$\mathcal P_\e=((\rho _1,\sigma _1, \dots ,\rho _q,\e), ((\Psi_r)_{1\le r\le n}, 1))\in \mathcal S_{(d_1,\dots,d_n,0)}$$ 
 -- we take $\Psi_{2q}=1\in \mathcal B_0$. From Theorem \ref{mainfiniterank}, we get $u_\e\in \mathcal V_{(d_1,\dots,d_{n+1})}$ such that
$\Phi (u_\e)=\mathcal P_\e$. 
As before, we can prove by a compactness argument that a subsequence of $u_\e$ has a limit $u\in \mathcal V _{(d_1,\dots,d_n)}$ as $\e$ tends to $0$ with $\Phi _{d_1,\dots ,d_n}(u)=\mathcal P$. We leave the details to the reader.
\end{proof}
\subsection{$\mathcal V _{(d_1,\dots,d_n)}$ is a manifold}
Let $d=n+2\sum _rd_r$. We consider the map
$$ \begin{array}{lll} \Theta : \mathcal S_{(d_1,\dots,d_n)}  &\longrightarrow  &\mathcal V(d)\\
\ \ \ \ \ ({\bf s},{\bf \Psi})&\longmapsto & u({\bf s},{\bf \Psi})\end{array} $$
This map is well defined and $C^\infty $ on $\mathcal S_{(d_1,\dots,d_n)}$. Moreover, from the previous section, it is a homeomorphism onto its range $\mathcal V _{(d_1,\dots,d_n)}$. In order to prove that $\mathcal V _{(d_1,\dots,d_n)}$ is a submanifold of $\mathcal V(d)$,  it is enough to check that the differential of $\Theta $ is injective at every point. 
From Lemma \ref{Phismooth}, near every point $u_0\in \mathcal V _{(d_1,\dots,d_n)}$, there exists a smooth function $\tilde \Phi _n$, defined on a neighborhood $V$ on $u_0$ in $\mathcal V(d)$, such that $\tilde \Phi _n$  
coincides with $\Phi _{d_1,\dots ,d_n}$ on $V\cap \mathcal V _{(d_1,\dots,d_n)}$. Consequently, $\tilde \Phi _n\circ \Theta $ is the identity on a neighborhood of
$\mathcal P_0:= \Phi _{d_1,\dots ,d_n}(u_0)$. In particular, the differential of $\Theta $ at $ \mathcal P_0$ is injective.
Finally, the dimension of $\mathcal V _{(d_1,\dots,d_n)}$ is 
 $$\dim \mathcal S_{d_1,\dots ,d_n}=\dim \Omega _n +\sum _{r=1}^n \dim \mathcal B_{d_r}=2n+2\sum_{r=1}^nd_r ,$$
 using Proposition \ref{Od}.

\section{Extension to the Hilbert-Schmidt class}\label{sectionHS}

In this section, we consider infinite rank Hankel operators in the Hilbert-Schmidt class\index{Hilbert-Schmidt}.
We set $$\mathcal V^{(2)}_{(d_r)_{r\ge 1}}:=\Phi^{-1}(\mathcal S_{(d_r)_{r\ge 1}}^{(2)}) .$$
 \begin{theorem}\label{HilbertSchmidt}
The mapping
\begin{eqnarray*}
\Phi : \mathcal V^{(2)}_{(d_r)_{r\ge 1}} &\longrightarrow & \mathcal S _{(d_r)_{r\ge 1}}^{(2)} \\
u&\longmapsto& ((s_r)_{r\ge 1}, (\Psi_r)_{ r\ge 1})
\end{eqnarray*}
is a homeomorphism. Furthermore, if $(s_r,\Psi _r)_{r\ge 1}\in \mathcal S_\infty^{(2)} $, then
its preimage $u$ by $\Phi $ is given by
\begin{equation}\label{luminybis}
u(z)=\lim _{q\rightarrow \infty} u_q(z)
\end{equation}
where $$u_q:=u((s_1,\dots, s_{2q}),(\Psi_1,\dots,\Psi_{2q})).$$ 
\end{theorem}

\begin{proof}
The fact that $\Phi $ is one-to-one follows from an explicit formula analogous to the one obtained in the finite rank case,
see section \ref{section explicit}. However, in this infinite rank situation, we have to proceed slightly differently, in order to deal with the continuity of infinite rank matrices
on appropriate $\ell ^2$ spaces. 
\s
Indeed, we still have
$$u=\sum_{j=1}^\infty \Psi_{2j-1} h_j$$
where $\mathcal H(z):=(h_j(z))_{j\ge 1}$ satisfies  the following  infinite dimensional system, for every $z\in \D $, 
\begin{equation}\label{eqHinfini}
\mathcal H(z)=\mathcal F(z)+z\mathcal D(z)\mathcal H(z)
\end{equation}
with
\begin{eqnarray*}
\mathcal F(z)&:=&\left (\frac{\tau _j^2}{\rho_j}\right )_{j\ge 1}\ ,\\
\mathcal D(z)&:=&\left (\frac{\tau _j^2}{\rho_j}\sum _{k=1}^\infty \frac {\kappa _k^2\sigma _k\Psi _{2k}(z)\Psi _{2\ell -1}(z)}{(\rho _j^2-\sigma _k^2)(\rho _\ell ^2-\sigma _k^2)}\right )_{j,\ell\ge 1}\ .
\end{eqnarray*}
In order to derive this system,  just  write 
\beno h_j(z)&=&\frac{1}{\rho _j}H_u(u_j)(z)=\frac{\tau _j^2}{\rho _j}+\frac{1}{\rho _j}SK_u(u_j)(z)\\
&=& \frac{\tau _j^2}{\rho _j}\left (1+z\sum _{k=1}^\infty \frac{K_u(u'_k)(z)}{\rho _j^2-\sigma _k^2}\right )
=\frac{\tau _j^2}{\rho _j}\left (1+ z\sum _{k=1}^\infty \frac{\sigma _k\Psi _{2k}(z)u'_k(z)}{\rho _j^2-\sigma _k^2}\right )\\
&=&\frac{\tau _j^2}{\rho _j}\left (1+ z\sum _{k=1}^\infty \sum _{\ell =1}^\infty \frac{\kappa _k^2\sigma _k\Psi _{2k}(z)\Psi _{2\ell -1}(z)h_\ell (z)}{(\rho _j^2-\sigma _k^2)(\rho _\ell ^2-\sigma _k^2)}\right )\ .
\eeno
Notice that  the coefficients of the infinite matrix $\mathcal D (z)$ depend holomorphically on $z\in \D$. 
We are going to prove that, for every $z\in \D $, $\mathcal D(z)$ defines a contraction  on the space $\ell ^2_\tau $ of sequences
$(v_j)_{j\ge 1}$ satisfying
$$\sum _{j=1}^\infty \frac{\vert v_j\vert ^2}{\tau _j^2}<\infty \ .$$
From the maximum principle, we may  assume that $z$ belongs to the unit circle. Then $z$ and $\Psi _r(z)$ have modulus $1$. We then compute $\mathcal D(z)\mathcal D(z)^*$, where the adjoint is taken  for the inner product associated to $\ell ^2_\tau $. We get, using identities (\ref{sommedoubletau}), (\ref{sommesimplekappa}) and (\ref{sommedoublekappa}),
\beno 
\left [\mathcal D(z)\mathcal D(z)^*\right ] _{jn}&=&\frac{\tau _j^2}{\rho_j\rho _n}\sum _{k,\ell ,m}
\frac {\kappa _k^2\sigma _k\Psi _{2k}(z)\tau _\ell ^2\kappa _m ^2\si _m\overline {\Psi _{2m}(z)}}{(\rho _j^2-\sigma _k^2)(\rho _\ell ^2-\sigma _k^2)(\rho _n^2-\sigma _m^2)(\rho _\ell ^2-\sigma _m^2)}\\
&=& \frac{\tau _j^2}{\rho_j\rho _n}\sum _{k}\frac{\kappa _k^2\si _k^2}{(\rho _j^2-\sigma _k^2)(\rho _n^2-\sigma _k^2)}\\
&=& -\frac{\tau _j^2}{\rho_j\rho _n}+\delta _{jn}\ .
\eeno
Since, from the identity (\ref{sommenu}), 
$$\sum _{j=1}^\infty \frac {\tau _j^2}{\rh _j^2}\le 1\ ,$$
we conclude that $\mathcal D(z)\mathcal D(z)^*\le I$ on $\ell ^2_\tau $, and consequently that
$$\Vert \mathcal D(z)\Vert _{\ell ^2_\tau \rightarrow \ell ^2_\tau }\le 1\ .$$
From the Cauchy inequalities, this implies
$$\Vert \mathcal D^{(n)}(0)\Vert _{\ell ^2_\tau \rightarrow \ell ^2_\tau }\le n!\ .$$
Coming back to equation (\ref{eqHinfini}), we observe that $\mathcal H(0)=\mathcal F(0)\in \ell ^2_\tau $, and that, for every $n\ge 0$,
$$\mathcal H^{(n+1)}(0)=(n+1)\sum _{p=0}^n \begin{pmatrix}n\\p\end{pmatrix}\mathcal D^{(p)}(0)\mathcal H^{(n-p)}(0)\ .$$
By induction on $n$, this determines $\mathcal H^{(n)}(0)\in \ell ^2_\tau $, whence the injectivity of $\Phi $.
\s
Next, we prove that $\Phi $ is onto. We pick an element $$\mathcal P\in\mathcal S_{(d_r)}^{(2)}$$ and we construct $u\in H^{1/2}_+$ so that $\Phi(u)=\mathcal P$.
Set $$\mathcal P=((\rho _1,\sigma _1, \rho _2,\dots), (\Psi _r)_{r\ge 1})$$
and consider, for any integer $N$,
$$\mathcal P_N:=((\rho _1,\sigma _1, \dots,\rho_N,\sigma_N), (\Psi_r)_{1\le r\le 2N})$$ in 
$$\mathcal S_{(d_1,\dots,d_{2N})}\ .$$
From Theorem \ref{mainfiniterank}, there exists $u_N\in \mathcal V _{(d_1,\dots,d_{2N})}$ with $\Phi(u_N)=\mathcal P_N$. As $${\rm Tr}(H_{u_N}^2)=\sum _{r=1}^{2N} d_{2r} s_{2r}^2+\sum _{j=1}^N (d_{2j-1}+1) s _{2j-1}^2\le \sum _{r=1}^\infty  d_r s_r^2+\sum _{j=1}^\infty s _{2j-1}^2<\infty\ ,$$
the sequence $(u_N)$ is bounded in $H^{1/2}_+$. Hence, there exists a subsequence converging weakly  to some $u$ in $H^{1/2}_+$. In particular, one may assume that $(u_N)$ converges strongly to $u$ in $L^2_+$.

Since $\Vert H_{u_N}\Vert =\rh _1$ is bounded, we infer the strong convergence of  operators,
$$\forall h\in L^2_+, H_{u_N}(h)\td_p,\infty H_u(h)\ .$$
We now observe that if $\rho^2$ is an eigenvalue of $H_{u_N}^2$ of multiplicity $m$ then $\rho^2$ is an eigenvalue of $H_u^2$ of multiplicity at most  $m$.
Let $(e_N^{(l)})_{1\le l\le m}$ be an orthonormal family of eigenvectors of $H_{u_N}^2$ associated to the eigenvector $\rho^2$. Let $h$ be in $L^2_+$ and write
$$
h=\sum_{l=1}^m (h\vert  e_N^{(l)})e_N^{(l)}+h_{0,N}$$ where $h_{0,N}$ is the orthogonal projection of $h$ on the orthogonal complement of $E_{u_N}(\rho)$ so that 
\begin{eqnarray*}
\Vert (H_{u_N}^2-\rho^2 I)h\Vert^2&=&\Vert (H_{u_N}^2-\rho^2 I)h_{0,N}\Vert^2\\
&\ge & d_{\rho^2}\Vert h_{0,N}\Vert ^2=d_{\rho^2}(\Vert h\Vert^2-\sum_{l=1}^m \vert (h\vert e_N^{(l)})|^2)\ ,
 \end{eqnarray*}
here $d_{\rho^2}$ denotes the distance to the other eigenvalues of $H_{u_N}^2$.
 By taking the limit as $N$ tends to $\infty$ one gets
 $$\Vert (H_{u}^2-\rho^2 I)h\Vert^2\ge d_{\rho^2}(\Vert h\Vert^2-\sum_{l=1}^m|(h\vert e^{(l)})|^2)$$
 where $e^{(l)}$ denotes a weak limit of $e_N^{(l)}$.  
 Assume now that the dimension of $E_u(\rho)$ is larger than $m+1$ then we could construct $h$ orthogonal to $(e^{(1)},\dots e^{(m)})$ with $H_u^2(h)=\rho^2h$, a contradiction. The same argument allows to obtain that if $\rho^2$ is not an eigenvalue of $H_{u_N}^2$, $\rho^2$ is not an eigenvalue of $H_u^2$.\\
 We now argue as in section \ref{closed} above. Let $u_{N,j}$ and $u'_{N,k}$ denote the orthogonal projections of $u_N$ respectively on $E_{u_N}(\rh _j)$ and on $F_{u_N}(\si _k)$ so that we have the orthogonal decompositions,
$$u_N=\sum_{j=1}^Nu_{N,j}=\sum_{k=1}^N u'_{N,k}\ .$$
As $(u_N)$ converges strongly in $L^2_+$,  $u_{N,j}$ and $u'_{N,k}$ converge in $L^2_+$ respectively to some $v^{(j)}$ and to some ${v'}^{(k)}$ with the identities
 $$\rh _jv_j=\Psi _{2j-1}H_u(v_j)\ ,\ H_u^2(v_j)=\rh _j^2v_j\ ,\ K_u(v'_k)=\si _k\Psi _{2k}v'_k\ ,\ K_u^2(v'_k)=\si _k^2v'_k\ .$$
 and  
 $$(u\vert v_j)=\tau _j^2\ ,\ (u\vert v'_k)=\kappa _k^2\ .$$
 This already implies that $v_j$, $v'_k$ are not $0$, and hence, in view of Lemmas \ref{crucialHuGeneral} and \ref{crucialKuGeneral}, that 
 $$\dim E_u(\rh _j)= m_j\ ,\ \dim F_u(\si _k)= \ell _k\ .$$
  We infer that $u\in \mathcal V_{(d_r)_{r\ge 1}}^{(2)}$ and that $\rho _j=s_{2j-1}(u)\ ,\ \si _k=s_{2k}(u)\ .$ It remains to identify $v_j$ with the orthogonal projection $u_j$  of $u$ onto $E_u(\rh _j)$,
 and $v'_k$ with the orthogonal projection $u'_k$ of $u$ onto $F_u(\si _k)$. The strategy of passing to the limit, as $N$ tends to infinity,  in the decompositions
 $$u_N=\sum _{j=1}^N u_{N,j}=\sum _{k=1}^Nu'_{N,k}\ $$
 is not easy to apply because of infinite sums. Hence we argue as follows. From the identity
 $$\Vert u_{N,j}\Vert^2=(u_N\vert u_{N,j})$$
 we get
 $$\Vert v_j\Vert ^2= (u\vert v_j)=(u_j\vert v_j)=\tau_j^2=\Vert u_j\Vert^2 ,$$
  hence 
$$\Vert v_j-u_j\Vert^2=0\ .$$
Similarly, $v'_k=u'_k$. This completes the proof of the surjectivity, and of the explicit formula (\ref{luminybis}).
Notice that the convergence is strong in $H^{1/2}$ since the norm of $u_N$ tends to the norm of $u$. 
\s
The continuity of $\Phi $ follows as in section \ref{continuity}. As for the continuity of $\Phi ^{-1}$, we argue exactly as for surjectivity
above.

\end{proof}
\section{Extension to compact Hankel operators}\label{compact}

The mapping $\Phi$ may be extended to $VMO_+$ which corresponds to the set of symbols of compact Hankel operators. Namely, let $\Omega_\infty$\index{$\Omega_\infty$} be the set of sequences $(s_r)_{r\ge 1}$ such that $$s_1>s_2>\dots >s_n\to 0 \ .$$ Given an arbitrary sequence $(d_r)_{r\ge 1}$ of nonnegative integers, we set 
$$\mathcal V_{(d_r)_{r\ge 1}}:=\Phi ^{-1}(\Omega_\infty\times \prod_{r=1}^\infty \mathcal B_{d_r} )\ . $$ 
 \begin{theorem}\label{mainCompact}
The mapping
\begin{eqnarray*}
\Phi : \mathcal V_{(d_r)_{r\ge 1}} &\longrightarrow &\Omega_\infty\times \prod_{r=1}^\infty \mathcal B_{d_r} \\
u&\longmapsto& ((s_r)_{r\ge 1}, (\Psi_r)_{ r\ge 1})
\end{eqnarray*}
is a homeomorphism.
\end{theorem}
\begin{proof}
The proof is the same as before except for the argument of surjectivity in which the boundedness of the sequence $(u_N)$ in $H^{1/2}_+$ does not hold anymore. However, the strong convergence in $L^2_+$ may be established. The proof is along the same lines as the one developed for Proposition 2 in \cite{GG3}, and is based on the Adamyan-Arov-Krein (AAK)theorem \cite{AAK}, \cite{Pe2}. Let us recall the argument. 
\s
First we recall that the  AAK theorem  \index{AAK theorem}states that the $(p+1)$-th singular value\index{singular value} of a Hankel operator, as the distance of this operator to operators of rank at most $p$,  is exactly achieved by some Hankel operator of rank at most $p$, hence, with a rational symbol. We refer to part (2) of the theorem \ref{AAK}. We set, for every $m\ge 1$,
$$p_m=m+\sum _{r\le 2m} d_r\ .$$
With the notation of , one easily checks that, for every $m$,
$$\lambda_{p_{m-1}}(u)>\lambda_{p_m}(u)=\rho _{m+1}(u)\ .$$
By part (1) of the AAK theorem \ref{AAK}, for every $N$ and every $m=1,\dots, N$, there exists a rational symbol $u_{N}^{(m)}$, defining a Hankel operator of rank $p_m$,  namely $u_{N}^{(m)}\in \mathcal V(2p_m)\cup {\mathcal V}(2p_m-1)$, such that 
$$\Vert H_{u_N}-H_{u_{N}^{(m)}}\Vert =\rho_{m+1}(u_N)=\rho _{m+1}.$$
In particular, we get
$$\Vert u_N-u_{N}^{(m)}\Vert_{L^2}\le \rho_{m+1}.$$
On the other hand, one has $$\Vert H_{u_{N}^{(m)}}\Vert \ge \frac{1}{\sqrt {p_m}}(Tr(H^2_{u_{N}^{(m)}}))^{1/2}\ge \frac 1{\sqrt {p_m}}\Vert u_{N}^{(m)}\Vert_{H^{1/2}_+}.$$  Hence, for fixed $m$, the sequence $(u_{N}^{(m)})_N$ is bounded in $H^{1/2}_+$.
Our aim is to prove that the sequence $(u_N)$ is precompact in $L^2_+$.
We show that, for any $\varepsilon>0$ there exists a finite sequence $v_k\in L^2_+$, $1\le k\le M$ so that $$\{u_N\}_N\subset\bigcup_{k=1}^M B_{L^2_+}(v_k,\varepsilon).\ $$
Let $m$ be fixed such that $\rho_{m+1}\le \varepsilon/2.$ Since the sequence $(u_{N}^{(m)})_N$ is uniformly bounded in $H^{1/2}_+$,  it is precompact in $L^2_+$, hence there exists $v_k\in L^2_+$, $1\le k\le M$, such that $$\{u_{N}^{(m)}\}_N\subset\bigcup_{k=1}^M B_{L^2_+}(v_k,\varepsilon/2)\ .$$ Then, for every $N$ there exists some $k$ such that
$$\Vert u_N-v_k\Vert_{L^2}\le \rho_{m+1}+\Vert u_{N}^{(m)}-v_k\Vert_{L^2}\le \varepsilon.$$
Therefore $\{ u_N \}$ is precompact in $L^2_+$ and, since $u_N$ converges weakly to $u$, it converges strongly to $u$ in $L^2_+$. The proof ends as in the Hilbert-Schmidt case.
\s
The continuity of $\Phi $ follows as in section \ref{continuity}. As for the continuity of $\Phi ^{-1}$, we argue exactly as for surjectivity
above, except that we have to prove the convergence of $u_N$ to $u$ in $VMO_+$. This can be achieved exactly as in the proof of Proposition 2 of \cite{GG3} : the Adamyan-Arov-Krein theorem allows to reduce to the following statement : if $w_N\in \mathcal V(2p)\cup \mathcal V(2p-1)$ 
strongly converges to $w\in \mathcal V(2p)\cup \mathcal V(2p-1)$, then the convergence takes place in $VMO $ --- in fact in $C^\infty $. See Lemma 3 of \cite{GG3}.
\end{proof}

\chapter{The Szeg\H{o} dynamics}\label{chapter Szego dynamics}

This chapter is devoted to the connection between the nonlinear Fourier transform\index{Non linear Fourier transform} and the Szeg\H{o} dynamics. In section \ref{szegoflow}, we show that the Szeg\H{o} evolution has a very simple translation in terms of the nonlinear Fourier transform. Geometric aspects of this evolution law will be discussed in more detail in chapter \ref{chapter geometry}. Then we revisit the classification of traveling waves of the cubic Szeg\H{o} equation.
The last section is devoted to the proof of the almost periodicity of $H^{1/2}$ solutions.

\section{Evolution under the cubic Szeg\H{o} flow} \label{szegoflow}
\subsection{The theorem}
In this section, we prove the following result.
\begin{theorem}\label{evolszegotexte}
Let $u_0\in H^{1/2}_+$ with $$\Phi(u_0)=((s_r),(\Psi_r)).$$
The solution of $$i\partial_t u=\Pi(|u|^2u),\; u(0)=u_0$$ is characterized by
$$\Phi(u(t))=((s_r),(\expo_{i(-1)^rs_r^2t}\Psi _r))\ .$$
\end{theorem}
\begin{remark} It is in fact possible to define the  flow of the cubic Szeg\H{o} equation on $BMO_+=BMO(\S^1 )\cap L^2_+$, see \cite{GK}.  The above theorem then extends to the case of  an initial datum $u_0$ in $VMO_+$ .
\end{remark}
\begin{proof}
In view of the continuity of the flow map on $H^{1/2}_+$,  see \cite{GG1}, we may assume that $H_{u_0}$ is of finite rank. Let $u$  be the corresponding solution of the cubic Szeg\H{o} equation. 
Let $\rho$ be a singular value\index{singular value} of $H_u$ in $\Sigma_H(u)$ such that $m:=\dim E_u(\rho)=\dim F_u(\rho)+1$ and denote by $u_\rho$ the orthogonal projection of $u$ on $E_u(\rho)$.
Hence, $u_\rho=1\!{\rm l}_{\{ \rho^2\}}(H_u^2)(u)$. Let us differentiate this equation with respect to time. 
Recall \cite{GG1}, \cite{GG4} that 
\begin{equation}\label{Laxpair}
\frac{dH_u}{dt}=[B_u,H_u] \text{ with }B_u=\frac i2H_u^2-i T_{|u|^2}\ .
\end{equation}
Here we recall that $T_b $ denotes the Toeplitz operator of symbol $b$,
$$
T_b(h)=\Pi (bh)\ ,\ h\in L^2_+\ ,\ b\in L^\infty \ .
$$
Equation (\ref{Laxpair}) implies, for every Borel function $f$,
$$\frac{df(H_u^2)}{dt}=-i[T_{\vert u\vert ^2},f(H_u^2)]\ .$$
We get from this Lax pair structure 
\begin{eqnarray*}
\frac{d u_\rho}{dt}&=& -i[T_{\vert u\vert ^2},1\!{\rm l}_{\{ \rho^2\}}(H_u^2)](u)+1\!{\rm l}_{\{ \rho^2\}}(H_u^2)\left (\frac{du}{dt}\right )\\
&=&  -i[T_{\vert u\vert ^2},1\!{\rm l}_{\{ \rho^2\}}(H_u^2)](u)+1\!{\rm l}_{\{ \rho^2\}}(H_u^2)\left (-iT_{\vert u\vert ^2}u\right )\ ,
\end{eqnarray*}
and eventually
\begin{equation}\label{deriveurho}
\frac{d u_\rho}{dt}=-iT_{\vert u\vert ^2}u_\rh \ .
\end{equation}

On the other hand, differentiating the equation
$$\rho u_\rho=\Psi H_u(u_\rho)$$ one obtains
$$
\rho \frac{d u_\rho}{dt}=\dot \Psi H_u(u_\rho)+\Psi\left([B_u,H_u](u_\rho)+H_u\left( \frac{d u_\rho}{dt}\right)\right )$$
Hence, using the expression (\ref{deriveurho}), we get
$$-i\rho T_{\vert u\vert ^2}(u_\rho)=\dot \Psi H_u(u_\rh )+\Psi \left (-iT_{\vert u\vert ^2}H_u(u_\rh ) +i\rh ^2H_u(u_\rh )\right )\ ,$$
hence
$$-i[T_{\vert u\vert ^2},\Psi ] H_u(u_\rh )=(\dot \Psi +i\rh ^2\Psi )H_u(u_\rh )\ .$$
We claim that the left hand side of this equality is zero. Assume this claim proved, we get, as $H_u(u_\rho)$ is not identically zero, that 
$\dot \Psi +i\rh ^2\Psi =0$, whence 
$$\Psi (t)=\expo_{-it\rh ^2}\Psi (0)\ .$$

It remains to prove the claim.
We first prove that, for any $p\in\D$ such that $\chi_p$ is a factor of $\chi$,
$$[T_{\vert u\vert ^2},\chi_p](e)=0$$
  for any $e\in E_u(\rho)$ such that $\chi _pe\in E_u(\rh )$.
Recall that
$$\chi _p(z)=\frac{z-p}{1-\overline pz}\ .$$
For any $L^2$ function $f$, $$\Pi(\chi_pf)-\chi_p\Pi(f)=K_{\chi _p}(g)=(1-\vert p\vert ^2)H_{1/(1-\overline pz)}(g)\ ,$$
where $\overline {(I-\Pi )f}=Sg\ .$ Consequently, the range of $[\Pi ,\chi _p]$ is one dimensional, directed by  $\frac 1{1-\overline pz}$. 
In particular, $[T_{\vert u\vert ^2},\chi_p](e)$ is proportional to $\frac 1{1-\overline pz}$. On the other hand,
\begin{eqnarray*}
([T_{\vert u\vert ^2},\chi_p](e)\vert 1)&=&((T_{|u|^2}(\chi_p e)-\chi_pT_{|u|^2}(e))\vert 1)\\
&=&(\chi_p(e)\vert H_u^2(1))-(\chi_p\vert 1)(e\vert H_u^2(1))\\
&=&(H_u^2(\chi_p(e))\vert 1)-(\chi_p\vert 1)(H_u^2(e)\vert 1)=0\ .
\end{eqnarray*}
This proves that $[T_{\vert u\vert ^2},\chi_p](e)=0$.

For the general case, we write $\Psi=\expo_{-i\psi}\chi_{p_1}\dots\chi_{p_{m-1}}$ and 
$$
[T_{\vert u\vert ^2},\Psi]H_u(u_\rho)=\expo_{-i\psi}\sum_{j=1}^{m-1}\prod_{k=1}^{j-1}\chi_{p_k}[T_{\vert u\vert ^2},\chi_{p_j}]\prod_{k=j+1}^{m-1}\chi_{p_k}H_u(u_\rho)=0\ .
$$
It remains to consider the evolution of the $\Psi _{2k}$'s. Let $\sigma$ be a singular value\index{singular value} of $K_u$ in $\Sigma_K(u)$ such that $\dim F_u(\sigma)=\dim E_u(\sigma)+1$ and denote by $u'_\sigma$ the orthogonal projection of $u$ onto $F_u(\si )$. Recall \cite{GG4} that
$$\frac{dK_u}{dt}=[C_u,K_u]\text{ with }C_u=\frac i2K_u^2-i T_{|u|^2} \ .$$
As before, we compute the derivative in time of $u'_\sigma=1\!{\rm l}_{\{\sigma^2\}}(K_u^2)(u)$, and get
\begin{equation}\label{deriveuprimesigma}
 \frac{d u'_\sigma}{dt}=-iT_{\vert u\vert ^2}u'_\sigma\ .
\end{equation}
On the other hand, differentiating the equation
$$K_u(u'_\sigma)=\sigma \Psi u'_\sigma$$
one obtains
$$
-i[T_{\vert u\vert ^2}, \Psi ]u'_\si =(\dot \Psi -i\si ^2\Psi )u'_\si \ .
$$
As before, we prove that the left hand side of the latter identity is $0$, by checking that, for every factor $\chi _p$ of $\Psi $, for any $f\in F_u(\si )$
such that $\chi _pf\in F_u(\si )$, 
$$([T_{\vert u\vert ^2},\chi _p](f)\vert 1)=0\ .$$
The calculation leads to
\beno
([T_{\vert u\vert ^2},\chi _p](f)\vert 1)&=&(H_u^2(\chi _pf)-(\chi _p\vert 1)H_u^2(f)\vert 1)\\
&=&((\chi _p-(\chi _p\vert 1))f\vert u)(u\vert 1),
\eeno 
where we have used (\ref{Ku2}). Now $(\chi _p-(\chi _p\vert 1))f\in F_u(\si )$ is orthogonal to $1$, hence, from Proposition \ref{action}, it belongs to $E_u(\si )$, hence it is orthogonal to $u$. This completes the proof.

\end{proof}

\section{Application: traveling waves revisited}

As an application of Theorems \ref{mainfiniterank} and \ref{HilbertSchmidt} and of the previous section, we revisit the traveling waves of the cubic Szeg\H{o} equation. These are the solutions of the form
 $${u(t,e^{ix})=e^{-i\omega t}u_0(e^{i(x -ct)})\ ,\ \omega,c\in\R\ .}$$
 For $c=0$, it is easy to see \cite{GG1} that this condition for $u_0\in H^{1/2}_+$ corresponds to finite Blaschke product. The problem of characterizing traveling waves with  $c\ne 0$ is more delicate, and was solved in \cite{GG1} by the following result.

\begin{theo} \cite{GG1}
A function $u$ in $H^{1/2}_+$ is a traveling wave\index{traveling wave} with $c\neq 0$ and $\omega\in \R$ if and only if  there exist non negative integers $\ell$ and $N$, $0\le \ell\le N-1$,  $\alpha\in \R$  and a complex number $p\in\C$ with $0<|p|<1$ so that
$$u(z)=\frac{\alpha z^\ell}{1-pz^N}$$
\end{theo}

Here we give an elementary proof of this theorem.
\begin{proof}
The idea is to keep track of the Blaschke products associated to $u$ through  the following unitary transform on $L^2(\S^1 )$,
$$\tau _\alpha f(\expo_{ix}):=f(\expo_{i(x-\alpha)})\ ,\ \alpha \in \R .$$
Since $\tau _\alpha $ commutes to $\Pi $, notice that 
$$\tau _\alpha (H_u(h))=H_{\tau _\alpha (u)}(\tau _\alpha (h))\ .$$
Consequently, $\tau _\alpha $ sends $E_u(\rh )$ onto $E_{\tau _\alpha (u)} (\rh )$, and 
$$\tau _\alpha (u_\rh )=[\tau _\alpha (u)]_\rh \ .$$
Applying $\tau _\alpha $ to the identity
$$\rh u_\rh =\Psi_\rh H_u(u_\rh )\ ,$$
we infer
$$\rh [\tau _\alpha (u)]_\rh =\tau _\alpha (\Psi_\rh )H_{\tau _\alpha (u)}\left ([\tau _\alpha (u)]_\rh \right )\ ,$$
and similarly 
$$\rh [\expo_{-i\beta }\tau _\alpha (u)]_\rh =\expo_{-i\beta }\tau _\alpha (\Psi_\rh )H_{\expo_{-i\beta }\tau _\alpha (u)}\left ([\expo_{-i\beta }\tau _\alpha (u)]_\rh \right )\  .$$
This leads, for every $\rh \in \Sigma _H(u)$, to
$$\Psi _\rh (\expo_{-i\beta }\tau _\alpha (u))=\expo_{-i\beta }\tau _\alpha (\Psi _\rh (u))\ .$$
Applying this identity to $u=u_0$, $\alpha =ct$ and $\beta =\omega t$, and comparing with Theorem \ref{evolszegotexte}, we conclude
$$\expo_{-it\rh ^2}\Psi _\rh (u_0)=\expo_{-i\omega t }\tau _{ct}(\Psi _\rh (u_0))\ .$$
Writing 
$$\Psi _\rh (u_0)={\rm e}^{-i\varphi}\prod_{1\le j\le m-1} \chi_{p_j},$$
we get, for every $t$,
$$\expo_{-it\rh ^2}\prod_{1\le j\le m-1} \chi_{p_j}=\expo_{-it(\omega +c(m-1))}\prod_{1\le j\le m-1} \chi_{\expo_{ict}p_j}\ .$$
This imposes, since $c\ne 0$, 
$$\rho^2=\omega +(m-1)c\ ,\ p_j=0\ ,$$
for every $\rho \in \Sigma _H(u_0)$. In other words, $\Psi _\rh (u_0)(z)=\expo_{-i\varphi } z^{m-1}\ .$
\s
We repeat the same argument for $\si \in \Sigma _K(u)$,  with $\ell=\dim F_u(\sigma)=\dim E_u(\sigma)+1$ and $$K_u(u'_\sigma)=\sigma \Psi_\sigma u'_\sigma\ ,$$
using this time
$$\tau _\alpha (K_u(h))=\expo _{i\alpha }K_{\tau _\alpha (u)}(\tau _\alpha (h))\ .$$
We get 
$$\sigma^2=\omega-\ell c \ ,$$
and 
$$\Psi _\si (u_0)(z)=\expo_{-i\theta}z^{\ell -1}\ .$$
If we assume that there exists at least two elements $\rh _1>\rh _2$ in $\Sigma _H(u_0)$, 
with $m_j=\dim E_{u_0}(\rho_j)$ for $j=1,2$, from Lemma \ref{rigidity}, there is at least one element $\sigma _1$ in $\Sigma _K(u_0)$, satisfying 
$$\rho _1 > \sigma_1  > \rho _2.$$
Set $\ell _1:=\dim F_{u_0}(\si _1)$, we get
$$(m_1-1)c > -\ell _1 c > (m_2-1)c$$ which is impossible since $m_1,\ell_1, m_2$ are positive integers.
Therefore,  there is only one element $\rho$ in $\Sigma _H(u_0)$, with $m=\dim E_{u_0}(\rho)$ and at most one element $\sigma$ in $\Sigma _K(u_0)$, of multiplicity $\ell$. 
Applying the results of section \ref{section explicit}, we obtain
$$u_0(z)=\frac {(\rh ^2-\si ^2)\expo_{-i\varphi }}{\rh }\frac{z^{m-1}}{1-\frac{\sigma }\rh {\rm e}^{-i(\varphi+\theta)} z^{\ell +m-1}}\ .$$
This completes the proof.
\end{proof}

\section{Application to almost periodicity}\index{almost periodicity}
As a second application of our main result, we prove that the solutions of the Szeg\H{o} equation are almost periodic. Let us recall some definitions. Let $\mathcal V$ be a finite dimensional smooth manifold. A function
$$f: \R \longrightarrow \mathcal V$$
is quasi-periodic \index{quasiperiodicity}if there exists a positive integer $N$, a vector $\omega \in \R ^N$, and a continuous function
$F:\T ^N\rightarrow \mathcal V$ such that
$$\forall t\in \R\ ,\ f(t)=F(\omega t)\ .$$
A similar definition holds for functions valued in a Banach space. \\
Now let us come to the definition of almost periodic functions. Let $X$ be a Banach space. A function
$$f: \R \longrightarrow X$$ is almost periodic  if it is the uniform limit of quasi-periodic functions, namely the uniform limit of  finite linear combinations of functions 
$$t\longmapsto \expo_{i\omega t}x\ ,$$
where $x\in X$ and $\omega \in \R $. Of course, from the explicit formula (\ref{luminy}) and from the evolution under the cubic Szeg\H{o} flow, for any $u_0\in \mathcal V(d)$, the solution $u(t)$ is quasi-periodic, valued in $\mathcal V(d)$, hence valued in every $H^s_+$. This is also  a consequence of the results of \cite{GG4}.\\
It remains to consider data in $H^{1/2}_+$ corresponding to infinite rank Hankel operators\index{Hankel operator}. We are going to use Bochner's criterion, see chapters 1, 2 of \cite{LZ}, namely that $f\in C(\R ,X)$ is almost periodic if and only if it is bounded and the set of functions
$$f_h:t\in \R \longmapsto f(t+h)\in X\ ,\  h\in \R \ ,$$
is relatively compact in the space of bounded continuous functions valued in $X$.
\\
Let $u_0\in \mathcal V^{(2)}_{(d_r)_{r\ge 1}}$. Set 
$$\Phi (u_0)=((s_r)_{r\ge 1}, (\Psi _r)_{r\ge 1})\ .$$
Then, from Theorem \ref{evolszegotexte}, 
$$\Phi (u(t))= ((s_r)_{r\ge 1}, (\expo_{i(-1)^rs_r^2t}\Psi _r)_{r\ge 1})\ .$$
By Theorem \ref{HilbertSchmidt}, it is enough to prove that the set of functions
$$t\in \R \longmapsto \Phi (u(t+h))\in \mathcal S_{(d_r)}^{(2)}$$
is relatively compact in $C(\R ,\mathcal S_{(d_r)}^{(2)})$. This is equivalent to the relative compactness
of the family $(\expo_{i(-1)^rs_r^2h})_{r\ge 1}$ in $(\S ^1)^\infty $, $h\in \R $, which is trivial.

\chapter[Long time instability]{Long Time instability and unbounded Sobolev orbits}\label{chapter instability}\index{Sobolev spaces}

The purpose of this chapter is to prove part 2 of Theorem \ref{dynamicIntro}, namely that generic data in $C^\infty _+(\S ^1)$ generate superpolynomially unbounded trajectories of the Szeg\H{o} evolution in every Sobolev space $H^s$, $s>\frac 12$. This of course is in strong contrast with the compactness properties established for the same trajectories in $H^{\frac 12}$ in the previous chapter. The proof takes advantage of the nonlinear Fourier transform\index{Non linear Fourier transform} constructed in the previous chapters, by proving superpolynomial long time instability for quasiperiodic solutions\index{quasiperiodicity}. Again, the key idea is collapsing singular values\index{singular value}.

\section{Instability and genericity of unbounded orbits}
In this section, we show that part 2 of Theorem \ref{dynamicIntro} is a consequence of the following long time instability result. We denote by $d_\infty $ a distance function on $C^\infty _+(\S ^1)$ which defines the $C^\infty $ topology.
\begin{theorem}\label{nonPolynomial}
For any $v\in C^\infty _+$, for any $M$, for any $s>\frac 12$, there exists a sequence $(v^{(n)})$ of elements of $C^\infty _+$  tending to $v$ in $C^\infty_+$ and  sequences of times $(\overline {t}_n)\ ,\ ( \underline {t}^n)$, tending to $\infty$, such that 
$$
\frac{\Vert Z(\overline t_n)v^{(n)}\Vert _{H^s}}{|\overline t_n|^M}\td_n,\infty \infty \ ,$$
and 
$$d_\infty (Z(\underline {t}^n)v^{(n)}, v^{(n)})\td_n,\infty 0\ .$$
\end{theorem}
Assuming Theorem \ref{nonPolynomial}, a Baire category argument leads to  the following proof of part 2 of Theorem \ref{dynamicIntro}.
\begin{proof}
For every positive integer $M$, we denote by 
$\mathcal O_M$ the set of functions $v\in C^\infty _+$ such that there  exist $\overline t, \underline t$ with   $\vert \overline t\vert >M,  \vert \underline t\vert >M,$ and $$\Vert Z(\overline t)v\Vert _{H^{\frac 12 +\frac 1M}}>M|\overline t|^M,  \ ,\ d_\infty (Z(\underline t)v, v)<\frac 1M\ .$$
From the global wellposedness theory \cite{GG1}, $\mathcal O_M$ is an open subset of $C^\infty _+$. Furthermore, Theorem \ref{nonPolynomial} implies that $\mathcal O_M$ is dense. From the Baire theorem applied to the Fr\'echet space $C^\infty _+$, we conclude that the countable intersection
$$\mathscr G:=\cap _{M\ge 1}\mathcal O_M $$
is a dense $G_\delta $ subset. 
\end{proof}

\section{A family of quasiperiodic solutions}\index{quasiperiodicity}
Let us come to the proof of Theorem \ref{nonPolynomial}. Our strategy is the following. First of all, rational functions are dense in $C^\infty _+$. Furthermore, from Theorem 7.1 in \cite{GG1}, in the finite dimensional manifold of rational functions with associated Hankel operators of given rank $q$, those functions $u$ for which the singular values of $H_u$ and $K_u$ are simple\index{singular value}, is an open dense subset. From Theorem \ref{FT}, every such rational function reads $v:=u({\bf s},{\bf \Phi})$ for some finite sequence ${\bf s}$ of positive numbers of length $2q$ or $2q-1$, and some sequence ${\bf \Phi}$, of the same length,  of complex numbers of modulus $1$. Up to adding a small positive number to  the sequence ${\bf s}$, we infer that 
 those functions $v:=u({\bf s},{\bf \Phi})$, with ${\bf s}\in \Omega _{2q}$ and ${\bf \Phi}\in (\S ^1) ^{2q}$, are dense in $C^\infty _+$. Therefore it is enough to prove the statement of Theorem  \ref{nonPolynomial} if $v$ is a rational function $v:=u({\bf s},{\bf \Phi})$  with 
 $${\bf s}=(\rho_1,\sigma_1,\dots \rho_q,\sigma_q ),\; {\bf \Phi}=(\Phi_j)_{1\le j\le 2q}\in (\S^1)^{2q}$$ where 
 $$\rho_1>\sigma_1>\rho_2>\sigma_2>\dots>\rho_q>\sigma_q>0.
$$
We are going to construct a sequence  $(v^{(n)})$ of elements of $C^\infty _+$  tending to $v$ in $C^\infty_+$ and  a sequence of times $(\overline {t}_n)$, such that 
$$
\frac{\Vert Z(\overline t_n)v^{(n)}\Vert _{H^s}}{|\overline t_n|^M}\td_n,\infty \infty \ ,$$
with 
$$v^{(n)}=u({\bf s}^{(n)}, {\bf \Phi }^{(n)})\ ,$$
where ${\bf s}^{(n)}, {\bf \Phi }^{(n)}$ are to be chosen. At this stage, we remark that the existence of the second sequence $(\underline{t}^n)$ comes for free. Indeed, for a given $n$, we already noticed that the mapping $t\mapsto Z(t)v^{(n)}$ is quasi-periodic valued into some manifold $\mathcal V(d_n)$, 
continuously imbedded into $C^\infty _+$. Hence the function 
$$t\in \R \mapsto d_\infty (Z(t)v^{(n)}, v^{(n)})\ $$
is quasiperiodic valued in $\R $. We then use a classical property, namely, if $f$ is  a quasiperiodic  function, for every $\varepsilon >0$, there exists infinitely many $t\in \R$ such that 
$$\vert f(t)-f(0)\vert <\varepsilon\ .$$
The existence of $\underline{t}^n$ follows.
\s
Let us come back to the construction of the sequences $(v^{(n)})$ and $(\overline {t}_n)$. For technical reasons, it is more convenient to start from the construction of a singular  sequence $u^{(n)}$, which will play the role of $Z(\overline{t}_n)v^{(n)}$ in Theorem  \ref{nonPolynomial}, and then to check that  $Z(-\overline{t}_n)u^{(n)}$ has the desired limit $v$ in $C^\infty _+$. We introduce the following class of rational functions.
\s
 Let $N\ge 2$ be an integer. Denote by $\mathcal X_{N}$ the subset of $(\xi ,\eta )\in \R ^{2N-1}$ such that
 $\xi=(\xi_1,\dots,\xi_N)$, $\eta=(\eta_1,\dots,\eta_{N-1})$ and
 \begin{equation}\label{XiEta}\xi_1>\eta_1>\xi_2>\eta_2>\dots\eta_{N-1}>\xi_N>0.
 \end{equation}  
 Given ${\bf \Psi}=(\Psi _r)_{1\le r\le 2q}\in (\S^1) ^{2q}$, we consider the family $u^{\delta,\e}$ for $\delta,\e\to 0$ with
 $$u^{\delta,\e}=u\left(\left({\bf s},\delta(1+\e\xi_1),\delta(1+\e\eta_1),\dots,\delta(1+\e\xi_N), 0 \right) ,\left({\bf \Psi}, 1,\dots,1\right)\right).$$
 From the explicit formula \ref{luminy}, we have
 \begin{equation}\label{FormuleUde}
 u^{\delta,\e}(z)=\left\langle {\mathscr  C_{\delta,\e}}^{-1}(z)\left (\begin{array}{ll}{\bf \Psi}_q\\{ 1}_N\end{array}\right),\left (\begin{array}{ll}{ 1}_q\\{ 1}_N\end{array}\right)\right\rangle
 \end{equation}
where ${\bf \Psi}_q=(\Psi_{2a-1})_{1\le a\le q}$, $$1_N=\left(\begin{array}{c}
1\\
\vdots\\
1\end{array}\right)\in \R^N$$
and 
$$ {\mathscr  C_{\delta,\e}}(z)=\left(\begin{array}{lll}{\mathscr E}(z)&{\mathscr A}_{\delta ,\e}(z)\\
\text{ }&\text{ }\\
{\mathscr B}_{\delta ,\e}(z)&\frac 1\delta {\mathscr C_{\e}}(z)\end{array}\right)$$
with $${\mathscr E}(z)=\left(\frac{\rho_a-\sigma_bz\Psi_{2a-1}\Psi_{2b}}{\rho_a^2-\sigma_b^2}\right)_{1\le a,b\le q},$$
$${\mathscr A}_{\delta ,\e}(z)=\left(\left(\frac{\rho_a-\delta z(1+\e\eta_k)\Psi_{2a-1}}{\rho_a^2-\delta^2(1+\e \eta_k)^2}\right)_{1\le a\le q\atop 1\le k\le N-1},\frac 1{\rho_a}_{1\le a\le q}\right)$$
$${\mathscr B}_{\delta ,\e}(z)=\left( \frac{\delta(1+\e\xi_j)-\sigma_b z\Psi_{2b}}{\delta^2(1+\e \xi_j)^2-\sigma_b^2}\right)_{1\le j\le N\atop 1\le b\le q}$$
${\mathscr C_\e}(z)=(c_{\e , jk}(z))_{1\le j,k\le N}$, with
\begin{eqnarray*}
c_{\e , jk}(z)&:=&\frac{1+\e \xi_j-z(1+\e \eta_k)}{(1+\e \xi_j)^2-(1+\e \eta_k)^2}\ ,\ 1\le k\le N-1\ ;\\
c_{\e , jN}(z)&:=&\frac 1{1+\e \xi_j}\ ,\ 1\le j\le N\ .
\end{eqnarray*}
The following proposition is the key of the proof of Theorem \ref{nonPolynomial}.
 \begin{proposition}\label{uepsilon}
 \text{ }
 There exists a nonempty open subset $\mathcal X'_N$ of $\mathcal X_N$ such that, for every $(\xi ,\eta )\in \mathcal X'_N$, the following properties hold.
 \begin{enumerate}
 \item The solution of the cubic Szeg\H{o} equation with initial datum $u^{\delta,\e}$ at time $\frac 1{2\e \delta ^2}$
 satisfies
 $$Z\left(\frac 1{2\e \delta ^2}\right)u^{\delta,\e}=u({\bf s},\tilde {\bf \Psi }_{\delta ,\e})+o(1)$$ in $C^\infty _+$ as $\e$ and $\delta$ tend to $0$, where
 $$\tilde \Psi _{\delta ,\e, 2a-1}=\Psi _{2a-1} \expo_{-i\frac {\rho _a^2}{2\e \delta ^2}}\ ,\ \tilde \Psi _{\delta ,\e, 2b}=
 \Psi _{2b} \expo_{i\frac {\sigma _b^2}{2\e \delta ^2}}\ .$$
\item   As $\e $ and $\delta $ tend  to $0$ with $\e \ll \delta $,
 $$\forall s\in (\frac 12, 1), \exists C_s>0: \Vert u^{\delta,\e}\Vert _{H^s}\ge C_s\frac \delta {\e^{(N-1)(2s-1)}}.$$
 \end{enumerate}
\end{proposition}
Let us show how Theorem \ref{nonPolynomial} is a  consequence of Proposition \ref{uepsilon}. Let $M$ be a positive integer and ${\bf \Phi }\in (\S ^1)^{2q}$, and $s>\frac 12$. We may assume that $s<1$. Choose an integer $N$ such that $(N-1)(2s-1)\ge 2M+1$, and $(\xi, \eta )\in \mathcal X'_N$. Consider 
$$u^{(n)}=u^{\delta _n ,\delta _n^2}\ ,$$
where $\delta _n$ is a sequence tending to $0$ such that
$$\expo_{-i\frac {\rho _a^2}{2\delta _n^4}}\rightarrow \expo_{-i\varphi_a}\ ,\ \expo_{i\frac {\sigma _b^2}{2\delta _n^4}}\rightarrow \expo_{i\theta _b}\ ,$$
and where ${\bf \Psi}$ is defined as
$$\Psi _{2a-1}=\Phi _{2a-1}\expo_{i\varphi _a}\ ,\ \Psi_{2b}=\Phi _{2b}\expo_{-i\theta _b}\ .$$
Then Proposition \ref{uepsilon} implies that
 \begin{enumerate}
 \item $v^{(n)}:=Z\left(\frac 1{2\delta _n^4}\right)u^{(n)}$ tends to $v=u({\bf s},{\bf \Phi})$ in $C^\infty _+$.
 \item The following estimates hold,
 $$\left \Vert Z\left(-\frac 1{2\delta _n^4}\right)v^{(n)}\right \Vert _{H^s}=\Vert  u^{(n)}\Vert _{H^s}\ge C_s\frac 1{\delta _n^{2(N-1)(2s-1)-1}}\gg \left (\frac 1{2\delta _n^4}\right )^M.$$
 \end{enumerate}
 This proves Theorem \ref{nonPolynomial} with $\overline t_n=-\frac 1{2\delta _n^4}$.
\s
The proof of Proposition \ref{uepsilon} requires several steps, which will be achieved in the two next sections.
 
 \section{Construction of the smooth family of data}
  We first address the first part of Proposition \ref{uepsilon}. 
From Theorem \ref{FT}, 
$$Z\left(\frac 1{2\e \delta ^2}\right)u^{\delta,\e}(z)=\left\langle \tilde{\mathscr  C}_{\delta,\e}^{-1}(z)\left (\begin{array}{ll}{\bf \Psi}_q\\{\bf \Psi_\e}_N\end{array}\right),\left (\begin{array}{ll}{ 1}_q\\{ 1}_N\end{array}\right)\right\rangle$$
where \beno {\bf \Psi}_q&:=&(\Psi_{2a-1}{\rm e}^{-i\frac{\rho_a^2}{2\e \delta ^2}})_{1\le a\le q}\\ {\bf \Psi_\e}_N&:=&({\rm e}^{-i\frac{(1+\e \xi_j)^2}{2\e}})_{1\le j\le N}={\rm e}^{-i\frac{1}{2\epsilon}}\left({\rm e}^{-i\xi_j(1+\e \frac {\xi _j}2)}\right)_{1\le j\le N}
\eeno
and
$$ \tilde{\mathscr  C}_{\delta,\e}(z)=\left(\begin{array}{lll} \tilde{\mathscr E_{\delta ,\e}}(z)&\tilde{\mathscr  A}_{\delta,\e}(z)\\
\text{ }&\text{ }\\
\tilde{\mathscr B_{\delta,\e}}(z)&\frac 1\delta  { \tilde{\mathscr C}_\e }(z)\end{array}\right)$$
with $$\tilde{\mathscr E}_{\delta,\e}(z)=\left(\frac{\rho_a-\sigma_bz\Psi_{2a-1}\Psi_{2b}{\rm e}^{-i(\frac{\rho_a^2-\sigma_b^2}{2\e \delta ^2})}}{\rho_a^2-\sigma_b^2}\right)_{1\le a,b\le q},$$
$$\tilde{\mathscr  A}_{\delta,\e}(z)=\left(\left(\frac{\rho_a-\delta z(1+\e\eta_k)\Psi_{2a-1}{\rm e}^{-i(\frac{\rho_a^2-\delta^2(1+\e\eta_k)^2}{2\e \delta ^2})}}{\rho_a^2-\delta^2(1+\e \eta_k)^2}\right)_{1\le a\le q\atop 1\le k\le N-1},\left(\frac 1{\rho_a}\right)_{1\le a\le q}\right)$$
$$\tilde{\mathscr  B}_{\delta,\e}(z)=\left( \frac{\delta(1+\e\xi_j)-\sigma_b z\Psi_{2b}{\rm e}^{-i(\frac{\delta^2(1+\e\xi_j)^2-\sigma_b^2}{2\e \delta ^2})}}{\delta^2(1+\e \xi_j)^2-\sigma_b^2}\right)_{1\le j\le N\atop 1\le b\le q}$$
and 
$$  \tilde  {\mathscr C}_{\e}(z)=\left(\frac{1+\e\xi_j-z(1+\e\eta_k){\rm e}^{-i(\xi_j-\eta_k)-i\e \frac{\xi _j^2-\eta _k^2}2}}{\e (\xi _j-\eta _k)(2+\e (\xi _j+\eta _k))},\frac 1{1+\e\xi_j}\right).$$
Remark that  from Theorem \ref{mainfiniterank}, the functions 
$$u({\bf s}, (\Psi _r\expo_{-i\psi _r})_{1\le r\le 2q})\ ,\  (\psi _r)\in \T ^{2q}\ ,$$
lie in a compact subset --- a torus --- of the manifold $\mathcal V (2q)$. This implies that the distance of the zeroes of the denominator 
$\det \tilde{\mathscr E}_{\delta,\e}(z)$ to the closed unit disc is bounded from below by a positive constant. 
As a consequence,  the matrices $\tilde{\mathscr E}_{\delta,\e}(z)$ are invertible with a bounded inverse, for every $z$ in a fixed neighborhood of $\overline{\D }$.
The first part of Proposition \ref{uepsilon} is a consequence of the following lemma.
\begin{lemma}\label{Cinverse}
There exists a nonempty open subset $\mathcal X''_N$ of $\mathcal X_N$, such that, for every $(\xi ,\eta )\in \mathcal X''_N$, there exist $r_{\xi ,\eta }>1$ and $\e _{\xi ,\eta }>0$ such that, for $0<\e <\e _{\xi ,\eta }$,  
the matrix $\tilde {\mathscr C}_\e (z)$ is invertible with a bounded inverse for every $z$ such that $\vert z\vert \le r_{\xi ,\eta}$.
\end{lemma}
Let us admit this lemma for a while. Using it, we can easily describe the inverse of the matrix $\tilde{\mathscr C}_{\delta,\e}(z)$.
\begin{lemma}\label{inverse} 
For every $(\xi ,\eta )\in \mathcal X''_N$, there exist $r_{\xi ,\eta }>1$ and $\gamma _{\xi ,\eta }>0$ , such that, for every $z$ with $\vert z\vert \le r_{\xi ,\eta}$, for every $\delta ,\e <\gamma _{\xi ,\eta}$,
the matrix $$\tilde {\mathscr J}_{\delta ,\e}(z):=\tilde{\mathscr E}_{\delta,\e}(z)-\delta\tilde{\mathscr A}_{\delta,\e}(z)\tilde{\mathscr C}_\e ^{-1}(z)\tilde{\mathscr B}_{\delta,t_\e}(z),$$
is invertible with a bounded inverse . The  inverse of the matrix $\tilde{\mathscr  C}_{\delta,\e}(z)$  is given by
$$(\tilde{\mathscr C}_{\delta,\e})^{-1}=\left(\begin{array}{lll}\tilde {\mathscr J}_{\delta ,\e}^{-1}&-\delta \tilde {\mathscr J}_{\delta ,\e}^{-1}\tilde{\mathscr A}_{\delta,\e}\tilde{\mathscr C}_\e ^{-1}\\
-\delta\tilde{\mathscr C}_\e ^{-1}\tilde{\mathscr B}_{\delta,\e}\tilde {\mathscr J}_{\delta ,\e}^{-1}& \delta\tilde{\mathscr C}_\e ^{-1}[I+\delta \tilde{\mathscr B}_{\delta,\e} \tilde {\mathscr J}_{\delta ,\e}^{-1}\tilde{\mathscr A}_{\delta,\e}\tilde{\mathscr C}_\e ^{-1}]\end{array}\right).$$
Furthermore, 
$$Z\left(\frac 1{2\e \delta ^2}\right)u^{\delta,\e}=u({\bf s},\tilde {\bf \Psi }_{\delta ,\e})+o(1)$$ in $C^\infty _+$ as $\e$ and $\delta$ tend to $0$, where
 $$\tilde \Psi _{\delta ,\e, 2j}=\Psi _{2j} \expo_{-i\frac {\rho _j^2}{2\e \delta ^2}}\ ,\ \tilde \Psi _{\delta ,\e, 2k-1}=
 \Psi _{2k-1} \expo_{i\frac {\sigma _k^2}{2\e \delta ^2}}\ .$$
\end{lemma}
\begin{proof}
The invertibility of the matrix $\tilde {\mathscr J}_{\delta ,\e}(z)$ for $\delta$ and $\e $ small enough and $|z|<r_{\xi,\eta}$ comes  from the already observed invertibility of $\tilde{\mathscr E}_{\delta,\e}(z)$ in a neighborhood of $\overline{D}$, with a bounded inverse,  and from Lemma \ref{Cinverse}. 
 In the next calculations, we drop the variable $z$ for simplicity.\\
Write, for $X_q,X'_q\in \R^q$, $Y_N,Y'_N\in \R^N$, $$\tilde{\mathscr C}_{\delta,\e}\left(\begin{array}{l}
X_q\\
Y_N\end{array}\right)=\left(\begin{array}{l}
X'_q\\
Y'_N\end{array}\right)$$ so that 
\begin{eqnarray*}
\tilde{\mathscr E}_{\delta,\e}X_q+\tilde{\mathscr A}_{\delta,\e} Y_N&=&Y'_q\\
\tilde{\mathscr B}_{\delta,\e}X_q+\frac 1\delta\tilde{\mathscr C}_{\e} Y_N&=&Y'_N.
\end{eqnarray*}
and solve this system to get
\begin{eqnarray*}
Y_N&=&\delta\tilde{\mathscr C}_{\e}^{-1}(Y'_N-\tilde{\mathscr B}_{\delta,\e}X_q)\\
\tilde {\mathscr J}_{\delta ,\e}X_q&=&Y'_q-\delta\tilde{\mathscr A}_{\delta,\e}\tilde{\mathscr C}_{\e}^{-1}(Y'_N)
\end{eqnarray*}
and eventually
\begin{eqnarray*}
X_q&=&\tilde {\mathscr J}_{\delta ,\e}^{-1}(Y'_q)-\delta\tilde {\mathscr J}_{\delta ,\e}^{-1}\tilde{\mathscr A}_{\delta,\e}\tilde{\mathscr C}_{\e}^{-1}(Y'_N)\\
Y_N&=&\delta\tilde{\mathscr C}_{\e}^{-1}[Y'_N+ \delta\tilde{\mathscr B}_{\delta,\e}\tilde {\mathscr J}_{\delta ,\e}^{-1}\tilde{\mathscr A}_{\delta,\e}\tilde{\mathscr C}_{\e}^{-1}(Y'_N)-\tilde{\mathscr B}_{\delta,\e}\tilde {\mathscr J}_{\delta ,\e}^{-1}(Y'_q)].\end{eqnarray*}
It gives the formula for the inverse. Furthermore, as $\delta, \e $ tend to $0$, we obtain
$$(\tilde{\mathscr C}_{\delta,\e})^{-1}=\left(\begin{array}{lll}\tilde{\mathscr E_{\delta ,\e}}^{-1}&0\\
0& 0\end{array}\right) + o(1).$$
This completes the proof of Lemma \ref{inverse}.
\end{proof}
The proof of  the first part of Proposition \ref{uepsilon} is thus reduced to  the proof of Lemma \ref{Cinverse}, which we  now begin.
\begin{proof}
Recall that 
$$\tilde  {\mathscr C}_{\e}(z)=\left(\frac{1+\e\xi_j-z(1+\e\eta_k){\rm e}^{-i(\xi_j-\eta_k)-i\e \frac{\xi _j^2-\eta _k^2}2}}{\e (\xi _j-\eta _k)(2+\e (\xi _j+\eta _k))},\frac 1{1+\e\xi_j}\right)_{1\le j\le N, 1\le k\le N-1}.$$
From Cramer's formulae, $$\tilde{\mathscr C}_\e (z)^{-1}=\frac 1{\det \tilde{\mathscr C}_\e (z)}^t{\rm Co}\tilde{\mathscr C}_\e (z),$$ where ${\rm Co}$ denotes the matrix of cofactors. As the coefficients of $\tilde{\mathscr C}_\e (z)$ are polynomials in $z$ of degree at most $1$, the coefficients of ${\rm Co}\tilde{\mathscr C}_\e (z)$ are polynomial in $z$ of degree at most $N-1$. Moreover, these coefficients  grow at most as $\frac 1{\e^{N-1}}$. Hence, it suffices to prove that, for  $(\xi ,\eta)$ in a suitable non empty open subset of $\mathcal X_N$, the family of polynomials $\e^{N-1}\det    \tilde{\mathscr C}_{\e}(z)$ converges, as $\e$ tends to zero,  to a polynomial of degree $N-1$  having all its roots outside a disc of radius $r_{\xi ,\eta }>1$. \\
Let us compute
\begin {eqnarray*}
P(\xi,\eta)(z)&=&\lim_{\e\to 0}(2\e)^{N-1}\det  \tilde{\mathscr C}_{\e}(z)\\
&=&\det\left(\left (\frac{1-z{\rm e}^{-i(\xi_j-\eta_k)}}{\xi_j-\eta_k}\right )_{1\le j\le N;1\le k\le N-1},{ 1_N}\right).
\end{eqnarray*}
 Notice that the determinant in the right hand side  is a polynomial in $z$ of degree $N-1$ whose coefficient of $z^{N-1}$ equals $$(-1)^{N-1}\left (\prod_k {\rm e}^{i\eta_k}\right )\det\left(\left (\frac{{\rm e}^{-i\xi_j}}{\xi_j-\eta_k}\right )_{1\le j\le N;1\le k\le N-1},{ 1_N}\right).$$ With the choice $\xi=\xi^*=\left(2\pi(N-j+1)\right)_{1\le j\le N}$, this  determinant is 
$$(-1)^{N-1}\left (\prod_k {\rm e}^{i\eta_k}\right )\det\left(\left (\frac{1}{\xi_j-\eta_k}\right )_{1\le j\le N;1\le k\le N-1},{ 1_N}\right)\ .$$
Developing this determinant with respect to the last column,  we are led to compute Cauchy determinants.
Let us recall that a Cauchy matrix is a matrix of the form $\displaystyle\left(\frac 1{a_j+b_k}\right)$. Its determinant is given by
 \begin{equation}\label{determinantCauchy}\det \left(\frac 1{a_j+b_k}\right)=\frac{\prod_{i<j}(a_i-a_j)\prod_{k<l}(b_k-b_l)}{\prod_{j,k}(a_j+b_k)}.\end{equation}
In view of this formula, we get
\begin{eqnarray*}
&&\det\left(\left (\frac{1}{\xi_j-\eta_k}\right )_{1\le j\le N;1\le k\le N-1},{ 1_N}\right)=\\
&&\sum _{r=1}^N(-1)^{r+N}
\frac{\prod _{i<j, i, j\ne r}(\xi _i-\xi _j)\prod _{k<\ell}(\eta _k-\eta _\ell)}{\prod _{j\ne r, k}(\xi _j-\eta _k)}\\
&=&\frac{ \sum _{r=1}^N(-1)^{r+N}
\prod _{1\le k\le N-1}(\xi _r-\eta _k)\prod _{i<j, i, j\ne r}(\xi _i-\xi _j)\prod _{k<\ell}(\eta _k-\eta _\ell) }{\prod _{j, k}(\xi _j-\eta _k)}
\end{eqnarray*}
and we observe that every term in the sum is different from $0$ and has the sign of $(-1)^{N-1}$. Therefore this
quantity is not zero.  On the other hand, Theorem \ref{FT} tells us that the roots of $\det \tilde {\mathscr C}_\e (z)$ are located outside the unit disc. Hence all the roots of $P(\xi,\eta)$ belong to $\{z,\; |z|\ge 1\}$. Furthermore, 
$$P(\xi^*,\eta)(z)=\det\left(\left (\frac 1{\xi_j-\eta_k}\right )_{1\le j\le N;1\le k\le N-1},{1_N}\right)\prod_{k=1}^{N-1}(1-z{\rm e}^{i\eta_k}).$$
The above calculation shows that $P(\xi^*,\eta )$ has $N-1$  zeroes which belong to the unit circle, namely ${\rm e}^{-i\eta_k}, k=1,\dots ,N-1$. Fix $\eta ^*\in \R ^{N-1}$ such that $(\xi ^*,\eta ^*)\in \mathcal X_N$ and $\eta_k^*\neq \eta_l^* \; {\rm mod}\,  2\pi $ for every $k\neq l$, so that these $N-1$ zeroes are simple if $\eta =\eta ^*$. 
 By the implicit function theorem, there exist $\alpha >0$ such that,  if $\vert \xi -\xi ^*\vert <\alpha $ and $\vert \eta -\eta ^*\vert <\alpha $, the polynomial $P(\xi,\eta)$ has $N-1$ simple zeroes $\left\{z_k(\xi,\eta);\; 1\le k\le N-1\right\}$. The first part of Proposition \ref{uepsilon} is a direct consequence of the following lemma.
 \begin{lemma}\label{outofdisc}
There exists a nonempty open subset $\mathcal X''_N$ of $\mathcal X_N$  such that, for every $(\xi ,\eta )\in \mathcal X''_N$, for  $k=1,\dots ,N-1$,   $\vert z_k(\xi ,\eta )\vert >1$.
\end{lemma}
\begin{proof} The functions $\xi\mapsto z_k(\xi, \eta)$ are analytic and satisfy $$|z_k(\xi, \eta)|^2\ge 1\ ,\  |z_k(\xi^*,\eta)|^2=1\ .$$Denote by $Q_k(\xi ,\eta )$ the  Hessian quadratic form  of the function $$\xi\mapsto |z_k(\xi, \eta)|^2\ .$$
We know that $Q_k(\xi ^*,\eta )$ is a nonnegative quadratic form. We claim that, for any $k$, $Q_k(\xi ^*,\eta )$ is not $0$ for $\eta$ in a dense open subset of the ball of radius $\alpha $ centered at $\eta ^*$. If this claim is correct, then, for such 
$\eta $, for $\xi $ close enough to $\xi ^*$ such that $\xi -\xi ^*$ does not belong to the union of the kernels of $Q_1(\xi ^*,\eta ),\dots ,Q_{N-1}(\xi ^*,\eta )$, we have $\vert z_k(\xi ,\eta )\vert >1$ for $k=1,\dots ,N-1$. This provides us with the nonempty open subset $\mathcal X''_N$ in the lemma. \\
Therefore  it suffices to prove that the Laplacian of $\xi\mapsto \vert z_k(\xi, \eta)\vert ^2$, which coincides with the trace of $Q_k(\xi ,\eta )$, is not $0$ identically in $\eta $ for $\xi =\xi ^*$. Let us compute  $\Delta (\vert z_k\vert ^2)(\xi ^*,\eta )$.
$$
\sum_{j=1}^N\frac{\partial^2}{\partial \xi_j^2}\left(\frac{|z_k|^2}2\right)_{|_{\xi=\xi^*}}=\sum_{j=1}^N\left ({\rm Re }\left(\overline z_k\frac{\partial^2 z_k}{\partial \xi_j^2}\right)_{|_{\xi=\xi^*}}+\left|   \frac{\partial z_k}{\partial \xi_j}\right|^2_{|_{\xi=\xi^*}}\right )\ .
$$
Differentiating the equation $P(\xi,\eta)(z_k(\xi ,\eta ))=0$, we obtain 
$$\frac{\partial z_k}{\partial \xi_j}=-\frac{\frac{\partial P}{\partial \xi_j}}{\frac{\partial P}{\partial z}}\ ,\ \frac{\partial^2 z_k}{\partial \xi_j^2}=\frac{-\frac{\partial^2P }{\partial \xi_j^2}}{\frac{\partial P}{\partial z}}+2\frac{\frac{\partial ^2P}{\partial \xi_j\partial z}\frac{\partial P}{\partial \xi_j}}{\left(\frac{\partial P}{\partial z}\right)^2}-\frac{\left(\frac{\partial P}{\partial \xi_j}\right)^2\frac{\partial ^2P}{\partial z^2}}{\left(\frac{\partial P}{\partial z}\right)^3}\ .$$
Introduce the following quantities. 
\begin{eqnarray*}
D_{jk}&:=&\frac 1{\xi_j^*-\eta_k}\det\left(\left (\frac 1{\xi_r^*-\eta_l}\right )_{r\neq j, l\neq k},{ 1_{N-1}}\right),\\
 D&:=&\det\left(\left (\frac 1{\xi_r^*-\eta_l}\right )_{1\le r\le N, 1\le l\le N-1},{ 1_N}\right)=\sum_{j=1}^N(-1)^{j+k}D_{jk}\ ,\ k=1,\dots ,N-1,\\
 \zeta_k&:=&\prod_{\l\neq k}(1-{\rm e}^{i (\eta_l-\eta_k)})\ ,\\
 a_{lk}&:=&-\frac{{\rm e}^{i (\eta_l-\eta_k)}}{1-{\rm e}^{i (\eta_l-\eta_k)}}=\frac{{\rm e}^{i \frac{(\eta_l-\eta_k)}2}}{2i\sin \frac{(\eta_l-\eta_k)}2}.
 \end{eqnarray*}
Notice that ${\rm Re}(a_{lk})=\frac 12$.\\ Differentiating
$$P(\xi,\eta)(z)= \det\left(\left (\frac{1-z{\rm e}^{-i (\xi_j-\eta_k)}}{\xi_j-\eta_k}\right )_{1\le j\le N;1\le k\le N-1},{ 1_N}\right),$$ one gets, after some computations,
\begin{eqnarray*}
&&\frac{\partial P}{\partial \xi_j}_{_|{\xi=\xi^*, z=z_k}}=i (-1)^{j+k}D_{jk}\zeta_k\; ,\;
\frac{\partial P}{\partial z}_{_|{\xi=\xi^*, z=z_k}}=-{\rm e}^{i \eta_k}\zeta_kD\ ,\\
&&\frac{\partial ^2P}{\partial \xi_j^2}_{_|{\xi=\xi^*, z=z_k}}= (-1)^{j+k}\zeta_kD_{jk}\left(1 -\frac{2i}{\xi_j^*-\eta_k}\right)\ ,\\
&&\frac{\partial^2 P}{\partial z^2}_{_|{\xi=\xi^*, z=z_k}}=2\sum_{l\neq k}\frac{\zeta_k{\rm e}^{i (\eta_l+\eta_k)}}{1-{\rm e}^{i (\eta_l-\eta_k)}} D=-2\zeta_k{\rm e}^{2i \eta_k}\sum_{l\neq k}a_{lk} D\ ,
\end{eqnarray*}
and
\begin{eqnarray*}
&&\frac{\partial ^2P}{\partial \xi_j\partial z}_{_|{\xi=\xi^*, z=z_k}}=(-1)^{j+k}\zeta_kD_{jk}{\rm e}^{i \eta_k}\left(i +\frac 1{\xi_j^*-\eta_k}-i \sum_{l\neq k}\frac{{\rm e}^{i (\eta_l-\eta_k)}}{1-{\rm e}^{i (\eta_l-\eta_k)}}\right)+\\
&+&\zeta_k{\rm e}^{i \eta_k}\sum_{l\neq k} (-1)^{j+l}D_{jl}\left(\frac{1}{\xi_j^*-\eta_l} -i \frac{{\rm e}^{i (\eta_l-\eta_k)}}{1-{\rm e}^{i (\eta_l-\eta_k)}}\right)\\
&=&\zeta_k{\rm e}^{i \eta_k}\left((-1)^{j+k}D_{jk}\left(i +\frac 1{\xi_j^*-\eta_k}+i \sum_{l\neq k}a_{lk}\right)+\sum_{l\neq k} (-1)^{j+l}D_{jl}\left(\frac{1}{\xi_j^*-\eta_l} +i a_{lk}\right)\right)\ .
\end{eqnarray*}
 Hence, inserting these formulae to compute the Laplacian, we obtain
$$
{\rm Re }\left(\overline z_k\frac{\partial^2 z_k}{\partial \xi_j^2}\right)_{|_{\xi=\xi^*}}={\rm Re }\left({\rm e}^{i \eta_k}\frac{\partial^2 z_k}{\partial \xi_j^2}\right)_{|_{\xi=\xi^*}}
 =: I_j+II_j+III_j
$$
 with 
$$I_j={\rm Re }\left({\rm e}^{i \eta_k}\frac{-\frac{\partial^2P }{\partial \xi_j^2}}{\frac{\partial P}{\partial z}}\right)=(-1)^{j+k} \frac{D_{jk}}D\ ,$$
 \begin{eqnarray*}
 II_j&=&2{\rm Re }\left({\rm e}^{i \eta_k}\frac{\frac{\partial ^2P}{\partial \xi_j\partial z}\frac{\partial P}{\partial \xi_j}}{\left(\frac{\partial P}{\partial z}\right)^2}\right)\\
 &=&2 {\rm Re }\left(\frac{i\left(D_{jk}^2\left( i +\frac 1{\xi_j^*-\eta_k}+i \sum_{l\neq k}a_{lk}\right)+\sum_{l\neq k} (-1)^{k+l}D_{jk}D_{jl}\left(\frac{1}{\xi_j^*-\eta_l} +i a_{lk}\right)\right)}{D^2}\right)\\
 &=&- \left((N-1)\frac{D_{jk}^2}{D^2}+\frac{\sum_{l=1}^{N-1} (-1)^{k+l}D_{jk}D_{jl}}{D^2}\right),
 \end{eqnarray*}
 and
 $$III_j=-{\rm Re }\left({\rm e}^{i \eta_k}\frac{\left(\frac{\partial P}{\partial \xi_j}\right)^2\frac{\partial ^2P}{\partial z^2}}{\left(\frac{\partial P}{\partial z}\right)^3}\right)={\rm Re }\left( \frac{2\sum_{l\neq k}a_{lk}D_{jk}^2}{D^2}\right)=(N-2) \frac{D_{jk}^2}{D^2}.$$
Remark that 
$$\sum_{l=1}^{N-1}(-1)^{j+l} D_{jl}=D-(-1)^{j+N}\det\left(\frac 1{\xi_r^*-\eta_l}\right)_{r\neq j}.$$
Putting these identities together, we obtain

 \begin{eqnarray*}
 \sum_{j=1}^N\frac{\partial^2}{\partial \xi_j^2}\left(\frac{|z_k|^2}2\right)_{|_{\xi=\xi^*}}&=& \sum_{j=1}^N\left(\sum_{l=1}^{N-1} (-1)^{k+l+1}\frac{D_{jk}D_{jl}}{D^2}+(-1)^{j+k}\frac{D_{jk}}D\right)\\
 &=& (-1)^{k+N}\frac 1{D^2}\sum_j D_{jk}\det\left(\frac 1{\xi_r^*-\eta_l}\right)_{r\neq j}
 \end{eqnarray*}
 It remains to check that the following analytic expression in $\eta$, 
$$(-1)^{k+N}\sum_{j=1}^N D_{jk}\det\left(\frac 1{\xi_r^*-\eta_l}\right)_{r\neq j}$$
is not identically zero.
This follows from the fact that the limit as $\eta_k$ tends to infinity of 
$$(-1)^{k+N}\eta_k^2\sum_{j=1}^N D_{jk}\det\left(\frac 1{\xi_r^*-\eta_l}\right)_{r\neq j}$$
equals $$\sum_{j=1}^N\det\left(\left (\frac 1{\xi_r^*-\eta_l}\right )_{r\neq j, l\neq k},{ 1_N}\right)^2$$ which is clearly not  zero .\\
This completes the proof of Lemma \ref{outofdisc}, hence of the first part of Proposition \ref{uepsilon}.
\end{proof}
\begin{remark}
At this stage, let us notice that the above construction would break down if, instead of formula (\ref{FormuleUde}),  we used for $u^{\delta ,\e}$  the more natural choice 
$$u\left(\left({\bf s},\delta(1+\e\xi_1),\delta(1+\e\eta_1),\dots,\delta(1+\e\xi_N), \delta (1+\e \eta _N) \right) ,\left({\bf \Psi}, 1,\dots,1\right)\right),$$
with real numbers $\xi _1>\eta _1>\dots >\xi _N>\eta _N$. Indeed, in that case, the corresponding matrix $\mathscr C_\e (z)$ would be
$$\left(\frac{1+\e\xi_j-z(1+\e\eta_k){\rm e}^{-i(\xi_j-\eta_k)-i\e \frac{\xi _j^2-\eta _k^2}2}}{\e (\xi _j-\eta _k)(2+\e (\xi _j+\eta _k))}\right),$$
and, as $\e $ tends to $0$,  its determinant would be equivalent to
$$\frac{1}{(2\e )^N} \det\left (\frac{1-z{\rm e}^{-i(\xi_j-\eta_k)}}{\xi_j-\eta_k}\right)$$
The product of the $N$ zeroes of the polynomial in $z$ in the right hand side is clearly of modulus $1$. Since, from Theorem \ref{FT},  all these zeroes live outside the open unit disc, they have to be located on the unit circle, which is precisely the opposite of what we want.
\end{remark}
\section{The singular behavior}
Let us come to the proof of the second part of Proposition \ref{uepsilon}. We are going to focus our analysis near the point $z=1$, where the singularity of $u^{\delta ,\e }(z)$ takes place. To accomplish this, we introduce a localized version of the $H^s$ norm. Given $s\in \left (\frac 12, 1\right )$, it is classical (see e.g. \cite{Pe2}) that, for every $u\in H^s_+(\S ^1)$, the following equivalence of norms holds,
$$\left (\sum _{n=1}^\infty n^{2s}\vert \hat u(n)\vert ^2 \right )^{\frac 12}\simeq \left (\int _{\vert z\vert <1}\vert u'(z)\vert ^2(1-\vert z\vert ^2)^{1-2s}\, dL(z)\right )^{\frac 12}\ ,$$
where $L$ denotes the Lebesgue measure on $\C =\R ^2$ and $u'$ denotes the holomorphic derivative of $u$.
\s
Given a small constant $\theta >0$ to be fixed later, we introduce the following subset of $\D $,
$$D_{\theta ,\e }:=\{ z\in \C : \vert z\vert <1\ ,\ \vert z-1\vert < \theta \e ^{2(N-1)}\} \ ,$$
and the corresponding localized $H^s$ norm,
$$\Vert u\Vert _{s,\theta , \e }:=\left (\int _{D_{\theta ,\e }}\vert u'(z)\vert ^2(1-\vert z\vert ^2)^{1-2s}\, dL(z)\right )^{\frac 12}\ ,$$
so that we have the following inequality,
$$\forall u\in H^s_+\ ,\  \Vert u\Vert _{H^s}\ge C_s \Vert u\Vert _{s,\theta , \e }\ .$$
Notice that an elementary computation yields
$$\left (\int _{D_{\theta ,\e }}(1-\vert z\vert ^2)^{1-2s}\, dL(z)\right )^{\frac 12}\simeq (\sqrt{\theta }\e ^{N-1})^{3-2s}\ .$$
The second part of Proposition \ref{uepsilon} therefore comes from the following
\begin{proposition}\label{derivee}
There exists $\theta ^*=\theta ^*(\xi ,\eta )>0$ and a dense open subset $\mathcal X'''_N$ such that, for $\theta \le \theta ^*$, for every 
$(\xi ,\eta )\in \mathcal X'''_N$, 
$$\forall z\in D_{\theta ,\e }\ ,\ \vert (u^{\delta ,\e })'(z)\vert \ge C_{\xi ,\eta }\frac{\delta }{\e ^{2(N-1)}}\ , $$
for some $C_{\xi ,\eta }>0$, uniformly in $\delta , \e $ such that $\e \ll \delta \ll 1.$
\end{proposition}
Indeed, the proof of  Proposition \ref{uepsilon} is completed by denoting by $\mathcal X'_N$ the intersection of $\mathcal X'''_N$ and the nonempty open subset $\mathcal X''_N$ provided by the proof of the first part of the proposition.
\s
Let us prove Proposition \ref{derivee}. Deriving with respect to $z$ the formula (\ref{FormuleUde}) giving $u^{\delta,\e}$, we get
\begin{equation}\label{deriveU}
 (u^{\delta,\e})'(z)=\left\langle \dot{{\mathscr C_{\delta,\e}}}{\mathscr C_{\delta,\e}}(z)^{-1}\left (\begin{array}{ll}{\bf \Psi}_q
 \\{ 1}_N\end{array}\right),^t{\mathscr C_{\delta,\e}}( z)^{-1}\left (\begin{array}{ll}{ 1}_q\\{ 1}_N\end{array}\right)\right\rangle
 \end{equation}
 where 
  \begin{equation}\label{tildeC}\dot{\mathscr C_{\delta,\e}}=
 \left(\begin{array}{lll}\dot{{\mathscr E}}&\dot{\mathscr A}_{\delta ,\e}\\
\text{ }&\text{ }\\
\dot{\mathscr B}_{\delta ,\e}&\frac 1\delta\dot{ {\mathscr C_\e}}\end{array}\right)\end{equation}

with 
 \begin{equation}\label{Czero}\dot{{\mathscr E}}=\left(\frac{\sigma_b\Psi_{2a-1}\Psi_{2b}}{\rho_a^2-\sigma_b^2}\right)_{1\le a,b\le q},\end{equation}
 \begin{equation}\label{belA}\dot{\mathscr A}_{\delta ,\e}=\left(\left(\frac{\delta (1+\e\eta_k)\Psi_{2a-1}}{\rho_a^2-\delta^2(1+\e \eta_k)^2}\right)_{1\le a\le q\atop 1\le k\le N-1},{\bf 0}_{1\le a\le q}\right)\end{equation}
  \begin{equation}\label{belB}\dot {\mathscr B}_{\delta ,\e}=\left( \frac{\sigma_b\Psi_{2b}}{\delta^2(1+\e \xi_j)^2-\sigma_b^2}\right)_{1\le j\le N\atop 1\le b\le q}\end{equation}
  \begin{equation}\label{deriveCepsilon}\dot {{\mathscr C_\e}}=(\dot c_{\e , jk})_{1\le j,k\le N},\end{equation}
$$\dot c_{\e , jk}:=\frac{1+\e \eta_k}{(1+\e \xi_j)^2-(1+\e \eta_k)^2}\ ,\ 1\le k\le N-1\ ,\ \dot c_{\e , jN}:=0\ ,\ 1\le j\le N\ .$$

Let us introduce the following notation.
$${\mathscr C_{\delta,\e}}(z)^{-1}\left (\begin{array}{ll}{\bf \Psi_q}\\{ 1}_N\end{array}\right)=:\left (\begin{array}{ll}X_q^{\delta ,\e }(z)\\ \ \\Y_N^{\delta ,\e }(z)\end{array}\right)$$
and 
$$^t{\mathscr C_{\delta,\e}}( z)^{-1}\left (\begin{array}{ll}{ 1}_q\\{ 1}_N\end{array}\right)=:\left (\begin{array}{ll}\hat X_q^{\delta ,\e }(z)\\ \ \\ \hat Y_N^{\delta ,\e }(z)\end{array}\right).$$
It gives rise to the following equations
$$\left\{\begin{array}{llll}
\mathscr E(z)X_q^{\delta ,\e }(z)+\mathscr A_{\delta ,\e}(z)Y_N^{\delta ,\e }(z)&=&{\bf \Psi_q}\\
\mathscr B_{\delta ,\e}(z)X_q^{\delta ,\e }(z)+\frac 1\delta \mathscr C_\e(z) Y_N^{\delta ,\e }(z)&=&{ 1}_N\end{array}\right.$$
$$\left\{\begin{array}{llll}
^t\mathscr E(z)\hat X_q^{\delta ,\e }(z)+^t\mathscr B_{\delta ,\e}(z)\hat Y_N^{\delta ,\e }(z)&=&{ 1}_q\\
^t\mathscr A_{\delta ,\e}(z)\hat X_q^{\delta ,\e }(z)+\frac 1\delta ^t\mathscr C_\e(z) \hat Y_N^{\delta ,\e }(z)&=&{ 1}_N\end{array}\right.$$
Hence, setting
$$\mathscr J_{\delta ,\e }(z):=\mathscr E(z)-\delta\mathscr A_{\delta ,\e}(z) \mathscr C_\e(z)^{-1}\mathscr B_{\delta ,\e}(z)\ ,$$
we obtain
\begin{equation}\label{Xq}
\mathscr J_{\delta ,\e }(z)X_q^{\delta ,\e }(z)={\bf \Psi_q}-\delta\mathscr A_{\delta ,\e}(z) \mathscr C_\e(z)^{-1}{ 1}_N
\end{equation}
\begin{equation}\label{X'q}
^t\mathscr J_{\delta ,\e }(z)\hat X_q^{\delta ,\e }(z)={ 1}_q-\delta^t\mathscr B_{\delta ,\e}(z) ^t\mathscr C_\e(z)^{-1}{ 1}_N
\end{equation}
\begin{equation}\label{YN}
Y_N^{\delta ,\e }(z)=\delta \mathscr C_\e(z)^{-1}({ 1}_N-\mathscr B_{\delta ,\e}(z)X_q^{\delta ,\e }(z))
\end{equation}
\begin{equation}\label{Y'N}
\hat Y_N^{\delta ,\e }(z)=\delta^t \mathscr C_\e(z)^{-1}({ 1}_N-^t\mathscr A_{\delta ,\e}(z)\hat X_q^{\delta ,\e }(z)).
\end{equation}
Our main analysis  lies in the following approximation result.
\begin{proposition}\label{Approximate}
For the norm $L^\infty (D_{\theta ,\e })$, with $\theta <\theta ^*(\xi ,\eta )$ small enough, we have, uniformly in $\delta ,\e $ such that $\e \ll \delta \ll 1$, 
\begin{eqnarray*}  
X_q^{\delta ,\e }(z)&=&\mathscr E(1)^{-1}{\bf \Psi_q}+\mathcal O(\delta)\\
\hat X_q^{\delta ,\e }(z)&=&^t\mathscr E(1)^{-1}{ 1}_q+\mathcal O(\delta)\\
Y_N^{\delta ,\e }(z)&=&\alpha\delta\, { \mathscr C_\e}(z)^{-1}(1_N)+ \mathcal O\left(\frac{\delta^2}{\e^{N-2}}\right)\\
\hat Y_N^{\delta ,\e }(z)&=&\beta\delta \, ^t{ \mathscr C_\e}(z)^{-1}(1_N)+\mathcal O\left(\frac{\delta^2}{\e^{N-1}}\right)\\
\end{eqnarray*}
where 
\begin{eqnarray*}\alpha :
&=&1-\left \langle {\mathscr E}(1)^{-1}({\bf \Psi_q}),\left(\frac{\Psi_{2b}}{\sigma_b}\right)_{1\le b\le q}\right \rangle,\\
 \beta :&=&1-\left \langle ^t{\mathscr E}(1)^{-1}({1_q}),\left(\frac{1}{\rho_a}\right)_{1\le a \le q}\right \rangle\ .
 \end{eqnarray*}
\end{proposition}
\begin{proof}
The proof involves several steps. The first one consists in a number of cancellations in the action of the matrix $\mathscr C_\e (1)^{-1}$ on some special vectors. From the formula giving ${\mathscr C_\e}(z)$, we have
 $${ \mathscr C_\e}(1)=\left(\left (\frac 1{2+\e(\xi_j+\eta_k)}\right )_{1\le j\le N;1\le k\le N-1}, \left (\frac 1{1+\e \xi_j}\right )_{1\le j\le N}\right)$$ which is a Cauchy matrix. From Cramer's formulae and formula \ref{determinantCauchy}, the inverse of a Cauchy matrix is given by
 \begin{equation}\label{InverseCauchy}
 \left(\left(\frac 1{a_j+b_k}\right)^{-1}\right)_{kj}=\left((-1)^{j+k}\frac{\lambda_j\mu_k}{a_j+b_k}\right)_{kj}
 \end{equation}
 with $$\lambda_j=\frac{\prod_l(a_j+b_l)}{\prod_{i<j}(a_i-a_j)\prod_{r>j}(a_j-a_r)}$$ and $$\mu_k=\frac{\prod_l(a_l+b_k)}{\prod_{i<k}(b_i-b_k)\prod_{r>k}(b_k-b_r)}.$$
 Applying this formula to $  \mathscr C_\e({ 1})$ with 
 $$a_j=1+\e\xi_j,\; 1\le j\le N, \;b_k= 1+\e \eta_k,\;1\le k\le N-1,\;b_N=0,$$  we get 
 \begin{eqnarray*}
 \left(  \mathscr C_\e({ 1})^{-1}\right)_{kj}&=&\frac{(-1)^{j+k}\lambda_j(\e )\mu_k(\e )}{2+\e (\xi_j+\eta_k)}\text{ if }k\le N-1;\\
  \left(  \mathscr C_\e({ 1})^{-1}\right)_{Nj}&=&\frac{(-1)^{j+N}\lambda_j(\e )\mu_N(\e )}{1+\e \xi_j}\ ,
 \end{eqnarray*}
 with
 $$\lambda_j(\e )=\frac{2^{N-1}}{\e^{N-1}}\frac{(1+\e \xi_j)\prod_l(1+\e (\xi_j+\eta_l)/2)}{\xi'_j}\ ,\ 1\le j\le N\ ,$$ and 
 $$\mu_k(\e )=\frac{2^{N}}{\e^{N-2}}\frac{\prod_l(1+\e (\xi_l+\eta_k)/2)}{\eta'_k(1+\e \eta_k)}, \; k\le N-1,\; \mu_N(\e )=\frac{\prod _i(1+\e \xi_i)}{\prod _k(1+\e\eta_k)}$$
 where we have set
 $$\xi'_j:=\prod_{i<j}(\xi_i-\xi_j)\prod_{r>j}(\xi_j-\xi_r),\; \eta'_k:=\prod_{i<k}(\eta_i-\eta_k)\prod_{r>k}(\eta_k-\eta_r).$$
 Finally, introduce the following special vectors $V_m(\e )$ and $W_p(\e )$ in $\R ^N$.
 \begin{eqnarray*}
 V_m(\e )_j:&=&(1+\e \xi _j)^m\ ,\ m\ge 0\ ,\ j=1,\dots ,N,\\
 W_p(\e )_k:&=& (1+\e \eta _k)^p(1-\delta _{kN})\ ,\ p\ge 1\ ,\ W_0(\e )_k:=1,\ ,\ k=1,\dots ,N, 
 \end{eqnarray*}
 where $\delta _{kN}$ denotes the Kronecker symbol.
  \begin{lemma}\label{matricesC_e}
The following identities hold  for every $(\xi ,\eta )\in \mathcal X_N$,  uniformly in $\e \ll 1$.
\begin{equation}\label{left}
\left(  \mathscr C_\e(1)^{-1}(1_N)\right)_k=(-1)^{k+1}\mu_k(\e )\ ,\ 1\le k\le N\ .
\end{equation}
\begin{equation}\label{right}
 \left( ^t\mathscr C_\e(1)^{-1}(1_N)\right )_j=(-1)^{j+1}\lambda_j (\e ) +\mathcal O\left (\frac 1{\e ^{N-2}}\right )\ ,\ 1\le j\le N\ .
 \end{equation}
 Furthermore, for every $(\xi ,\eta )\in \mathcal X_N$, there exist a constant $C_{\xi ,\eta }$ and $\theta ^*(\xi ,\eta )$ such that, for the norm $L^\infty (D_{\theta ,\e })$ with $\theta <\theta ^*(\xi ,\eta )$, uniformly for $0<\e \le 1$, we have the estimates 
\begin{eqnarray}
\ &&\vert \mathscr C_\e (z)^{-1}(V_m(\e ))\vert \le C_{\xi ,\eta }\frac{(1+\xi _1)^m}{\e ^{N-2}}\ ,\ m\ge 0\ . \label {CVm}\\
 \ &&\vert ^t\mathscr C_\e (z)^{-1}(W_p(\e ))\vert \le C_{\xi ,\eta }\frac{(1+\eta _1)^{p-1}}{\e ^{N-1}}\ ,\ p\ge 1\ . \label{CWp}\\
\  &&\vert \langle  \mathscr C_\e (z)^{-1}(V_m(\e )),W_p(\e ) \rangle \vert \le  C_{\xi ,\eta }(1+\xi _1)^m(1+\eta _1)^p, m,p\ge 0\ . \label{CVmWp}
\end{eqnarray}
\end{lemma}
\begin{proof} The main algebraic cancellation is displayed in the following lemma.
\begin{lemma}\label{technic}
 Let $a,b$ with $0\le a+b\le  N-1$,
 $$\sum_{j=1}^N (-1)^j\frac{\xi_j^a}{\xi'_j}\sum_{|I|=b,I\subset\{1,\dots,N\}\setminus\{j\}}(\prod_{i\in I} \xi_i)=\left\{\begin{array}{ll}0 &\text{ if } 0\le a+b< N-1\\
 -1&\text{ if } a=N-1, b=0.\end{array}\right.$$
Analogous identities hold for the $\eta$'s.
 \end{lemma}
 \begin{proof}
We view $$\sum_{j=1}^N (-1)^j\frac{\xi_j^a}{\xi'_j}\sum_{|I|=b,I\subset\{1,\dots,N\}\setminus\{j\}}(\prod_{i\in I} \xi_i)$$ as a rational function of $\xi_N$ denoted by $Q_{a,b}(\xi_N)$. Its poles are simple and equal to  $\xi_1,\dots, \xi_{N-1}$. Identifying the residue at each of these poles, we get
 \begin{eqnarray*}
 &&{\rm Res}(Q_{a,b}; \xi_N=\xi_r)=\\
 &&=\frac{(-1)^{r+1}\xi _r^{a}\left [\sum_{|I'|=b-1,I'\subset\{1,\dots,N-1\}\setminus\{r\}}(\prod_{i\in I'} \xi_i)\xi _r+\sum_{|I|=b,I\subset\{1,\dots,N-1\}\setminus \{ r\} }(\prod_{i\in I} \xi_i)\right ]}{\prod _{i<r}(\xi _i-\xi _r)\prod _{N-1\ge j>r}(\xi _r-\xi _j)}\\
 &&+
 \frac{(-1)^{N+1}\xi _r^{a}\sum_{|I|=b,I\subset\{1,\dots,N-1\}}(\prod_{i\in I} \xi_i)}{\prod _{i<N, i\ne r}(\xi _i-\xi _r)}\\
 &&=0\ .
 \end{eqnarray*} 
 Hence $Q_{a,b}(\xi_N)$ is in fact a polynomial. If $a+b<N-1$, it tends to $0$ as $\xi _N$ tends to $\infty $, therefore it is identically $0$.  If $a=N-1$ and $b=0$ it tends to $-1$. This completes the proof.
\end{proof}
Let us come to the proof of Lemma \ref{matricesC_e}. From the formulae giving $ \mathscr C_\e({ 1})^{-1}$, we get, 
if $k\le N-1$,
 \begin{eqnarray*}
 \left(  \mathscr C_\e({ 1})^{-1}({ 1_N})\right)_k&=&(-1)^k\mu_k(\e )\sum_{j=1}^N\frac{(-1)^j\lambda_j(\e )}{2+\e (\xi_j+\eta_k)}\\
 &=&
 (-1)^k\mu_k(\e )\frac {2^{N-2}}{\e ^{N-1}}\sum_{j=1}^N(-1)^j\frac{(1+\e \xi_j)}{\xi'_j}\prod_{l\neq k}\left(1+\e\frac{\xi_j+\eta_l}2\right);
 \end{eqnarray*}
 and 
 \begin{eqnarray*}
\left((  \mathscr C_\e({ 1})^{-1}({ 1_N})\right)_N&=&(-1)^N\mu_N(\e )\sum_{j=1}^N\frac{(-1)^j\lambda_j(\e )}{1+\e \xi_j}\\
 &=&
 (-1)^N\mu_N(\e )\frac {2^{N-1}}{\e ^{N-1}}\sum_{j=1}^N\frac{(-1)^j}{\xi'_j}\prod_{l=1}^{N-1}\left(1+\e\frac{\xi_j+\eta_l}2\right).
 \end{eqnarray*}
 Expanding in powers of $\e $  the above formula giving $ \left(  \mathscr C_\e({ 1})^{-1}({ 1_N})\right)_k$, and using Lemma \ref{technic} with $b=0$, we infer 
$$
\left(  \mathscr C_\e({ 1})^{-1}({ 1_N})\right)_k=(-1)^{k+1}\mu_k(\e )\ ,\ k\le N\ ,
$$
which is (\ref{left}).
We now compute $^t\mathscr C_\e^{-1}(1)({ 1_N})$ in a similar way.
 \begin{eqnarray*}
\left( ^t\mathscr C_\e(1)^{-1}({ 1_N})\right)_j&=&(-1)^j\lambda_j(\e )\left(\sum_{k=1}^{N-1}\frac{(-1)^k\mu_k(\e )}{2+\e(\xi_j+\eta_k)}+\frac{(-1)^N\mu_N(\e )}{1+\e \xi_j}\right)\\
&=:&(-1)^j\lambda_j (\e )\tilde S_j(\e).
 \end{eqnarray*}
 We expand in powers of $\e $ and use again Lemma \ref{technic} with $b=0$, changing $N$ into $N-1$. We get
 \begin{eqnarray*}
 \tilde S_j(\e)&=&\frac{2^{N-1}}{\e^{N-2}}\sum_{k=1}^{N-1}\frac{(-1)^k\prod_{i\neq j}\left(1+\e\frac{\xi_i+\eta_k}2\right)}{\eta'_k(1+\e \eta_k)}+(-1)^N+\mathcal O(\e)\\
 &=&(-1)^N+\mathcal O(\e)+\frac{2^{N-1}}{\e^{N-2}}\sum_{k=1}^{N-1}(-1)^k\left(\e^{N-2}\sum_{p+q=N-2}(-1)^qC^p_{N-1}\frac{\eta_k^{p+q}}{2^p\eta'_k}+\mathcal O(\e^{N-1})\right)\\
 &=&-1+\mathcal O(\e).
 \end{eqnarray*}
 As a consequence, we infer
$$
 \left( ^t\mathscr C_\e(1)^{-1}({ 1_N})\right)_j=(-1)^{j+1}\lambda_j (\e ) (1+\mathcal O(\e))\ ,
$$
which is (\ref{right}). We next prove (\ref{CVm}) and (\ref{CWp}), in the special case $z=1$. Notice that the cases $m=0$ and $p=0$ were addressed above, so we may assume $m,p\ge 1$. We have
$$
(\mathscr C_\e({ 1})^{-1}(V_m(\e ))_{k}=(-1)^k\mu _k(\e )\frac{2^{N-1}}{\e^{N-1}}h_{m,k}(\e,\xi,\eta),
$$
with
$$h_{m,k} (\e,\xi,\eta):=\left \{ \begin{array}{lll}
{\displaystyle \frac 12\sum_{j=1}^N(-1)^{j}(1+\e\xi_j)^{m+1}\; \frac{\prod _{l\ne k}\left (1+\e \frac{\xi _j+\eta _l}2\right )}{\xi '_j}}\ &{\rm if}\ k\le N-1\ ,\\
\   \\
{\displaystyle \sum_{j=1}^N(-1)^{j}(1+\e\xi_j)^{m}\; \frac{\prod _{l=1}^{N-1}\left (1+\e \frac{\xi _j+\eta _l}2\right )}{\xi '_j}}\ &{\rm if}\ k=N\ .
\end{array}\right .$$
Notice that $h_{m,k}$ is a holomorphic function of $\e $. From Lemma \ref{technic} with $b=0$, the derivatives 
$\partial _\e ^rh_{m,k}(0,\xi ,\eta )$ vanish for $r<N-1$. Therefore, from the maximum principle,
\begin{equation}\label{esthmk}
\frac{|h_{m,k}(\e,\xi,\eta)|}{\e^{N-1}}\le \sup _{\vert \zeta \vert =1}\vert h_{m,k}(\zeta, \xi ,\eta )\vert \le C_{\xi ,\eta }(1+\xi_1)^m\ ,\ 0<\e \le 1\ .
\end{equation}
This provides estimate (\ref{CVm}) for $z=1$. Estimate (\ref{CWp}) for $z=1$ is obtained in the same way, by observing that
$$^t\mathscr C_\e (1)^{-1}(W_p(\e ))=(-1)^{j}\lambda _j(\e ) \frac{2^N}{\e^{N-2}}g_{p,j}(\e,\xi,\eta), \ p\ge 1\ ,$$
where we  have set
 \begin{eqnarray*}
g_{p,j}(\e,\xi,\eta)&:=& \frac{\e^{N-2}}{2^{N}}\sum_{k=1}^{N-1}(-1)^{k}\frac{\mu_k(\e )(1+\e\eta_k)^p}{2+\e (\xi _j+\eta _k)}\\
&=&\frac 12\sum_{k=1}^{N-1}(-1)^{k}\frac{\prod_{l\ne j}(1+\e (\xi_l+\eta_k)/2)}{\eta'_k}(1+\e\eta_k)^{p-1}\ ,
 \end{eqnarray*}
 and proving in the same way that
 $$
 \frac{|g_{p,j}(\e,\xi,\eta)|}{\e^{N-2}}\le \sup_{|\zeta |=1}|g_{p,j}(\zeta ,\xi,\eta)|\le C_{\xi ,\eta }(1+\eta_1)^{p-1}.
$$
 As for (\ref{CVmWp}) for $z=1$, we compute
 $$\langle  C_\e (1)^{-1}(V_m(\e )),W_p(\e ) \rangle =    \frac{2^{2N-1}}{\e^{2N-3}}f_{pm}(\e,\xi,\eta)\ ,$$
 where 
\begin{eqnarray*}
&&f_{0m}(\e,\xi,\eta)=\\
&=&\frac 12\sum_{1\le j\le N, 1\le k\le N-1}(-1)^{j+k}\frac{(1+\e\xi_j)^{m+1}\prod_{l\neq k}(1+\e\frac{\xi_j+\eta_l}2)\prod_{i}(1+\e\frac{\xi_i+\eta_k}2)}{(1+\e \eta _k)\xi'_j\eta'_k}\\
&&+\frac{\e ^{N-2}}{2^N}\frac{\prod_{i=1}^N(1+\e\xi_i)}{\prod _{l=1}^{N-1}(1+\e \eta _l)}\sum _{j=1}^N (-1)^{j+N}\frac{(1+\e\xi_j)^{m}\prod_{l}(1+\e\frac{\xi_j+\eta_l}2)}{\xi'_j}\\
&&=:\tilde f_{0m}(\e ,\xi ,\eta )+(-1)^N\frac{\e ^{N-2}}{2^N}\frac{\prod_{i=1}^N(1+\e\xi_i)}{\prod _{l=1}^{N-1}(1+\e \eta _l)}h_{mN}(\e ,\xi ,\eta )\ ,
\end{eqnarray*}
 and, for $p\ge 1$, 
\begin{eqnarray*}
&&f_{pm}(\e,\xi,\eta)=\\
&=&\frac 12\sum_{1\le j\le N, 1\le k\le N-1}(-1)^{j+k}\frac{(1+\e\xi_j)^{m+1}(1+\e\eta_k)^{p-1}\prod_{l\neq k}(1+\e\frac{\xi_j+\eta_l}2)\prod_{i}(1+\e\frac{\xi_i+\eta_k}2)}{\xi'_j\eta'_k}\ .
\end{eqnarray*}
 Since we already proved  in  (\ref{esthmk}) that
 $$\vert h_{mN} (\e ,\xi ,\eta )\vert \le C_{\xi \eta}\e ^{N-1}(1+\xi _1)^m\ ,$$
 estimate (\ref{CVmWp}) will be a consequence of 
$$
 \vert \tilde f_{0m}(\e ,\xi \eta)\vert \le C_{\xi, \eta }\e ^{2N-3}(1+\xi _1)^m\ ,\ \vert \tilde f_{pm}(\e ,\xi ,\eta)\vert \le C_{\xi, \eta }\e ^{2N-3}(1+\xi _1)^m(1+\eta _1)^{p-1}
 $$
 for $m\ge 0, p\ge 1$.
Notice that $\tilde f_{0m}(\zeta ,\xi \eta)$ and $f_{pm}(\zeta, \xi ,\eta )$ are holomorphic function of $\zeta $ for $\vert \zeta \vert <\e ^*$. Using  again the maximum principle, the proof of the lemma will be complete if we prove that
for any $0\le r< 2N-3$, all the derivatives   $\partial _\e ^rf_{pm}(0,\xi,\eta)$ and $\partial _\e ^r\tilde f_{0m}(0,\xi,\eta)$ vanish.\\ 
 Indeed, such a derivative involves a sum of terms
$$\sum_{1\le j\le N, 1\le k\le N-1}\frac{(-1)^{j+k}\xi_j^{a}\eta_k^t}{\xi'_j\eta'_k}\sum_{ |L|=c, |I|=d \atop k\notin L}\prod_{l\in L}(\xi_j+\eta_l)\prod_{i\in I}(\xi_i+\eta_k)$$
with $a+t+c+d=r\le 2N-4$.
We symmetrize this expression by writing
$$\sum_{|I|=d}\prod_{i\in I}(\xi_i+\eta_k)=(\xi_j+\eta_k)\sum_{|I|=d-1\atop j\notin I}\prod_{i\in I}(\xi_i+\eta_k)+\sum_{|I|=d\atop j\notin I}\prod_{i\in I}(\xi_i+\eta_k)$$ which gives rise to a sum of terms
$$\sum_{1\le j\le N, 1\le k\le N-1}\frac{(-1)^{j+k}\xi_j^{a}\eta_k^t}{\xi'_j\eta'_k}\sum_{ |L|=c, |I|=d \atop k\notin L, j\notin I}\prod_{l\in L}(\xi_j+\eta_l)\prod_{i\in I}(\xi_i+\eta_k)$$
with $a+t+c+d=r\le 2N-4$.
Expanding everything, we have to calculate 
$$\sum_{1\le j\le N, 1\le k\le N-1}\frac{(-1)^{j+k}\xi_j^{a+s}\eta_k^{t+u}}{\xi'_j\eta'_k}\sum_{ |L|=c-s,  |I|=d-u \atop k\notin L, j\notin I}\prod_{l\in L}\eta_l\prod_{i\in I}\xi_i$$ with either $a+s+d-u<N-1$ or $t+u+c-s<N-2$ since 
$(a+s+d-u)+(r+u+c-s)=a+t+c+d\le 2N-4$. Using Lemma \ref{technic}, we get that all these terms vanish.
\s 
Finally, we show how to pass from estimates (\ref{CVm}), (\ref{CWp}),  (\ref{CVmWp}) for $z=1$ to uniform estimates
for $z\in D_{\theta ,\e }$. We write 
\begin{eqnarray*}
\mathscr C_\e (z)&=&\mathscr C_\e (1)+(1-z)\dot {\mathscr C}_\e\\
&=&\mathscr C_\e (1)(I+(1-z)\mathscr C_\e (1)^{-1}\dot {\mathscr C}_\e )\\
&=&(I+(1-z)\dot {\mathscr C}_\e \mathscr C_\e (1)^{-1})\mathscr C_\e (1)\ .
\end{eqnarray*}
In view of the expressions of $\dot{\mathscr C}_\e $ and of $\mathscr C_\e (1)^{-1} $, we have
$$\vert \mathscr C_\e (1)^{-1}\dot {\mathscr C}_\e \vert  +\vert \dot {\mathscr C}_\e \mathscr C_\e (1)^{-1}\vert   \le \frac{A(\xi ,\eta )}{\e ^{2(N-1)}}\ .$$
Consequently, for $z\in D_{\theta ,\e }$, we have 
$$\vert 1-z\vert (\vert \mathscr C_\e (1)^{-1}\dot {\mathscr C}_\e \vert +\vert \dot {\mathscr C}_\e \mathscr C_\e (1)^{-1}\vert  ) \le A(\xi ,\eta )\theta \le \frac 12$$
if $\theta \le \theta ^*(\xi ,\eta )$, and the matrices
$$S_\e (z):=I+(1-z)\mathscr C_\e (1)^{-1}\dot {\mathscr C}_\e \ ,\ \hat S_\e (z):=I+(1-z)\dot {\mathscr C}_\e \mathscr C_\e (1)^{-1}$$
are invertible, with
$$\vert S_\e (z)^{-1}\vert + \vert \hat S_\e (z)^{-1}\vert \le 2$$
Estimates (\ref{CVm}), (\ref{CWp}),  (\ref{CVmWp}) on $D_{\theta ,\e }$ are then consequences of the estimates for $z=1$ and of the identities
\begin{eqnarray*}
\mathscr C_\e (z)^{-1}V_m(\e )&=&S_\e (z)^{-1}\mathscr C_\e (1)^{-1}V_m(\e )\ ,\\
^t\mathscr C_\e (z)^{-1}W_p(\e )&=&\hat S_\e (z)^{-1}\, ^t\mathscr C_\e (1)^{-1}W_p(\e )\ ,
\end{eqnarray*}
and \begin{eqnarray*}
&&\langle \mathscr C_\e (z)^{-1}V_m(\e ),W_p(\e )\rangle =\langle S_\e (z)^{-1}\mathscr C_\e (1)^{-1}V_m(\e ),W_p(\e )\rangle \\
&=& \langle \mathscr C_\e (1)^{-1}V_m(\e ),W_p(\e )\rangle - (1-z)\langle \dot{\mathscr C}_\e \mathscr C_\e (1)^{-1}S_\e (z)^{-1}\mathscr C_\e (1)^{-1}V_m(\e ),\, ^t\mathscr C_\e (1)^{-1}W_p(\e )\rangle\ . 
\end{eqnarray*}
This completes the proof of Lemma \ref{matricesC_e}.
 \end{proof}
As a consequence of Lemma \ref{matricesC_e}, we get
\begin{lemma}\label{Vectors}
For $\theta <\theta ^*(\xi ,\eta )$, the following matrix expansions hold in $L^\infty (D_{\theta ,\e})$  as $\e , \delta $ tend to $0$.
\begin{eqnarray*}
\mathscr A_{\delta ,\e}(z)\mathscr C_\e({ z})^{-1}&=&\left(\frac 1{\rho_a}\right)_{1\le a\le q}\otimes ^t\mathscr C_\e({ z})^{-1}(1_N)+\mathcal O\left(\frac{\delta}{\e^{N-1}}\right)\\
\mathscr C_\e({ z})^{-1}\mathscr B_{\delta ,\e}(z)&=&\mathscr C_\e({ z})^{-1}(1_N)\otimes \left(\frac {\Psi_{2b}}{\sigma_b}\right)_{1\le b\le q}+\mathcal O\left(\frac{\delta}{\e^{N-2}}\right)
\end{eqnarray*}
Furthermore, for $z\in D_{\theta ,\e}\ $, the vectors $$\mathscr A_{\delta ,\e}(z)\mathscr C_\e({ z})^{-1}({ 1_N}), \; ^t\mathscr B_{\delta ,\e}(z)^t\mathscr C_\e({ z})^{-1}({ 1_N})$$ 
and the matrix 
$$\mathscr A_{\delta ,\e}(z) \mathscr C_\e( z)^{-1}\mathscr B_{\delta ,\e}(z)$$
are uniformly bounded as $\e, \delta$ tend to $0$.
\end{lemma}
\begin{proof} Recall that
\begin{eqnarray*}
{\mathscr A}_{\delta ,\e}(z)&=&\left(\left(\frac{\rho_a-\delta z(1+\e\eta_k)\Psi_{2a-1}}{\rho_a^2-\delta^2(1+\e \eta_k)^2}\right)_{1\le a\le q\atop 1\le k\le N-1},\left (\frac 1{\rho_a}\right )_{1\le a\le q}\right)\\
{\mathscr B}_{\delta ,\e}(z)&=&\left( \frac{\delta(1+\e\xi_j)-\sigma_b z\Psi_{2b}}{\delta^2(1+\e \xi_j)^2-\sigma_b^2}\right)_{1\le j\le N\atop 1\le b\le q}.
\end{eqnarray*}
Expanding in powers of $\delta $, we obtain 
\begin{eqnarray*}
\mathscr A_{\delta ,\e}(z) &=&\sum_{p=0}^\infty \delta^p\mathscr A_p(\e ,z)\ ,\
\mathscr A_p(\e ,z):=U_p(z)\otimes W_p(\e )\ , \ p\ge 0,\\
\mathscr B_{\delta ,\e}(z) &=&\sum_{m=0}^\infty \delta^m\mathscr B_m(\e ,z)\ ,\ 
\mathscr B_m(\e ,z):= V_m(\e )\otimes T_m(z)\ ,\ m\ge 0,
\end{eqnarray*}
where we have set
\begin{eqnarray*}
U_p(z) &:=&\left(\left\{\begin{array}{l}\displaystyle{\frac{1}{\rho _a^{p+1}}}\text{ if } p\text{ is even, }\\ \ \\
\displaystyle{ -\frac{z\Psi _{2a-1}}{\rho _a^{p+1}}}\text{ if } p\text{ is odd}. \end{array}\right. \right )_{1\le a\le q}\\
T_m(z)&:=&\left(\left\{\begin{array}{l}\displaystyle{\frac{z\Psi _{2b}}{\sigma _b^{m+1}}}\text{ if } m\text{ is even, }\\
\displaystyle{-\frac{1}{\sigma _b^{m+1}}}\text{ if } m\text{ is odd}. \end{array}\right. \right )_{1\le b\le q}
\end{eqnarray*}
Using these formulae, we get
 \begin{eqnarray*}
 \mathscr A_{\delta ,\e}(z)\mathscr C_\e({ z})^{-1}&=&\sum_{p=0}^\infty \delta^pU_p(z)\otimes ^t\mathscr C_\e({ z})^{-1}W_p(\e ),\\
 \mathscr C_\e({ z})^{-1}\mathscr B_{\delta ,\e}(z)&=&\sum_{m=0}^\infty \delta^m \mathscr C_\e({ z})^{-1}V_m(\e )\otimes T_m(z)\ ,\\
  \mathscr A_{\delta ,\e}(z)\mathscr C_\e({ z})^{-1}(1_N)&=&\sum_{p=0}^\infty \delta^p\langle W_p(\e ),\mathscr C_\e({ z})^{-1}(1_N)\rangle U_p(z)\ ,\\ 
^t\mathscr B_{\delta ,\e}(z)^t \mathscr C_\e({ z})^{-1}(1_N)&=&\sum_{m=0}^\infty \delta^m \langle \mathscr C_\e({ z})^{-1}V_m(\e ), 1_N\rangle T_m(z)\ ,\\
\mathscr A_{\delta ,\e}(z) \mathscr C_\e( z)^{-1}\mathscr B_{\delta ,\e}(z)&=&\sum_{p,m\ge 0}\langle \mathscr C_\e({ z})^{-1}V_m(\e ),W_p(\e )\rangle U_p(z)\otimes T_m(z)\ .
\end{eqnarray*}
The statement is then a direct consequence of Lemma \ref{matricesC_e}.
\end{proof}
Let us complete the proof of Proposition \ref{Approximate}. In view of  Lemma \ref{Vectors}, and of formulae (\ref{Xq}), (\ref{X'q}), (\ref{YN}), (\ref{Y'N}), we observe that
$$\mathscr J_{\delta ,\e}(z)=\mathscr E(z)-\delta\mathscr A_{\delta ,\e}(z) \mathscr C_\e(z)^{-1}\mathscr B_{\delta ,\e}(z)=\mathscr E(z)+\mathcal O(\delta )$$
is invertible for $z\in D_{\theta ,\e }$ and $\delta $ and $\e $ small enough, and we get  the formulae
\begin{eqnarray*}
X_q^{\delta ,\e}(z)&=&\mathscr J_{\delta ,\e}(z)^{-1}({\bf \Psi_q}-\delta\mathscr A_{\delta ,\e}(z) \mathscr C_\e(z)^{-1}{ 1}_N)\\
\hat X_q^{\delta ,\e}(z)&=&^t\mathscr J_{\delta ,\e}(z)^{-1}({ 1}_q-\delta^t\mathscr B_{\delta ,\e}(z)\, ^t\mathscr C_\e(z)^{-1}{ 1}_N)\\
Y_N^{\delta ,\e}(z)&=&\delta \mathscr C_\e(z)^{-1}({ 1}_N-\mathscr B_{\delta ,\e}(z)X_q^{\delta ,\e}(z))\\
\hat Y_N^{\delta ,\e}(z)&=&\delta\, ^t \mathscr C_\e(z)^{-1}({ 1}_N-^t\mathscr A_{\delta ,\e}(z)\hat X_q^{\delta ,\e}(z)).
\end{eqnarray*}
Furthermore, observing that $\mathscr E(z)$ is invertible for $z$ in a fixed neighborhood of $\D $, we have, if $z\in D_{\theta ,\e}\ $,
$$\mathscr E(z)^{-1}=\mathscr E(1)^{-1} +O(\e ^{2(N-1)})\ .$$
Applying again Lemma Lemma \ref{Vectors}, this completes the proof.
\end{proof} 

Let us complete the proof of Proposition \ref{uepsilon}.  Proposition \ref{Approximate} and the expression of $(u^{\delta,\e})'(1)$ lead to
 \begin{eqnarray*}
 (u^{\delta,\e})'(z)&=&\left\langle\dot{\mathscr C_{\delta,\e}}\left (\begin{array}{ll}{X}_q^{\delta ,\e}(z)\\{Y}_N^{\delta ,\e}(z)\end{array}\right),\left (\begin{array}{ll}\hat {X}_q^{\delta ,\e}(z)\\ \hat {Y}_N^{\delta ,\e}(z)\end{array}\right)\right\rangle\\
 &=& \left\langle\dot{\mathscr C_{\delta,\e}}\left (\begin{array}{ll} \mathscr E(z)^{-1}{\bf \Psi_q}+\mathcal O(\delta)\\
 \alpha\delta{ \mathscr C_\e}(z)^{-1}(1_N)+ \mathcal O\left(\frac{\delta^2}{\e^{N-2}}\right) \end{array}\right),
  \left (\begin{array}{ll}^t\mathscr E(z)^{-1}{ 1}_q+\mathcal O(\delta)\\\beta\delta ^t{ \mathscr C_\e}(z)^{-1}(1_N)+\mathcal O\left(\frac{\delta^2}{\e^{N-1}}\right)\end{array}\right)\right\rangle \ .
  \end{eqnarray*}
 Observing that 
$$\dot{\mathscr C}_{\delta ,\e }=\left (\begin{array}{ll} \mathcal O(1)  & \mathcal O(\delta)\\
\mathcal O(1) & \frac{1}{\delta} \dot{\mathscr C}_\e  \end{array}\right )=\left (\begin{array}{ll} \mathcal O(1)  & \mathcal O(\delta)\\
\mathcal O(1) & \mathcal O \left (\frac{1}{\delta \e }\right )  \end{array}\right )\ ,$$
 we infer, if $\e \ll \delta $, 
\begin{eqnarray*} 
 (u^{\delta,\e})'(z)&=& \left\langle \left (\begin{array}{ll}  \mathcal O(1)+\mathcal O\left (\frac{\delta ^2}{\e ^{N-2}}\right )\\
 \mathcal O(1)+\alpha \dot {\mathscr C}_\e \mathscr C_\e (z)^{-1}(1_N)+\mathcal O\left (\frac{\delta }{\e ^{N-1}}\right ) \end{array}\right),
 \left (\begin{array}{ll} \mathcal O(1)\\   \beta \delta ^t\mathscr C_\e (z)^{-1}(1_N)+\mathcal O\left (\frac {\delta^2}{\e ^{N-1}}\right ) \end{array}\right)
 \right\rangle\\   
&=&\alpha \beta \delta \langle \dot {\mathscr C}_\e \mathscr C_\e (z)^{-1}(1_N), ^t\mathscr C_\e (z)^{-1}(1_N)  \rangle +\mathcal O\left ( \frac{\delta ^2}{\e ^{2(N-1)}} \right ).
\end{eqnarray*}
Furthermore, writing again
$$\mathscr C_\e (z)^{-1}=S_\e (z)^{-1}\mathscr C_\e (1)^{-1}=\mathscr C_\e (1)^{-1}-(1-z)\dot{\mathscr C}_\e \mathscr C_\e (1)^{-1}S_\e (z)^{-1}\mathscr C_\e (1)^{-1}, $$
we infer
$$  (u^{\delta,\e})'(z)=\alpha \beta \delta \langle \dot {\mathscr C}_\e \mathscr C_\e (1)^{-1}(1_N), ^t\mathscr C_\e (1)^{-1}(1_N)  \rangle +\mathcal O\left ( \frac{\delta ^2+\theta }{\e ^{2(N-1)}} \right )\ .  $$
We claim that the product 
$$\alpha \beta =\left(1-\left \langle \mathscr E^{-1} {\bf \Psi_q},\left(\frac{\Psi_{2b}}{\sigma_b}\right)\right \rangle\right)\left(1-\left \langle^t \mathscr E^{-1}{1_q},\left(\frac{1}{\rho_a}\right)\right \rangle\right)$$ is not zero. Indeed, write $Z_b:=\mathscr E^{-1} {\bf \Psi_q}$ so that, for $a=1,\dots, q$, 
$$\sum_{b=1}^q \frac{\rho_a-\sigma_b\Psi_{2a-1}\Psi_{2b}}{\rho_a^2-\sigma_b^2}Z_b=\Psi_{2a-1}$$ or 
\begin{equation}\label{EZ}
\sum_{b=1}^q \frac{\rho_a\Psi_{2a-1}^{-1}-\sigma_b\Psi_{2b}}{\rho_a^2-\sigma_b^2}Z_b=1\ ,\ a=1,\dots ,q.
\end{equation}
Assume $\alpha =0$, namely
\begin{equation}\label{alpha0}
\sum_{b=1}^q Z_b\frac{\Psi_{2b}}{\sigma_b}=1\ .
\end{equation}
 Then substracting (\ref{EZ}) from (\ref{alpha0}) leads to
$$\rho _a\sum_{b=1}^q \frac{\rho_a-\sigma_b\Psi_{2a-1}^{-1}\Psi_{2b}^{-1}}{\rho_a^2-\sigma_b^2}\frac{\Psi_{2b}Z_b}{\sigma_b}=0\ ,\ a=1,\dots ,q.$$
This is a contradiction since, from Theorem \ref{FT}, the matrix 
$$\left( \frac{\rho_a-\sigma_b \Psi_{2a-1}^{-1}\Psi_{2b}^{-1}}{{\rho_a^2-\sigma_b^2}}\right)_{1\le a,b\le q} $$ is known to be invertible since $\Psi_{2a-1}^{-1}$ and $\Psi_{2b}^{-1}$ are complex numbers of modulus 1. 
We conclude that $\alpha \ne 0$. A similar argument leads to $\beta \ne 0$.

Eventually, we calculate 
 \begin{eqnarray*}
\langle\dot{\mathscr C_\e}\mathscr C_\e^{-1}({ 1_N}), ^t\mathscr C_\e^{-1}({ 1_N})\rangle&=&\sum_{1\le j\le N\atop 1\le k\le N-1}\frac{(-1)^{j+k} \lambda_j (\e )\mu_k (\e )(1+\e \eta_k)}{(1+\e \xi_j)^2-(1+\e \eta_k)^2}(1+\mathcal O(\e ))\\
 &=&\frac{2^{2(N-1)}}{\e^{2(N-1)}}\left(\sum_{1\le j\le N}\sum _{1\le k\le N-1}\frac{(-1)^{j+k}}{(\xi_j-\eta_k)\xi'_j\eta'_k}+\mathcal O(\e)\right).
\end{eqnarray*} 
Considering the analytic expression $$\sum_{1\le j\le N}\sum _{1\le k\le N-1}\frac{(-1)^{j+k}}{(\xi_j-\eta_k)\xi'_j\eta'_k}$$ as a function of $\xi_N$, the pole $\eta_{N-1}$ appears only once. Therefore this  quantity does not vanish if $(\xi ,\eta )$ belongs to  an open dense  set $\mathcal X'''_N$ of V. In that case, if $\theta \le \theta ^*(\xi ,\eta )$ small enough, and $\e \ll \delta \ll 1$, we conclude
 $$\forall z\in D_{\theta ,\e }\ ,\ | (u^{\delta,\e})'(z)|\ge\frac {C\delta }{\e^{2(N-1)}}\ .$$ 
 This completes the proof of Proposition \ref{derivee}. 
\\
\end{proof}

 \chapter{Geometry of the Fourier transform}\label{chapter geometry}\index{Non linear Fourier transform}
 
 This chapter is devoted to the proof of two results. The first one  describes the restriction of the symplectic form to the finite dimensional manifolds made of symbols corresponding to pairs of Hankel operators with a given finite list of multiplicities. These manifolds turn out to involutive, and are the disjoint union of symplectic manifolds on which the nonlinear Fourier transform defines action angle variables for the cubic Szeg\H{o} flow. The proof of this result uses the evolution of the nonlinear Fourier transform through the flows of the Szeg\H{o} hierarchy introduced in \cite{GG1} and used in \cite{GG4}. 
 The second result characterizes the sets of symbols associated to pairs of Hankel operators with the same singular values \index{singular value}and Blaschke products admitting a given set of zeroes, as classes of some unitary equivalence. These sets are precisely tori obtained in the above symplectic manifolds by fixing the action variables and making angles vary. 
 
 \section{Evolution under the Szeg\H{o} hierarchy} \label{szegohier}

The Szeg\H{o} hierarchy was introduced in \cite{GG1} and used in \cite{GG2} and \cite{GG4}. 
In \cite{GG2}, it was used to identify the symplectic form on the generic part of $\mathcal V(d)$.
Similarly, our purpose in this section is  to establish preliminary formulae, towards the identification of the symplectic form 
on $ \mathcal V_{(d_1,\dots ,d_n)}$ in section \ref{symplectic}.
\s
For the convenience of the reader, we recall the main properties of the hierarchy. For $y>0$ and $u\in H^{\frac 12}_+$, we set
$$J^y(u)= ((I+yH_u^2)^{-1}(1)\vert 1)\ .$$
Notice that the connection with the Szeg\H{o} equation is made by
$$E(u)=\frac 14(\pa _y^2J^y_{\vert y=0}-(\pa _yJ^y_{\vert y=0})^2)\ .$$
Thanks to formula (\ref{J}), $J^y(u)$ is a function of the singular values $s_r(u)$\index{singular value}.
For every $s>\frac 12$, $J^y $ is a smooth real valued function on $H^s_+$, and its Hamiltonian vector field  is given by
$$X_{J^y}(u)=2iy w^yH_uw^y\ ,\ w^y:=(I+yH_u^2)^{-1}(1)\ ,$$
which is a Lipschitz vector field on bounded subsets of $H^s_+$. 
By the Cauchy--Lipschitz theorem, the evolution equation 
\begin{equation}\label{hierarchy}
\dot u=X_{J^y}(u)
\end{equation}
admits local in time solutions for every initial data in $H^s_+$ for $s>1$, and the lifetime is bounded from below if the 
data are bounded in $H^s_+$. We recall that this evolution equation admits a Lax pair structure (\cite{GG4})\index{Lax pair}. 
\begin{theorem}\label{Laxhier}
For every $u\in H^s_+$, we have
\begin{eqnarray*}
H_{iX_{J^y}(u)}&=&H_u F_u^y+F_u^yH_u\ ,\\
K_{iX_{J^y}(u)}&=&K_uG_u^y+G_u^yK_u\ ,\\
 G_u^y(h)&:=&-yw^y\, \Pi (\overline {w^y}\, h)+y^2H_uw^y\, \Pi (\overline {H_uw^y}\, h)\ ,\\
 F_u^y(h)&:=&G_u^y(h)-y^2(h\vert H_uw^y)H_uw^y\ .
\end{eqnarray*}
If $u\in C^\infty (\mathcal I ,H^s_+)$ is a solution of equation (\ref{hierarchy}) on a time interval $\mathcal I$, then
\begin{eqnarray*}
\frac{dH_u}{dt}&=&[B_u^y,H_u]\ ,\ \frac{dK_u}{dt}=[C_u^y,K_u]\ ,\\
B_u^y&=&-iF_u^y\ ,\ C_u^y=-iG_u^y\ .
\end{eqnarray*}
\end{theorem}
In particular, $\Sigma _H(u_0)=\Sigma _H(u(t))$ and $\Sigma _K(u_0)=\Sigma _K(u(t))$ for every $t$, therefore $J^y(u(t))$ is a constant $J^y$. We now state the main result of this section.
\begin{theorem}\label{hierarchyevol}
Let $u_0\in H^s_+\ ,\ s>1,$  with $$\Phi(u_0)=((s_r),(\Psi_r)).$$
The solution of  $$\dot u=X_{J^y}(u)\ ,\ u(0)=u_0\ ,$$  is characterized by
$$\Phi(u(t))=((s_r),(\expo_{i\omega _rt}\Psi _r))\ ,\ \omega _r:=(-1)^{r-1}\frac {2yJ^y}{1+ys_r^2}\ .$$
\end{theorem}
\begin{proof}
Let $\rh \in \Sigma _H(u_0)$. Denote by  $u_\rho$ the orthogonal projection of $u$ on $E_u(\rho)$.
Hence, $u_\rho=1\!{\rm l}_{\{ \rho^2\}}(H_u^2)(u)$. Let us differentiate this equation with respect to time. 
We get from the Lax pair structure\index{Lax pair}
\begin{eqnarray*}
\frac{d u_\rho}{dt}&=& [B_u^y,1\!{\rm l}_{\{ \rho^2\}}(H_u^2)](u)+1\!{\rm l}_{\{ \rho^2\}}(H_u^2)[B_u^y,H_u](1)\\
&=& B_u^y(u_\rho)-1\!{\rm l}_{\{ \rho^2\}}(H_u^2)(H_u(B_u^y(1)))\ .
\end{eqnarray*}
Since $B_u^y(1)=iyJ^yw^y$, and since  $1\!{\rm l}_{\{ \rho^2\}}(H_u^2)(H_uw^y)=\frac 1{1+y\rho^2}u_\rho$, we get

\begin{equation}\label{deriveurhoJ}
\frac{d u_\rho}{dt}=B_u^y(u_\rho)+i\frac{yJ^y}{1+y\rho^2}u_\rho\ .
\end{equation}

On the other hand, differentiating the equation
$$\rho u_\rho=\Psi  H_u(u_\rho)$$
one obtains
$$
\rho \frac{d u_\rho}{dt}=\dot \Psi H_u(u_\rho)+\Psi\left([B_u^y,H_u](u_\rho)+H_u\left( \frac{d u_\rho}{dt}\right)\right)$$
Hence, using the expression (\ref{deriveurhoJ}), we get
$$\rho\left (B_u^y(u_\rho)+i\frac{yJ^y}{1+y\rho^2}u_\rho\right )=\left (\dot \Psi-i\frac{yJ^y}{1+y\rho^2}\Psi \right )H_u(u_\rh )+\Psi B_u^yH_u(u_\rho)\ ,$$
hence
$$[B_u^y,\Psi]H_u(u_\rho)=\left (\dot \Psi-2i\frac {yJ^y}{1+y\rho^2}\Psi \right )H_u(u_\rh ).$$
It remains to prove  that the left hand side of this equality is zero.
We first show that, for any $p\in\D$ such that $\chi_p$ is a factor of $\chi$, for every $e\in E_u(\rho)$ such that $\chi _pe\in E_u(\rh )$,
$[B_u^y,\chi_p](e)=0$. We write
\begin{eqnarray*}
i[B_u^y,\chi_p](e)&=&-yw^y\,( \Pi (\overline {w^y}\, \chi_pe)-\chi_p \Pi (\overline {w^y}\, e))\\
&+&y^2H_uw^y\, (\Pi (\overline {H_uw^y}\, \chi_p e)-\chi_p\Pi (\overline {H_uw^y}\,  e))\\
&-&y^2((\chi_p e\vert H_uw^y)H_uw^y-\chi_p(e\vert H_uw^y)H_uw^y) \\
\end{eqnarray*}
We already used that, for any function $f\in L^2$, $\Pi(\chi_pf)-\chi_p\Pi(f)$ is proportional to $\frac 1{1-\overline pz}$. 
Hence, we obtain
\begin{eqnarray*}
i[B_u^y,\chi_p](e)&=&-yw^y \frac{c}{1-\overline pz}+y^2H_uw^y\frac{\tilde c}{1-\overline pz}\\
&-&y^2((\chi_p e\vert H_uw^y)H_uw^y-\chi_p(e\vert H_uw^y)H_uw^y) \\
\end{eqnarray*}
with 
$$c=( \Pi (\overline {w^y}\, \chi_pe)-\chi_p \Pi (\overline {w^y}\, e)\vert 1)=( \chi_pe\vert w^y)-(\chi_p\vert 1)(e\vert w^y)$$ and 
$$\tilde c=(\Pi (\overline {H_uw^y}\, \chi_p e)-\chi_p\Pi (\overline {H_uw^y}\,  e)\vert 1)=( \chi_pe\vert H_u(w^y))-(\chi_p\vert 1)(e\vert H_u(w^y))\ .$$
Now,  for any $v\in E_u(\rho)$
$$( v\vert w^y)=( v\vert 1\!{\rm l}_{\{ \rho^2\}}(H_u^2)(w^y))=\frac 1{1+y\rho^2}( v\vert 1)$$
hence $c=0$.
On the other hand,
\begin{eqnarray*}
(v\vert H_u w^y)&=&(v\vert 1\!{\rm l}_{\{ \rho^2\}}(H_u^2)(H_u(w^y)))=\frac 1{1+y\rho^2}( v\vert u_\rho)\\
&=&\frac 1{1+y\rho^2}(v\vert H_u(1))=\frac 1{1+y\rho^2}(1\vert H_u(v)) \ .
\end{eqnarray*}
We infer
$$i[B_u^y,\chi_p](e)=C(z)\frac 1{1+y\rho^2}y^2H_uw^y$$
where 
\begin{eqnarray*}
C(z)&=&\frac 1{1-\overline pz}\left ((1\vert H_u(\chi_p e))-(\chi_p\vert 1)(1\vert H_u(e)\right)
-(1\vert H_u(\chi_p e)+\chi_p(1\vert H_u(e))\\
&=&(1\vert H_u(\chi_p e))(\frac 1{1-\overline pz}-1)+(1\vert H_u(e))(\chi_p+\frac{p}{1-\overline pz})\\
&=&\frac z{1-\overline pz}(\overline p(1\vert H_u(\chi_p e))+(1\vert H_u(e)))\ .
\end{eqnarray*}
We claim that $H_u(e)=\chi_pH_u(\chi_p e)$. Indeed, from the assumption $e\in E_u(\rh )$ and $\chi _pe\in E_u(\rh )$, we can write
 $e=fH_u(u_\rh )$ with $\Pi (\Psi \overline f)=
\Psi \overline f$ and $\Pi (\Psi \overline {\chi _pf})=
\Psi \overline {\chi _pf}$. From Lemma \ref{crucialHuGeneral}, we infer
$$H_u(\chi _pe)=\rh \Psi \overline {\chi _pf}H_u(u_\rh )\ ,\ H_u(e)=\rh \Psi \overline fH_u(u_\rh )\ .$$
This proves the claim. Since $(1\vert \chi_p)=-\overline p$, we conclude that $C(z)=0$. Hence $[B^y_u,\chi_p](e)=0$. Arguing as in the previous section, we conclude that $[B_u,\chi]H_u(u_\rho)=0$.
\s
It remains to consider the  other eigenvalues. Let $\sigma\in \Sigma _K(u_0)$.  Denote by $u'_\sigma$ the orthogonal projection of $u$ on $F_u(\sigma )$. We compute  the derivative of $u'_\sigma=1\!{\rm l}_{\{\sigma^2\}}(K_u^2)(u)$ as before.
From the Lax pair formula, we get
\begin{eqnarray*}
\frac{d u'_\sigma}{dt}&=& [C_u^y,1\!{\rm l}_{\{\sigma^2\}}(K_u^2)](u)+1\!{\rm l}_{\{\sigma^2\}}(K_u^2)[B^y_u,H_u](1)\\
&=& C_u^y(u'_\sigma)+1\!{\rm l}_{\{\sigma^2\}}(K_u^2)(B_u^y(u)-C_u^y(u)-H_u(B_u^y(1)))\\
&=&C_u^y(u'_\sigma)+1\!{\rm l}_{\{\sigma^2\}}(K_u^2)(iy^2(u\vert H_uw^y)H_uw^y+iyJ^y H_u w^y)\\
&=&C_u^y(u'_\sigma)+iy1\!{\rm l}_{\{\sigma^2\}}(K_u^2)H_uw^y
\end{eqnarray*}
since
$(B_u^y-C_u^y)(h)=iy^2(h\vert H_uw^y)H_uw^y$ and $-yH^2_uw^y=w^y-1$ so that $(u\vert-yH_uw^y)=(-yH_u^2w^y\vert 1)=J^y-1$.

We claim that 
\begin{equation}\label{ProjSigma}
1\!{\rm l}_{\{\sigma^2\}}(K_u^2)(H_uw^y)=\frac{J^y}{1+y\sigma^2} u'_\sigma\ .
\end{equation}
Using  $K_u^2=H_u^2-(\cdot\vert u)u$ one gets, for any $f\in L^2_+$
\begin{equation}\label{sandrine}
(I+yH_u^2)^{-1}f=(I+yK_u^2)^{-1}f-y((I+yH_u^2)^{-1}f\vert u)(I+yK_u^2)^{-1}u\ .
\end{equation}
Applying formula (\ref{sandrine}) to $f=u$, we get
$$
H_u w^y=(I+yH_u^2)^{-1}(u)=(I+yK_u^2)^{-1}(u)-y((I+yH_u^2)^{-1}(u)\vert u)(I+yK_u^2)^{-1}(u)\ ,$$ hence
\begin{equation}\label{Hu(w)}
H_u w^y=J^y(I+yK_u^2)^{-1}(u) \ .
\end{equation}
Formula (\ref{ProjSigma}) follows by taking the orthogonal projection on $F_u(\sigma )$.
Using Formula (\ref{ProjSigma}), we get
\begin{equation}\label{deriveUsigma}
\frac{d u'_\sigma}{dt}=C_u^y(u'_\sigma)+iy\frac{J^y}{1+y\sigma^2} u'_\sigma\ .
\end{equation}
On the other hand, differentiating the equation
$$K_u(u'_\sigma)=\sigma \Psi  u'_\sigma $$ 
one obtains
$$
[C_u^y,K_u](u'_\sigma)+K_u\left(\frac{d u'_\sigma}{dt}\right)=\sigma \dot \Psi u'_\sigma+\sigma \Psi  \frac{d u'_\sigma}{dt}.$$
From identity (\ref{Hu(w)}), we use the expression of $\frac{d u'_\sigma}{dt}$ 
obtained above to get
$$\left (\dot \Psi+2i\frac{yJ^y}{1+\sigma^2 y}\Psi \right )u'_\sigma=\sigma[C_u^y,\Psi](u'_\sigma)\ .$$
The result follows once we prove that $[C_u^y,\Psi](u'_\sigma)=0$.

From the arguments developed before, it is sufficient to prove that 
$[C_u^y, \chi_p](f)=0$ for any $f\in F_u(\sigma )$ such that $\chi _pf\in F_u(\sigma )$. 
As before
$$[C_u^y, \chi_p](f)=i\frac{c}{1-\overline pz} yw^y-iy^2H_uw^y\frac{\tilde c}{1-\overline pz}$$
where 
$$c=(( \chi_p-( \chi_p\vert 1))f\vert w^y)$$ and 
$$\tilde c=((\chi_p -( \chi_p\vert 1))f\vert H_u w^y).$$
Notice that $w^y=1-yH_uw^y$, hence $c=-y\tilde c$. 
Let us first prove that $\tilde c=0$. Using formula (\ref{Hu(w)}), 
\begin{eqnarray*}
\tilde c&=&( \chi_p f\vert 1\!{\rm l}_{\{\sigma^2\}}(K_u^2)H_u w^y)-( \chi_p\vert 1)(f\vert  1\!{\rm l}_{\{\sigma^2\}}(K_u^2)H_u w^y)\\
&=&\frac{J^y}{1+y\sigma^2}(( \chi_p -( \chi_p\vert 1))f\vert u)= 0\ ,
\end{eqnarray*}
since, as we already observed at the end of the proof of Theorem \ref{evolszegotexte},
 $$F_u(\sigma )\cap 1^\perp =E_u(\sigma )=F_u(\sigma )\cap u^\perp \ . $$
This completes the proof.
\end{proof}
We close this section by stating a corollary which will be useful for describing the symplectic form on $ \mathcal V_{(d_1,\dots ,d_n)}$.
\begin{corollary}\label{XJy}
On $ \mathcal V_{(d_1,\dots ,d_n)}$, we have 
\begin{equation}\label{decXJy}
X_{J^y}=\sum _{r=1}^n (-1)^{r}\frac{2yJ^y}{1+ys_r^2}\frac{\partial}{\partial \psi _r}\ .
\end{equation}
The vector fields $X_{J^y}, y\in \R _+,$ generate an integrable sub-bundle of rank  $n$ of the tangent bundle of  $ \mathcal V_{(d_1,\dots ,d_n)}$. The  leaves of the corresponding foliation are the isotropic tori 
$$\mathcal T ((s_1,\dots ,s_n),(\Psi _1,\dots ,\Psi _n)):=\Phi ^{-1}\left (\{ (s_1,\dots ,s_n)\} \times \S^1\Psi _1\times \dots \times \S^1\Psi _n  \right )\ ,$$
where $(s_1,\dots ,s_n)\in \Omega _n$  and $(\Psi _1,\dots ,\Psi _n)\in \mathcal B^\sharp _{d_1}\times \dots \times \mathcal B^\sharp _{d_n}$ are given.
\end{corollary}
\begin{proof}
For every $y\in \R _+$,  Theorem  \ref{hierarchyevol} can be rephrased as the following identities for Lie derivatives along $X_{J^y}$.
$$X_{J^y}(s_r)=0\ ,\ X_{J^y}(\chi _r)=0\ ,\ X_{J^y}(\psi _r)=(-1)^{r}\frac{2yJ^y}{1+ys_r^2}\ ,\ r=1,\dots ,n\ .$$
This implies identity (\ref{decXJy})  on $\mathcal V_{(d_1,\dots ,d_n)}$.
Given $n$ positive numbers $y_1>\dots >y_n$, the matrix 
$$\left (\frac{1}{1+y_\ell s_r^2}\right )_{1\le \ell ,r\le n}$$
is invertible. This implies that, for every $u\in  \mathcal V_{(d_1,\dots ,d_n)}$, the vector subspace of $T_u \mathcal V_{(d_1,\dots ,d_n)}$ spanned by 
the $X_{J^y}(u), y\in \R _+$ is exactly
$${\rm span} \left (\frac{\partial}{\partial \psi _r}, r=1,\dots ,n\right )\ .$$
The integrability follows, as well as the identification of the leaves, while the isotropy of the tori comes from the identity
$$\{ J^y, J^{y'}\} =0$$
which was proved in \cite{GG1} and is also a consequence of identity \ref{J(x)Appendix} and of the conservation of the $s_r$'s 
along the Hamiltonian curves of $J^y$, as stated in Theorem \ref{hierarchyevol}. 
\end{proof}

\section{The symplectic form on $ \mathcal V_{(d_1,\dots ,d_n)}$}\label{symplectic}

In this section, we prove that the symplectic form $\omega$ restricted to $ \mathcal V_{(d_1,\dots ,d_n)}$ is given by
\begin{equation}\label{omegasurV(d)}
\omega =\sum _{r=1}^n d\left (\frac{s_r^2}{2}\right )\wedge d\psi _r\ .
\end{equation}
Recall that the variable $\psi _r$ is connected to the Blaschke product $\Psi _r$ through the identity
$$\Psi _r=\expo_{-i\psi _r}\chi _r\  ,$$
where $\chi _r$ is a Blaschke product built with a monic polynomial. Given an integer $k$, we denote by $\mathcal B_k ^\sharp $ the submanifold
of $\mathcal B_k$ made with Blaschke products built with monic polynomials of degree $k$.

Let us first point out that  we get the following result as a corollary.
\begin{corollary}\label{Vinvolutive}
The manifold $ \mathcal V_{(d_1,\dots ,d_n)}$ is an involutive submanifold of $\mathcal V(d)$, where $$d=2\sum _{r=1}^nd_r+n.$$
Moreover, $ \mathcal V_{(d_1,\dots ,d_n)}$ is the disjoint union of the symplectic manifolds 
$$\mathcal W(\chi _1,\dots ,\chi _n):=\Phi ^{-1}(\Omega _n\times (\S ^1\chi _1\times \dots \times \S ^1\chi_n))\ ,$$
on which 
$$\left (\frac{s_r^2}{2}, \psi _r\right )_{1\le r\le n}$$
are action angle variables for the cubic Szeg\H{o} flow.
\end{corollary}

\begin{proof}
From the definition of an involutive submanifold, one has to prove that, at every point $u$ of  $ \mathcal V_{(d_1,\dots ,d_n)}$,  the tangent space $T_u \mathcal V_{(d_1,\dots ,d_n)}$ contains its orthogonal  relatively to $\omega$. We use an argument of dimension.
Namely, one has 
\beno
\dim _\R (T_u \mathcal V_{(d_1,\dots ,d_n)})^\perp&=&\dim _\R T_u\mathcal V(d)-\dim _\R T_u \mathcal V_{(d_1,\dots ,d_n)}\\
&=&2d-(2n+2\sum _{r=1}^nd_r)=2\sum _{r=1}^nd_r\ .
\eeno
One the other hand, from equation (\ref{omegasurV(d)}),  the tangent space to the manifold
$$\mathcal F (u):=\Phi ^{-1}\left (\{ (s_r(u))\}\times \prod_{r=1}^n \expo_{-i\psi _r(u)}\mathcal B_{d_r}^\sharp \right )$$
 is clearly a subset of $(T_u \mathcal V_{(d_1,\dots ,d_n)})^\perp$. Since its dimension equals $2\sum d_r$, we get the equality and hence the first result. The second result is an immediate consequence of the previous sections.
\end{proof}
\begin{remark} As this is the case for any involutive submanifold of a symplectic manifold,  the subbundle $(T \mathcal V_{(d_1,\dots ,d_n)})^\perp $ of 
$T \mathcal V_{(d_1,\dots ,d_n)}$ is integrable. The leaves of the corresponding isotropic foliation are the  manifolds $\mathcal F(u)$ above. 
\end{remark}
Now, we prove equality (\ref{omegasurV(d)}).
We first establish the following lemma, as a consequence of Theorem \ref{hierarchyevol}.
\begin{lemma}
On $\mathcal V_{(d_1,\dots ,d_n)}$, 
$$\omega=\sum _{r=1}^n d\left (\frac{s_r^2}{2}\right )\wedge d\psi _r+\tilde \omega\ .$$
where,  for any $1\le r\le n$, 
$$i_{\frac{\partial}{\partial\psi_r}} \tilde \omega=0\ .$$
\end{lemma}
\begin{proof}
Taking the interior product of both sides of identity (\ref{decXJy}) with the restriction of $\omega $ to $\mathcal V_{(d_1,\dots ,d_n)}$, we obtain
$$-d(\log J^y)=\sum _{r=1}^n (-1)^{r}\frac{2y}{1+ys_r^2}i_{\frac{\partial}{\partial \psi _r}}\omega \ .$$
On the other hand, from formula (\ref{J}),
$$d(\log (J^y))=\sum _{r=1}^n (-1)^r \frac{2y}{1+ys_r^2}d\left (\frac{s_r^2}2\right )\ .$$
Identification of residues in the $y$ variables yields
$$d\left (\frac{s_r^2}2\right )=-i_{\frac{\partial}{\partial \psi _r}}\omega \  ,\ r=1,\dots ,n\ .$$
Since
$$i_{\frac{\partial}{\partial \psi _r}}\left (\sum _{r'=1}^n d\left (\frac{s_{r'}^2}{2}\right )\wedge d\psi _{r'}\right )=-d\left (\frac{s_r^2}2\right )\ ,$$
this completes the proof.
\end{proof}
Since $d\omega=0$, we have  $d\tilde \omega=0$. Combining this information with  $i_{\frac{\partial}{\partial\psi_r}} \tilde \omega=0\ ,$
we conclude that
$$\tilde \omega =\pi ^*\beta \ ,$$
where $\beta $ is a closed $2$-form  on $\Omega _n\times \prod _{r=1}^n \mathcal B_{d_r}^\sharp $ , and
$$\pi (u):=((s_r(u))_{1\le r\le n},(\chi _r(u)_{1\le r\le n})\ .$$
In order to prove that $\tilde \omega =0$, it is therefore sufficient to prove that $\tilde \omega =0$ on the submanifold $$\mathcal V_{(d_1,\dots ,d_n),{\rm red}}:=\Phi ^{-1}\left (\Omega _n\times \prod _{r=1}^n \mathcal B_{d_r}^\sharp \right )\  $$
given by the equations $\psi _r=0\ ,\ r=1,\dots ,n$.
\begin{lemma}\label{restriction omega}
The restriction of $\omega $ to $\mathcal V_{(d_1,\dots,d_n),{\rm red}}$ is $0$.
\end{lemma}
\begin{proof} Consider the differential form $\alpha $ of of degree $1$ defined 
$$\langle \alpha (u), h\rangle :={\rm Im}(u\vert h)\ .$$
It is elementary to check that
$$\frac 12 d\alpha =\omega \ ,$$
hence  the statement is consequence of the fact  that the restriction of $\alpha $ to $\mathcal V_{(d_1,\dots,d_n),{\rm red}}$ is $0$.
Let us prove this stronger fact. By a density argument, we may assume that $n=2q$ is even, and that the Blaschke products $\chi _r(u)$ have only simple zeroes. Firstly, we describe the tangent space of $\mathcal V_{(d_1,\dots,d_n),{\rm red}}$ at a generic point. We use the notation of section \ref{section explicit}.
\begin{lemma}\label{espace tangent}
The tangent vectors to $\mathcal V_{(d_1,\dots,d_n),{\rm red}}$ at a generic point $u$ where every $\chi _r$  has
only simple zeroes are linear combinations with real coefficients of $u_j, u_jH_u(u_\ell ), 1\le j,\ell \le q $, and of the following functions,
for $ \zeta \in \C$ and $1\le j,k\le q$,
\begin{eqnarray*}
\dot u_{\chi _{2j-1}, \zeta }(z)&:=& \left (\overline \zeta \frac{z}{1-\overline pz}-\zeta \frac{1}{z-p}\right )u_j(z)H_u(u_j)(z)\ ,\ \chi _{2j-1}(p)=0\ ,\\
\dot u_{\chi _{2k},\zeta}(z)&:=&  \left (\overline \zeta \frac{z}{1-\overline pz}-\zeta \frac{1}{z-p}\right )zu'_k(z)K_u(u_k')(z)\ ,\ \chi _{2k}(p)=0\ .
\end{eqnarray*}
\end{lemma}
We assume  Lemma \ref{espace tangent}  and show how it implies Lemma \ref{restriction omega}.
Notice that 
\beno
(u\vert u_j)&=&\Vert u_j\Vert ^2\ ,\  (u\vert u_jH_u(u_j))=\Vert H_u(u_j)\Vert ^2\ ,\\
 (u\vert H_u(u_\ell )u_j)&=&(H_u(u_j)\vert H_u(u_\ell ))=0\ ,\ j\ne \ell \ ,
 \eeno
and therefore 
$\alpha (u)$ cancels on $u_j, u_jH_u(u_\ell ), 1\le j,\ell \le q $.  We now deal with  vectors $\dot u_{\chi _r,\zeta }$. 
\begin{eqnarray*}
(u\vert \dot u_{\chi _{2j-1},\zeta })&=& \zeta \left (u\Big \vert \frac{z}{1-\overline pz}u_jH_u(u_j)\right )-\overline \zeta \left (u\Big \vert \frac{u_j}{z-p}H_u(u_j)\right )\\&=& \zeta \left (H_u(u_j)\Big \vert \frac{z}{1-\overline pz}H_u(u_j)\right )-\overline \zeta \left (H_u^2(u_j)\Big \vert \frac{u_j}{z-p}\right )\ ,
\end{eqnarray*}
where we used that $u_j/(z-p)$ belongs to $L^2_+$. Since $H_u^2(u_j)=\rho _j^2u_j$ and $\rho _j^2\vert u_j\vert ^2=\vert H_u(u_j)\vert ^2$ on the 
unit circle, we infer
$$(u\vert \dot u_{\chi _{2j-1},\zeta })=\zeta \left (H_u(u_j)\Big \vert \frac{z}{1-\overline pz}H_u(u_j)\right )-\overline \zeta \left (\frac{z}{1-\overline pz}H_u(u_j)\Big \vert H_u(u_j)\right )$$
and consequently
$$\langle \alpha (u),\dot u_{\chi _{2j-1},\zeta }\rangle =2{\rm Im}\ \zeta\left (H_u(u_j)\Big \vert \frac{z}{1-\overline pz}H_u(u_j)\right )\ ,$$
which is $0$ for every $\zeta \in \C $ if and only if
$$\left (H_u(u_j)\Big \vert \frac{z}{1-\overline pz}H_u(u_j)\right )=0\ .$$
Let us prove this identity. Set
$$v:=\frac{z}{1-\overline pz}H_u(u_j)\ .$$
Notice that, since $\chi _{2j-1}(p)=0$,  $v\in E_u(\rh _j)$, and moreover
$$(v\vert 1)=v(0)=0\ .$$
Therefore
$$(H_u(u_j)\vert v)=(H_u(v)\vert u_j)=(H_u(v)\vert u)=(1\vert H_u^2(v))=\rho _j^2(1\vert v)=0\ .$$
We conclude that
$$\langle \alpha (u),\dot u_{\chi _{2j-1},\zeta }\rangle=0\ .$$
Similarly, we calculate
\begin{eqnarray*}
(u\vert \dot u_{\chi _{2k},\zeta })&=& \zeta \left (u\Big \vert \frac{z}{1-\overline pz}zu_k'K_u(u_k')\right )-\overline \zeta \left (u\Big \vert \frac{K_u(u_k')}{z-p}zu_k'\right )\\&=& \zeta \left (K_u(u_k')\Big \vert \frac{z}{1-\overline pz}K_u(u_k')\right )-\overline \zeta \left (K_u(u_k')\Big \vert \frac{K_u(u_k')}{z-p}\right )\ ,
\end{eqnarray*}
where we have used that $K_u(u_k')/(z-p)$ belongs to $L^2_+$. We conclude that
$$\langle \alpha (u),\dot u_{\chi _{2k},\zeta }\rangle =2{\rm Im}\ \zeta\left (K_u(u_k')\Big \vert \frac{z}{1-\overline pz}K_u(u_k')\right )\ ,$$
which is $0$ for every $\zeta \in \C $ if and only if
$$\left (K_u(u_k')\Big \vert \frac{z}{1-\overline pz}K_u(u_k')\right )=0\ .$$
Since $\vert K_u(u_k')\vert ^2=\sigma _k^2\vert u_k'\vert ^2$ on the unit circle, we are left to prove
$$\left (u_k'\Big \vert \frac{z}{1-\overline pz}u_k'\right )=0\ .$$
Set
$$w:=\frac{1}{1-\overline pz}u_k'\ .$$
We notice that $w\in F_u(\sigma _k)$, and that $zw \in F_u(\sigma _k)$. Moreover,
$$w=\frac{1}{1-\vert p\vert ^2}(\overline p\chi _p+1)u_k'\ ,$$
therefore, setting $\chi _{2k}:=g_k\chi _p$,
$$K_u(w)=\frac{\sigma _k}{1-\vert p\vert ^2}(pg_ku'_k+\chi _{2k}u'_k)=\frac{\sigma _kg_k}{1-\vert p\vert ^2}(p+\chi _p)u'_k=\sigma _kg_kzw\ .$$
In particular,
$$(K_u(w)\vert 1)=K_u(w)(0)=0\ .$$
We now conclude as follows,
$$\left (u_k'\Big \vert \frac{z}{1-\overline pz}u_k'\right )=(u_k'\vert zw)=(u\vert zw)=(K_u(w)\vert 1)=0\ .$$
This completes the proof up to the proof of lemma \ref{espace tangent}.
\end{proof}
Let us prove lemma \ref{espace tangent}. We are going to use formulae (\ref{luminy}), namely
$$u(z)=\langle \mathscr C(z)^{-1}(\chi _{2j-1}(z))_{1\le j\le q}, {\bf 1}\rangle ,$$
where  
$$\mathscr C(z)=\left (\frac{\rho _j -\sigma _k z\chi _{2k}(z)\chi _{2j -1}(z)}{\rho _j ^2-\sigma _k^2}\right )_{1\le j,k\le q}\ .$$
If we denote by $\dot {\  }$  the derivative with respect to one of the parameters $\rho _j, \sigma _k$ or one of the coefficients of the $\chi _r$,
we have 
$$\dot u(z)=- \langle \mathscr C(z)^{-1}\dot {\mathscr C}(z)\mathscr C(z)^{-1}(\chi _{2j-1}(z))_{1\le j\le q}, {\bf 1}\rangle+ \langle \mathscr C(z)^{-1}(\dot \chi _{2j-1}(z))_{1\le j\le q}, {\bf 1}\rangle\ .$$
In the case of the derivative with respect to $\rho _j$, 
$$
\dot {\mathscr C}(z)_{ik}=-\frac{\rh _j^2+\si _k^2-2\si _k\rh _j z\chi _{2j-1}(z)\chi_{2k}(z)}{(\rh _j^2-\si _k^2)^2}\delta _{ij}\ ,$$
and, using 
$$\mathscr C(z)^{-1}(\chi _{2j-1}(z))_{1\le j\le q}=(u'_k(z))_{1\le k\le q}\ ,\ ^t{\mathscr C}(z)^{-1}({\bf 1})=(h_j(z))_{1\le j\le q}\ ,$$
one gets 
\beno
\dot u(z)&=& h_j(z)\sum _{k=1}^qu'_k(z)\frac{ (\rh _j^2+\si _k^2-2\si _k\rh _j z\chi _{2j-1}(z)\chi_{2k}(z))}
{ (\rh _j^2-\si _k^2)^2} \\
&=& \frac {H_u(u_j)(z)}{\rho _j}\sum _{k=1}^q \frac{\rh _j^2+\si _k^2}{(\rh _j^2-\si _k^2)^2}u'_k(z)-2\rh _ju_j(z)\sum _{k=1}^q\frac{zK_u(u'_k)(z)}
{(\rh _j^2-\si _k^2)^2}\ .
\eeno 
Observe that 
$$zK_u(u'_k)(z)=[SS^*H_u(u'_k)](z)=H_u(u'_k)(z)-\kappa _k^2\ ,$$
and that $u'_k$ is a linear combination with real coefficients of $u_\ell $ in view of (\ref{urho}). We infer that, in this case, $\dot u$ is a linear combination with real coefficients of $u_j$ and $u_mH_u(u_\ell )$. 

In the case of the derivative with respect to $\si _k$, one  similarly gets
\beno
\dot u(z)&=& u'_k(z)\sum _{j=1}^qh_j(z)\frac{ z\chi _{2j-1}(z)\chi_{2k}(z)(\rh _j^2+\si _k^2)-2\si _k\rh _j }
{ (\rh _j^2-\si _k^2)^2} \\
&=& \frac{zK_u(u'_k)(z)}{\si _k}\sum _{j=1}^q \frac{(\rh _j^2+\si _k^2)u_j(z)}{(\rh _j^2-\si _k^2)^2}-2\si _ku'_k(z)\sum _{j=1}^q\frac {H_u(u_j)(z)}
{(\rh _j^2-\si _k^2)^2}\ ,
\eeno 
which is a linear combination with real coefficients of $u_j$ and $u_jH_u(u_\ell )$.

In the case of a derivative with respect to one of the zeroes of $\chi _{2j-1}$, we obtain 
\beno
\dot u(z)&=&\sum _{k=1}^q u'_k(z)h_j(z)\frac{\sigma _kz\dot \chi _{2j-1}(z)\chi _{2k}(z)}{\rh _j^2-\si _k^2}+h_j(z)\dot \chi _{2j-1}(z)\\
&=&\frac{\dot \chi _{2j-1}(z)}{\chi _{2j-1}(z)}u_j(z)\left [ \sum _{k=1}^qu'_k(z)\frac{\si _k z\chi _{2k}(z)}{\rh _j^2-\si _k^2} +1\right ]\\
&=&\frac{\dot \chi _{2j-1}(z)}{\chi _{2j-1}(z)}u_j(z)\left (z\frac{K_u(u_j)}{\tau _j^2}+1\right )\\
&=& \frac{\dot \chi _{2j-1}(z)}{\chi _{2j-1}(z)}u_j(z)\frac{H_u(u_j)}{\tau_j^2}\ .
\eeno
In the case of a derivative with respect to one of the zeroes of $\chi _{2k}$, we obtain similarly
$$
\dot u(z)= \frac{\dot \chi _{2k}(z)}{\chi _{2k}(z)}\frac{u'_k(z)zK_u(u'_k)(z)}{\kappa _k^2}\ .
$$
The proof of Lemma \ref{espace tangent} is completed by observing that, since $\chi _r$ is a product of functions $\chi _p$ for $\vert p\vert <1$,
$\dot \chi _r/\chi _r$ is a sum of terms of the form
$$ \left (\overline \zeta \frac{z}{1-\overline pz}-\zeta \frac{1}{z-p}\right )$$
where $\zeta :=\dot p$.

\section[Invariant tori and unitary equivalence]{Invariant tori of the Szeg\H{o} hierarchy and unitary equivalence of  pairs of Hankel operators} \label{invtori}\index{Hankel operator}

In this section, we identify the sets of symbols $u\in VMO_+\setminus \{ 0\} $ having the same list of singular values $(s_r)$\index{singular value} and the same list $(\chi _r)$ of monic Blaschke products, for the pair $(H_u,K_u)$. In view of Theorems \ref{mainfiniterank} and \ref{HilbertSchmidt}, these sets are tori. In the finite dimensional case, they are precisely the Lagrangian tori of the symplectic manifolds $W(\chi _1,\dots,\chi _n)$ of Corollary \ref{Vinvolutive}, obtained by freezing the action variables. Moreover, $VMO_+\setminus \{ 0\} $ is the disjoint union of these tori, and, from sections \ref{szegoflow} and \ref{szegohier}, the Hamilton flows of the Szeg\H{o} hierarchy act on them. We prove that these tori  are classes of some specific unitary equivalence between the pairs $(H_u,K_u)$, which we now define.
\begin{definition}\label{equivalence}
Given $u,\tilde u\in VMO_+\setminus \{ 0\} $, we set $u\sim \tilde u$ if there exist unitary operators $U, V$ on $L^2_+$ such that
$$H_{\tilde u}=UH_uU^*\ ,\ K_{\tilde u}=VK_uV^*\ ,$$
and there exists  a Borel function $F:\R _+\rightarrow \S^1$ such that
$$U(u)=F(H_{\tilde u}^2)\tilde u\ ,\ V(u)=F(K_{\tilde u}^2)\tilde u\ ,\ U^*V=\overline F(H_u^2)F(K_u^2)\ .$$
\end{definition}
It is easy to check that the above definition gives rise to an equivalence relation.
\begin{theorem}\label{specialisospectral}
Given $u,\tilde u\in VMO_+\setminus \{ 0\} $, the following assertions are equivalent.
\begin{enumerate}
\item $u\sim \tilde u$.
\item $\forall r\ge 1, s_r(u)=s_r(\tilde u) \ \textrm{and}\  \exists \gamma _r\in \T : \Psi_r(\tilde u)=\expo_{i\gamma _r}\Psi _r(u)\ .$
\end{enumerate}
\end{theorem}
\begin{proof}
Assume that (1) holds. Then $H_{\tilde u}^2$ is unitarily equivalent to $H_u^2$, and $K_{\tilde u}^2$ is unitarily equivalent to $K_u^2$.
This clearly implies $\Sigma _H(\tilde u)=\Sigma _H(u)$ and $\Sigma _K(\tilde u)=\Sigma _K(u)$, so that $s_r(\tilde u)=s_r(u)$ for every $r$.
Let us show that, for every $r$, $\Psi _r(u)$ and $\Psi _r(\tilde u)$ only differ by a phase factor. Of course the only  cases to be addressed are $d_r\ge 1$. We start with $r=2j-1$. From the hypothesis, we infer
$$U(u_j)=U({\bf 1}_{\{ \rho _j^2\} }(H_u^2)(u))={\bf 1}_{\{ \rho _j^2\} }(H_{\tilde u}^2)(U(u))=F(\rho _j^2)\tilde u_j\ ,$$
and, consequently,
\begin{equation}\label{UH_u}
U(H_u(u_j))=H_{\tilde u}(U(u_j))=\overline F(\rho _j^2)H_{\tilde u}(\tilde u_j)\ .
\end{equation}
Next we take advantage of the identity
$$U^*V=\overline F(H_u^2)F(K_u^2)\ ,$$
by evaluating $U^*S^*U$ on the closed range of $H_u$. We compute
\beno
U^*S^*UH_u&=&U^*S^*H_{\tilde u}U=U^*K_{\tilde u}U=U^*VK_uV^*U\\
&=&\overline F(H_u^2)F(K_u^2)K_u\overline F(K_u^2)F(H_u^2)=\overline F(H_u^2)F(K_u^2)^2K_uF(H_u^2)\\
&=&\overline F(H_u^2)F(K_u^2)^2S^*\overline F(H_u^2)H_u\ ,
\eeno
and we conclude, on $\overline {{\rm Ran}(H_u)}$, 
\begin{equation}\label{U*S*U}
U^*S^*U=\overline F(H_u^2)F(K_u^2)^2S^*\overline F(H_u^2)\ .
\end{equation}
For simplicity, set $D:=D_{2j-1}$ and $d:=d_{2j-1}$. Recall from Proposition \ref{action} that a basis of $E_u(\rho _j)$ is 
$$\left (\frac{z^a}{D}H_u(u_j)\ ,\ a=0,\dots ,d\right ) ,$$
and a basis of $F_u(\rho _j)=E_u(\rho _j)\cap u^\perp $ is 
$$\left (\frac{z^b}{D}H_u(u_j)\ ,\ b=0,\dots ,d-1\right ) .$$
For $a=1,\dots ,d_{2j-1}$, we infer
$$U^*S^*U\left (\frac{z^a}{D}H_u(u_j)\right )=\frac{z^{a-1}}{D}H_u(u_j)\ ,$$
or
$$U\left (\frac{z^a}{D}H_u(u_j)\right )=(S^*)^{d-a}U\left (\frac{z^d}{D}H_u(u_j)\right )\ ,\ a=0,\dots ,d\ .$$
This implies, for $a=0,\dots ,d-1$, that the right hand side belongs to $F_{\tilde u}(\rho _j)$. On the other hand, if $P\in \C [z]$ has degree at most $d$,
one easily checks that
$$S^*\left (\frac{P}{\tilde D}H_{\tilde u}(\tilde u_j)\right )=P(0)K_{\tilde u}(\tilde u_j)+ R\ ,\ R\in F_{\tilde u}(\rho _j)\ .$$
Notice that the right hand side belongs to $F_{\tilde u}(\rho _j)$ if and only if $K_{\tilde u}(\tilde u_j)\in F_{\tilde u}(\rho _j)$ or $P(0)=0$. Assume for a while that $K_{\tilde u}(\tilde u_j)$ does not belong to $F_{\tilde u}(\rho _j)$. Then, writing
$$U\left (\frac{z^d}{D}H_u(u_j)\right )=\frac{P}{\tilde D}H_{\tilde u}(\tilde u_j)\ ,$$
and using that, for $a=0,\dots ,d-1$, 
$$(S^*)^{d-a}\left (\frac{P}{\tilde D}H_{\tilde u}(\tilde u_j)\right )\in F_{\tilde u}(\rho _j)\ ,$$
we infer $P(0)=0$, and, by iterating this argument,  that $P$ is divisible by $z^d$, in other words,
$$U\left (\frac{z^d}{D}H_u(u_j)\right )=c\frac{z^d}{\tilde D}H_{\tilde u}(\tilde u_j)\ ,$$
for some $c\in \C $, and conclude
$$U\left (\frac{z^a}{D}H_u(u_j)\right )=c\frac{z^a}{\tilde D}H_{\tilde u}(\tilde u_j)\ ,\ a=0,\dots ,d\ .$$
Comparing to formula (\ref{UH_u}) for $U(H_u(u_j))$, we obtain
$$ cD(z)=\overline F(\rho _j^2)\tilde D(z)\ .$$
Since $D(0)=1=\tilde D(0)$, we conclude $c=\overline F(\rho _j^2)$, $D=\tilde D$, and finally $$\Psi _{2j-1}(\tilde u)=\overline F(\rho _j^2)^2\Psi _{2j-1}(u)\ .$$
We now turn to study the special case $K_{\tilde u}(\tilde u_j)\in F_{\tilde u}(\rho _j)$. This reads
$$0=(K_{\tilde u}^2-\rho _j^2I)K_{\tilde u}(\tilde u_j)=K_{\tilde u}((H_{\tilde u}^2-\rho _j^2I)\tilde u_j-\Vert \tilde u_j\Vert ^2\tilde u)=-\Vert \tilde u_j\Vert ^2K_{\tilde u}(\tilde u)\ .$$
In other words, this imposes $K_{\tilde u}(\tilde u)=0$, or $\tilde u=\rho \tilde \Psi $, where $\tilde \Psi $ is a Blaschke product of degree $d$. Making $V^*$ act on the identity $K_{\tilde u}(\tilde u)=0$, we similarly conclude $ u=\rho  \Psi $, where $\Psi $ is a Blaschke product of degree $d$, so what we have to check is that $\Psi $ and $\tilde \Psi $ only differ by a phase factor. In this case, $S^*$ sends $E_u(\rho )={\rm Ran}(H_u)$ into $F_u(\rho )$, so that (\ref{U*S*U}) becomes, on ${\rm Ran}(H_u)$,
$$U^*S^*U=S^*\ .$$
In other words, the actions of $S^*$ on $W:={\rm span}\left (\frac{z^a}{D}, a=0,\dots ,d\right )$ and on $\tilde W:={\rm span}\left (\frac{z^a}{\tilde D}, a=0,\dots ,d\right )$ are conjugated. Writing
$$D(z)=\prod _{p\in \mathcal P}(1-\overline pz)^{m_p}\ ,$$
where $\mathcal P$ is a finite subset of $\D \setminus \{ 0\}$, and $m_p$ are positive integers, and using the elementary identities 
$$(S^* -\overline pI)\left(\frac 1{(1-\overline pz)^k}\right )=S^* \left (\frac 1{(1-\overline pz)^{k-1}}\right )\ ,$$
one easily checks that the eigenvalues of $S^*$ on $W$ are precisely the $\overline p$'s, for $p\in \mathcal P$, and $0$,  with the corresponding algebraic multiplicities $m_p$ and 
$$m_0=1+d-\sum _{p\in \mathcal P}m_p\ .$$
We conclude that $D=\tilde D$, whence the claim.
\s
Next, we study the case $r=2k$. Then
$$V(u'_k)=F(\sigma _k^2)\tilde u'_k\ ,\ V(K_u(u'_k))=\overline F(\sigma _k^2)K_{\tilde u}(\tilde u'_k)\ .$$
Denote by $P_u$ the orthogonal projector onto $\overline {{\rm Ran}(H_u)}$, and compute
\beno
V^*P_{\tilde u}SVK_u&=&V^*P_{\tilde u}SK_{\tilde u}V=V^*(H_{\tilde u}-(\tilde u\vert\, .\, )P_{\tilde u}(1)) V\\
&=&V^*U(H_u-(U^*(\tilde u)\vert \, .\, )U^*(P_{\tilde u}(1)))U^*V\\
&=&\overline F(K_u^2)F(H_u^2)(H_u-(\overline F(H_u^2)u\vert \, .\, )F(H_u^2)P_u(1))\overline F(H_u^2)F(K_u^2)\\
&=&\overline F(K_u^2)F(H_u^2)^2(H_u-(u\vert \, .\, )P_u(1))F(K_u^2)\\
&=& \overline F(K_u^2)F(H_u^2)^2P_uSK_uF(K_u^2)= \overline F(K_u^2)F(H_u^2)^2P_uS\overline F(K_u^2)K_u\ ,
\eeno
so that, on $\overline {{\rm Ran}(K_u)}$, 
\begin{equation}\label{V*PSV}
V^*P_{\tilde u}SV=\overline F(K_u^2)F(H_u^2)^2P_uS\overline F(K_u^2)\ .
\end{equation}
For simplicity again, set $D:=D_{2k}$ and $d:=d_{2k}$. Recall from proposition \ref{action} that a basis of $F_u(\sigma _k)$ is 
$$\left (\frac{z^a}{D}u'_k\ ,\ a=0,\dots ,d\right ) ,$$
and a basis of $E_u(\sigma _k)=F_u(\sigma _k)\cap u^\perp $ is 
$$\left (\frac{z^a}{D}u'_k\ ,\ a=1,\dots ,d\right ) \ .$$
Applying identity (\ref{V*PSV}) to $\frac{z^a}{D}u'_k$ for $a=0,\dots ,d-1$, we infer
$$V\left (\frac{z^a}{D}u'_k\right )=(P_{\tilde u}S)^a V\left (\frac 1{ D} u'_k\right )\ , \ a=0,\dots ,d\ .$$
In particular, the right hand side belongs to $E_{\tilde u}(\sigma _k)$ for $a=1,\dots ,d$. On the other hand, if $Q\in \C [z]$ has degree at most $d$, 
$$P_{\tilde u}S\left (\frac{Q}{\tilde D}\tilde u'_k\right )= \gamma P_{\tilde u}SK_{\tilde u}(\tilde u'_k) +R\ ,\ R\in E_{\tilde u}(\sigma _k)\ ,$$
where $\gamma \sigma _k \expo_{-i\tilde \psi _{2k}}$ is the coefficient of $z^d$ in $Q$.  Therefore the left hand side belongs to $E_{\tilde u}(\sigma _k)$ if and only if
\beno
0&=&\gamma (H_{\tilde u}^2-\sigma _k^2I)P_{\tilde u}SK_{\tilde u}(\tilde u'_k)\\
&=&\gamma (H_{\tilde u}^2-\sigma _k^2I)(H_{\tilde u}(\tilde u'_k)-\Vert \tilde u'_k\Vert ^2P_{\tilde u}(1))\\
&=&\gamma \sigma _k^2\Vert \tilde u'_k\Vert ^2P_{\tilde u}(1)\ ,
\eeno
which is impossible. We conclude that $\gamma =0$, which means that the degree of $Q$ is at most $d-1$. Iterating this argument, we infer 
$$V\left (\frac 1{ D} u'_k\right )=c\frac{1}{\tilde D}\tilde u'_k\ ,$$
for some $c\in \C $, and finally
$$V\left (\frac{z^a}{D}u'_k\right )=c\frac{z^a}{\tilde D}\tilde u'_k\ ,\ a=0,\dots d\ .$$
Comparing to the above formula for $V(u'_k)$, we obtain 
$$cD=F(\sigma _k^2)\tilde D\  ,$$
thus, since $D(0)=1=\tilde D(0)$, we have $D=\tilde D$, $c=F(\sigma _k^2)$, and finally $\Psi _{2k}(\tilde u)=F(\sigma _k^2)^2\Psi _{2k}(u)\ .$
\s
Assume that (2) holds. Define $F:\Sigma _H(u)\cup \Sigma _K(u)\rightarrow \S^1$ by
$$F(\rho _j^2)=\expo_{-i\frac{\gamma _{2j-1}}2}\ ;\ F(\sigma _k^2)=\expo_{i\frac{\gamma _{2k}}2}\ ,$$
and, if necessary,  we define $F(0)$ to be any complex number of modulus $1$. Next we define $U$ on the closed range of $H_u$, which is the closed orthogonal sum of $E_u(s_r)$. Thus we just have to define $$U:E_u(s_r)\rightarrow E_{\tilde u}(s_r).$$ 
\s
If $r=2j-1$, we set
\begin{equation}\label{Uj}
U\left (\frac{z^a}{D_{2j-1}}H_u(u_j)\right )=\overline F(\rho _j^2) \frac{z^a}{D_{2j-1}}H_{\tilde u}(\tilde u_j)\ ,\ a=0,\dots ,d_{2j-1}\ .
\end{equation}
If $r=2k$ and $d_{2k}\ge 1$, we set
\begin{equation}\label{Uk}
U\left (\frac{z^b}{D_{2k}}u'_k\right )=F(\sigma _k^2) \frac{z^b}{D_{2k}}\tilde u'_k\ ,\ b=1,\dots ,d_{2k}\ .
\end{equation}
Using (\ref{Uj}) we obtain
\beno
U(u_j)&=&\frac 1{\rho _j}U(\Psi _{2j-1}(u)H_u(u_j))=\frac 1{\rho _j}\overline F(\rho _j^2)\Psi _{2j-1}(u)H_{\tilde u}(\tilde u_j)\\
&=&\frac 1{\rho _j}F(\rho _j^2)\Psi _{2j-1}(\tilde u)H_{\tilde u}(\tilde u_j)=F(\rho _j^2)\tilde u_j\ . 
\eeno
Consequently, we get
$$U(u)=\sum _jU(u_j)=\sum _jF(\rho _j^2)\tilde u_j=F(H_{\tilde u}^2)\tilde u\ .$$
A similar argument combined to Proposition \ref{action} leads to 
$$UH_u=H_{\tilde u}U.$$
Next, we prove that $U$ is unitary. It is enough to prove that every map $U:E_u(s_r)\rightarrow E_{\tilde u}(s_r)$ is unitary, or that the Gram matrix of a basis of $E_u(s_r)$ is equal to the Gram matrix of its image. We first deal with $r=2j-1$. Equivalently, we prove that, for $a,b=0,\dots ,d_{2j-1}-1$,
$$\left (\frac{z^a}{D_{2j-1}}H_u(u_j)\vert \frac{z^b}{D_{2j-1}}H_u(u_j)\right )=\left (\frac{z^a}{D_{2j-1}}H_{\tilde u}({\tilde u}_j)\vert \frac{z^b}{D_{2j-1}}H_{\tilde u}({\tilde u}_j)\right )\ .$$
We set 
$$\zeta _{a-b}:=\left (\frac{z^a}{D_{2j-1}}H_u(u_j)\vert \frac{z^b}{D_{2j-1}}H_u(u_j)\right ), \ a,b=0,\dots ,d_{2j-1}-1\ ,$$
and we notice that $\zeta _{-k}=\overline \zeta _k\ ,\ k=-d_{2j-1},\dots ,d_{2j-1}\ .$ We drop the subscript $2j-1$ for simplicity and we set 
$$D(z):=1+\overline a_1z+\dots +\overline a_d z^d.$$
As $\Psi H_u(u_j)$ is orthogonal to $\frac {z ^a}{D} H_u(u_j)$ for $a=0,\dots, d-1$, and $\Vert H_u(u_j)\Vert^2 =\rho_j^2\tau_j^2$, we obtain the system
\begin{equation}\label{system}
\left\{\begin{array}{ll}
\zeta_{d-b}+a_1\zeta_{d-b-1}+\dots+a_d\zeta_{-b}=0\ , \ b=0,\dots, d-1\ ,\\
\zeta_0+a_1\zeta_{-1}+\dots +a_d\zeta_{-d}=\rho_j^2\tau_j^2\ .\end{array}\right.
\end{equation}
\begin{lemma}\label{Solsystem}
Let $a_1,\dots a_d$ be complex numbers such that the polynomial
$ z^d+a_1z^{d-1}+\dots+a_d $ has all its roots in $\D$. Then the system (\ref{system}) has at most one solution $\zeta_k$, $k=-d\dots, d$ with $\overline \zeta_k=\zeta_{-k}$.
\end{lemma}
Assume for a while that this lemma is proved. Since $\tau_j^2$  can be expressed in terms of the $(s_r)$'s --- see (\ref{tau}), we infer that 
$U: E_u(\rho_j)\rightarrow E_{\tilde u}(\rho_j)$ is unitary. Similarly, one proves that the Gram matrix of the basis 
$$\frac {z^a}{D_{2k}}u'_k, \ a=0,\dots,d_{2k}$$ of $F_u(\sigma_k)$ only depends on the $(s_r)$'s and on $D_{2k}$. In particular, $$U:E_u(\sigma_k)\rightarrow E_{\tilde u}(\sigma_k)$$ is unitary and finally is unitary from the closed range of $H_u$ onto the closed range of $H_{\tilde u}$.
\s
Next, we construct $V$ on the closed range of $H_u$ which is the orthogonal sum of the $F_u(\sigma)$ for $\sigma \in \Sigma_H\cup \Sigma_K$.
Thus we just have to define $V:F_u(\sigma)\rightarrow E_{\tilde u}(\sigma)$ for $\sigma  \in \Sigma_H\cup \Sigma_K$.
\s
If $r=2j-1$, we set
\begin{equation}\label{Vj}
V\left (\frac{z^a}{D_{2j-1}}H_u(u_j)\right )=\overline F(\rho _j^2) \frac{z^a}{D_{2j-1}}H_{\tilde u}(\tilde u_j)\ ,\ a=1,\dots ,d_{2j-1}\ .
\end{equation}
If $r=2k$ and $d_{2k}\ge 1$, we set
\begin{equation}\label{Vk}
V\left (\frac{z^b}{D_{2k}}u'_k\right )=F(\sigma _k^2) \frac{z^b}{D_{2k}}\tilde u'_k\ ,\ b=0,\dots ,d_{2k}\ .
\end{equation}
Similarly, if $0\in\Sigma_K$, we define $V(u'_0)=F(0)\tilde u'_0$.
Using (\ref{Vk}) we get $V(u'_k)=F(\sigma _k^2)\tilde u'_k.$ 
Consequently, 
$$V(u)=V(u'_0)+\sum _kV(u'_k)=F(K_{\tilde u}^2)\tilde u\ .$$
A similar argument combined with Proposition \ref{action} leads to 
$$VK_u=K_{\tilde u}V.$$
Using again Lemma \ref{Solsystem}, $V$ is unitary from the closed range of $H_u$ onto the closed range of $H_{\tilde u}$.
\s
Now we define $U$ and $V$ on the kernel of $H_u$ which is either $\{0\}$ or an infinite dimensional separable Hilbert space. From Corollary \ref{kernel}, the cancellation of $\ker H_u$ only depends on the $s_r$'s. Therefore, $\ker H_u$ and $\ker H_{\tilde u}$ are isometric. We then define $U=V$ from $\ker H_u$ onto $\ker H_{\tilde u}$ to be any unitary operator.
\s
It remains to prove that $U^*V= \overline F(H_u^2)F(K_u^2)$. On $\ker H_u$, it is trivial since $U^*V=I=\overline F(0) F(0)$.
Similarly, it is trivial on vectors 
$$\frac {z^a}{D_{2k}}u'_k \ ,a=1,\dots , d_{2k}\ .$$
It remains to prove the equality for $u'_0$, $u'_k$. We write 
\beno
U^*V(u'_k)&=&F(\sigma_k^2)U^*(\tilde u'_k)=F(\sigma_k^2)U^*\left(\kappa_k^2\sum_j\frac{\tilde u_j}{\rho_j^2-\sigma_k^2}\right)\\
&=&F(\sigma_k^2)\sum_j\overline F(\rho_j^2)\kappa_k^2 \frac{u_j}{\rho_j^2-\sigma_k^2}= F(\sigma_k^2)\overline F(H_u^2)u'_k\\
&=& \overline F(H_u^2)F(K_u^2)(u'_k)\ .
\eeno 
A similar arguments holds for $U^*V(u'_0)$.
\s 
It remains to prove Lemma \ref{Solsystem}. It is sufficient to prove that the only solution of the homogeneous system
\begin{equation}\label{systemHomogene}
\left\{\begin{array}{ll}
\zeta_{d-b}+a_1\zeta_{d-b-1}+\dots+a_d\zeta_{-b}=0\ , \ b=0,\dots, d-1\ ,\\
\zeta_0+a_1\zeta_{-1}+\dots +a_d\zeta_{-d}=0\ ,\end{array}\right.
\end{equation} with $\overline \zeta_k=\zeta_{-k} \ , k=0,\dots, d$, is the trivial solution $\zeta =0$.

 We proceed by induction on $d$. For $d=1$, the system reads
$$ \left\{\begin{array}{ll}
\zeta_{1}+a_1\zeta_{0}=0\ , \ ,\\
\zeta_0+a_1\overline\zeta_{1}=0 \ .\end{array}\right.$$
Since $|a_1|<1$, this trivially implies $\zeta_0=\zeta_1=0$.
\s
For a general $d$, we plug the expression
$$\zeta_d=-(a_1\zeta_{d-1}+\dots +a_d\zeta_{0})$$ into the last equation. We get
\begin{equation}\label{eq1}
\zeta_0+b_1\overline\zeta_{1}+\dots +b_{d-1}\overline\zeta_{d-1}=0
\end{equation} with 
$$b_k=\frac{a_k-a_d\overline a_{d-k}}{1-|a_d|^2}\ , k=1,\dots, d-1\ .$$
Notice that from Proposition \ref{Od}, $|a_d|<1$ and the polynomial
$z^{d-1}+b_1z^{d-2}+\dots+b_{d-1}$ has all its roots in $\D$.
For $b=1,\dots ,d-1$, we multiply by $a_d$ the conjugate of equation 
$$\zeta_{b}+a_1\zeta_{b-1}+\dots+a_d\zeta_{b-d}=0$$
and substract the result from equation
$$\zeta_{d-b}+a_1\zeta_{d-b-1}+\dots+a_d\zeta_{-b}=0\ .$$
This yields
$$\zeta_{d-b}+b_1\zeta_{d-b-1}+\dots+b_{d-1}\zeta_{1-b}=0\ .$$
Together with Equation (\ref{eq1}), this is exactly the system at order $d-1$ with coefficients $b_1,\dots, b_{d-1}$. By induction, we obtain 
$$\zeta_0=\zeta_1=\dots=\zeta_{d-1}=0$$ and finally $\zeta_d=0$.

This completes the proof.
 
\end{proof}

\backmatter

\end{document}